\documentclass{article}

\usepackage[utf8]{inputenc}
\usepackage[normalem]{ulem} 	% allows to break lines when in underline, using \uline rather than \underline
\usepackage{xparse}         	% -> \NewDocumentCommand{}{}{} 
\usepackage{comment}

\usepackage{mathtools,amssymb,amsthm,empheq}
\usepackage{bm} 	   	% allows bold greek symbols, using \bm rather than \boldsymbol
\usepackage{relsize}   	% -> "\mathlarger" that enlarges math symbols
\usepackage{mathrsfs}  	% -> \mathscr, a math font
\usepackage{cancel}   	% -> "\cancel" for "crossing out" diagonals
\usepackage{esint}    	% for the weird integral symbol
\usepackage{xcolor} 	% gives color to cancel
\usepackage{cases}
\usepackage{pgf,tikz}
\usetikzlibrary{cd,arrows,babel,shapes,positioning}
\usepackage{hyperref}
\hypersetup{
	colorlinks,
	citecolor=blue,
	filecolor=blue,
	linkcolor=blue,
	urlcolor=blue
}
\DeclareMathOperator{\Dist}{dist}
\DeclareMathOperator{\Div}{div}

\newcommand{\R}{\mathbb{R}}

\newcommand{\EPS}{\varepsilon}
\newcommand{\EMP}{\emptyset}

\newcommand{\Ann}{\mathfrak{A}}

\newcommand{\ColorWord}[2]{\color{#1} #2 \color{black} }
\newcommand{\Dt}{\mathcal{D}_t^{\frac{1}{2}}}
\usepackage{enumerate}
\numberwithin{equation}{section}

\theoremstyle{plain}

\newtheorem{thm}[equation]{Theorem}
\newcommand{\refthm}[1]{\emph{\ColorWord{blue}{Theorem} \ref{#1}}}

\newtheorem{lemma}[equation]{Lemma}
\newcommand{\reflemma}[1]{\emph{\ColorWord{blue}{Lemma} \ref{#1}}}

\newtheorem{prop}[equation]{Proposition}
\newcommand{\refprop}[1]{\emph{\ColorWord{blue}{Proposition} \ref{#1}}}

\newtheorem{cor}[equation]{Corollary}
\newcommand{\refcor}[1]{\emph{\ColorWord{blue}{Corollary} \ref{#1}}}

\theoremstyle{definition}

\newtheorem{defin}[equation]{Definition}
\newcommand{\refdef}[1]{\emph{\ColorWord{blue}{Definition} \ref{#1}}}

\theoremstyle{remark}

\newtheorem{rem}[equation]{Remark}
\newcommand{\refrem}[1]{\textit{Remark \ref{#1}}}

\newtheoremstyle{named}{}{}{\itshape}{}{\bfseries}{}{.5em}{#1 #3}
\theoremstyle{named}
\allowdisplaybreaks[3]

\title{Perturbation theory for the parabolic Regularity and Neumann problem}
\author{Martin Ulmer}
\date{\today}

\begin{document}

\maketitle

\begin{abstract}
\noindent
We show small and large Carleson perturbation results for the parabolic Regularity boundary value problem with boundary data in \(\dot{L}_{1,1/2}^p\) and small Carelson perturbation results for the Neumann problem with boundary data in \(L^p\). The operator we consider is \(L:=\partial_t -\mathrm{div}(A\nabla\cdot)\) and the domains are parabolic cylinders \(\Omega=\mathcal{O}\times\mathbb{R}\), where \(\mathcal{O}\) is a Lipschitz domain.

\end{abstract}

\tableofcontents 
\bigskip

\section{Introduction}

In this work we study a class of linear parabolic operators with time-varying coefficients. We prove that if such an operator solves the Regularity or the Neumann boundary value problem, then perturbations of this operator also solve the Regularity or Neumann boundary value problem respectively. We measure these perturbations in terms of a well-studied Carleson norm on the difference of the coefficients (see \eqref{eq:CarlesonWithSupNorm}). Our results address both small and large perturbations and go beyond the already established perturbation theory for the Dirichlet boundary value problem. 
\smallskip

Boundary value problems for elliptic and parabolic operators with rough boundary data have been studied extensively over the last few decades. The question of perturbation theory, we address in this work, goes back to Dahlberg in \cite{dahlberg_absolute_1986}. In this article, Dahlberg considers two symmetric elliptic operators \(L_0=\mathrm{div}(A_0\nabla\cdot)\) and \(L_1=\mathrm{div}(A_1\nabla\cdot)\) and supposes that the \(L^2\) Dirichlet problem is solvable for the "good" operator \(L_0\). He poses the question under what condition on the difference between \(A_0\) and \(A_1\) we can conclude that \(L_1\) is also a "good" operator in the sense that the \(L^2\) Dirichlet problem is also solvable for \(L_1\). He introduces a vanishing Carleson condition on the difference \(\EPS:=|A_1-A_0|\) and proves the posed perturbation question. However, this first positive result gave rise to a lot of new questions. In particular, is Dahlberg's vanishing Carleson condition the best perturbative condition on closeness of \(A_0\) and \(A_1\) that allows to transfer \(L^p\) solvability from one operator to the other one and what is the situation if we merely want solvability of the \(L^q\) Dirichlet problem for \(L_1\), where \(p\leq q<\infty\)? These questions were addressed in \cite{fefferman_criterion_1989} and \cite{fefferman_theory_1991}, and the answer is also given in terms of a Carleson measure condition on the discrepancy function \(\EPS\). Interestingly, there are two types of results, which we generally refer to as large Carleson perturbation and small Carleson perturbation results. First, if the Carleson measure involving \(\EPS\) has large Carleson norm, then solvability of the Dirichlet problem with boundary data in \(L^p\) for \(L_0\) implies solvability of the Dirichlet problem with boundary data in \(L^q\) for \(L_1\), where \(q\) is some number potentially larger than \(p\). Second, if the Carleson measure norm is sufficiently small depending on the given \(p\), then \(L^p\) solvability for \(L_0\) yields \(L^p\) solvability for \(L_1\). Hence, in this case, \(L^p\) solvability is preserved, as was in \cite{dahlberg_absolute_1986} for \(L^2\).
\smallskip

Since then, perturbation theory has experienced much attention and extensions to rougher domains, non symmetric operators, other boundary value problems, and the parabolic setting (refer to Section \ref{subsection:BVPs} for definitions of the boundary value problems). For the elliptic Dirichlet problem we have extensions of the small and the large perturbation results (cf. \cite{milakis_harmonic_2013}, \cite{hofmann_extrapolation_2012} \cite{cavero_perturbations_2019}, \cite{cavero_perturbations_2020}, \cite{akman_perturbation_2023}, \cite{dindos_perturbation_2023}, \cite{feneuil_generalized_2022}) to very rough domains and non symmetric operators. For the elliptic Regularity problem, \cite{kenig_neumann_1995} establishes small and large perturbation results and both could also be extended to rough settings and non symmetric operators (cf. \cite{dai_carleson_2022}). However, in contrast to that, for the elliptic Neumann problem there are currently only small perturbation results for symmetric operators and only on Lipschitz domains. This result was also established in \cite{kenig_neumann_1995}.
\smallskip

For the parabolic problems, we have fewer perturbation results. So far, only the Dirichlet boundary value problem has affirmative answers to the perturbation question. We have small and large perturbation results, also for non symmetric operators, on Lip(1,1/2) or Lipschitz cylinder domains (cf. \cite{nystrom_dirichlet_1997}, \cite{sweezy_bq_1998}, \cite{hofmann_dirichlet_2001}). We are going to discuss the parabolic Dirichlet perturbation results in more detail in \refprop{prop:ParabolicPerturbationTheoryForDP}. However, there are no perturbation results yet for the parabolic Regularity or Neumann boundary value problem. This is the gap that we address with our main results in this work.
\medskip

In the following we consider the parabolic drift-less PDE
\begin{align}L_iu:=\partial_tu -\Div(A_i(X,t)\nabla u)=0,\label{eq:StandardPDE}\end{align}
for \(i=0,1\), where \(A_i(X,t)\in \mathbb{R}^{n}\times\mathbb{R}^{n}\) are elliptic matrices, i.e. there exists \(\lambda>0\) such that
\begin{align*}\lambda |\xi|^2\leq \xi^T A_i(X,t) \xi\qquad\textrm{for all }\xi\in\mathbb{R}^{n}, \textrm{ and a.e. }(X,t)\in \Omega\subset \mathbb{R}^{n+1}.\end{align*}
Furthermore, we only assume that the coefficients of \(A_i\) are bounded and measurable functions.
Our domains will be parabolic cylinder domains \(\Omega:=\mathcal{O}\times\mathbb{R}\), where \(\mathcal{O}\subset\mathbb{R}^n\) is the spatial component and the last coordinate for \((X,t)\in\Omega\) is the \(t\) component or time component. The closeness condition of \(A_0\) and \(A_1\) is given as
\begin{align}d\mu:=\frac{\sup_{B(X,t,\delta(X,t)/2))}|A_0-A_1|^2}{\delta}dXdt\qquad\textrm{ is a Carleson measure,}\label{eq:CarlesonWithSupNorm}\end{align}
where \(B(X,t,\delta(X,t)/2))\) is a parabolic ball centered at \((X,t)\) with radius \(\delta(X,t)/2\). The main theorem is the following:

\begin{thm}\label{MainTheorem}
    Let \(L_0, L_1\) be two parabolic operators of form \eqref{ParabolicOperator} on the cylinder domain \(\Omega:=\mathcal{O}\times\mathbb{R}\subset \mathbb{R}^{n+1}\) with a Lipschitz domain \(\mathcal{O}\). Assume \eqref{eq:CarlesonWithSupNorm} and that the Regularity problem \((R)^{L_0}_{q_0}\) is solvable for \(q_0>1\). Then following results are true:
    \begin{itemize}
        \item[(S)] There exists a sufficiently small \(\EPS=\EPS(n,\lambda_0,\mathcal{O}, q_0)>0\), such that if \(\Vert \mu \Vert_{\mathcal{C}}\leq \EPS\), then the Regularity problem \((R)^{L_1}_{q_0}\) is solvable.
        \item[(L)] If \(\Vert \mu \Vert_{\mathcal{C}}<\infty\), then the Regularity problem \((R)^{L_1}_{q_1}\) is solvable for some \(1< q_1<q_0\).
    \end{itemize} 
\end{thm}

The first conclusion \((S)\) is a small Carleson perturbation result, while the second conclusion \((L)\) is a large Carleson perturbation result. Hence, we extend both types of perturbative results from the elliptic to the parabolic setting. Our proof generally follows the road map of the elliptic perturbation theory for the Regularity problem in \cite{kenig_neumann_1995}. But due to the parabolic setting we encounter new difficulties, since fundamental properties of solutions, like the Harnack inequality, are time dependent, and the PDE has scaling differences in space and time. For instance, moving poles of the Green's function in the parabolic setting requires much more care and the decay of a solution in time is different from its decay in space. Since the proof of \cite{kenig_neumann_1995} relies heavily on these tools, we can roughly follow the outline in \cite{kenig_neumann_1995}, but the core lemmas will need adaptations which require new proof ideas. We will point out these differences in Section \ref{subsection:UnboundedStructure}. 
\smallskip

For certain applications it suffices to study operators with a pointwise bound on the weak derivative of the coefficients of the form \(\delta |\nabla A_1|\leq C\) (cf. for instance \cite{hofmann_a_infty_2017}, \cite{feneuil_green_2023}). In this case, we can relax the perturbative condition. Specifically, we can drop the \(L^\infty\) norm on the ball \(B(x,t,\delta(X,t)/2)\).

\begin{cor}\label{cor:MainThmWithoutSupCond}
    Let \(L_0, L_1\) be two parabolic operators of form \eqref{ParabolicOperator} on the cylinder domain \(\Omega:=\mathcal{O}\times\mathbb{R}\subset \mathbb{R}^{n+1}\) with a Lipschitz domain \(\mathcal{O}\). Assume that
    \begin{align} \delta |\nabla A_1|\leq C<\infty\textrm{ and }
        d\mu:=\frac{|A_0-A_1|^2}{\delta}dXdt\textrm{ is a Carleson measure.} \label{eq:CarlesonWithoutSupNorm} \end{align}
    Further assume that the Regularity problem \((R)^{L_0}_{q_0}\) is solvable for \(q_0>1\). Then we have conclusion \((S)\) and \((L)\) from \refthm{MainTheorem}.
\end{cor}

The measure in condition \eqref{eq:CarlesonWithoutSupNorm} is always less or equal to that of \eqref{eq:CarlesonWithSupNorm} and hence might cover a wider class of perturbations, but it comes at the cost of having a pointwise bound on the gradient of \(A_1\). Because the proofs for both conditions differ in only one point (proof of \reflemma{lemma:defv}), we include this corollary here. 
\smallskip

Apart from that, we establish a parabolic version of the small perturbation result for the \(L^p\) Neumann boundary value problem on a parabolic cylinder.
\begin{thm}\label{thm:MaintheoremNeumann}
    Assume that \(L_0, L_1\) are two parabolic operators of form \eqref{ParabolicOperator} on the cylinder domain \(\Omega:=\mathcal{O}\times\mathbb{R}\) with a Lipschitz \(\mathcal{O}\). Assume \eqref{eq:CarlesonWithSupNorm}. If \((N)^{L_0}_{q_0}\) and \((D)^{L_1^*}_{q_0'}\) are solvable for some \(q_0>1\), then there exists a small \(\EPS=\EPS(n,\lambda_0,\mathcal{O}, q_0)>0\), such that if \(\Vert \mu \Vert_{\mathcal{C}}\leq \EPS\), then \((N)^{L_1}_{q_0}\) is solvable.
\end{thm}

Compared to the Regularity problem, the Neumann perturbation contains the additional assumption of solvability of the Dirichlet problem for the adjoint. In fact, this assumption is also present in the elliptic case (cf. \cite{kenig_neumann_1995}) and does not appear in the Regularity perturbation because \((R)_{L_0}^{q_0}\) already implies \((D)_{L_0^*}^{q_0'}\). However, the Neumann problem \((N)_{L_0}^{q_0}\) is not known to imply \((D)_{L_0^*}^{q_0'}\). 

\begin{comment}
\smallskip

The article is organised as follows: Section 2 introduces all the necessary preliminaries and definitions. Section 3.1 shows the proof of \refthm{MainTheorem} omitting the proofs of some auxilliary lemmas, which are then proved in the following subsections of Section 3. As it turns out, we need to deal with the case of bounded \(\mathcal{O}\) separately, and we will do so in Section 4. Lastly Section 5 shows the proof of the perturbation for the Neumann problem. 
\end{comment}

\subsection{Extensions to domains with rougher boundary}

All of our main results are formulated on cylinders \(\Omega=\mathcal{O}\times\mathbb{R}\) where \(\mathcal{O}\) is a Lipschitz domain (see \refdef{def:Lipschitzdomain}). However, some of the important ground work for us has been done on domains with rougher boundary (cf. \cite{dindos_regularity_2023}, \cite{dindos_satterqvist_24}, \cite{ulmer_l_2024}). 

With all these tools at our disposal on domains with rougher base \(\mathcal{O}\), it is natural to wonder whether the perturbation results, the main results of this work, could not be extended to these types of domains. We present a detailed discussion in Section \ref{subsection:geometryOfBoundary}. Here, we would like to mention that by imposing \(\omega_{L_1}^*\in B_{q_1}(\sigma)\) (or equivalently that \((D)_{L_1^*}^{q_0'}\) is solvable) as an additional condition, we obtain conclusion \((S)\) of \refthm{MainTheorem} also on domains with uniform base \(\mathcal{O}\) (see \refdef{def:unifromDomain}).

\begin{cor}[Alternative version of \((S)\) in \refthm{MainTheorem}]\label{cor:AltMainThm}
     Let \(L_0, L_1\) be two parabolic operators of form \eqref{ParabolicOperator} on the cylinder domain \(\Omega:=\mathcal{O}\times\mathbb{R}\) with uniform \(\mathcal{O}\). Assume \eqref{eq:CarlesonWithSupNorm}, that the Regularity problem \((R)^{L_0}_{q_0}\) is solvable for \(q_0>1\), and that \(\omega_{L_1}^*\in B_{q_1}(\sigma)\). Then there exists a sufficiently small \(\EPS=\EPS(n,\lambda_0,\mathcal{O}, q_0)>0\), such that if \(\Vert \mu \Vert_{\mathcal{C}}\leq \EPS\), then the Regularity problem \((R)^{L_1}_{q_0}\) is solvable.
\end{cor}

For an alternative of the large perturbation result \((L)\) of \refthm{MainTheorem} we have to impose more additional conditions and this is also part of the discussion in Section \ref{subsection:geometryOfBoundary}. There we will also elaborate that our perturbation results do not use that \(\mathcal{O}\) is Lipschitz intrinsically, and that a full extension to the case of uniform base \(\mathcal{O}\) at this point falls short due to certain missing results in the rougher setting.

\subsection{Application of perturbation theory to DKP operators}
Another question for the Dirichlet boundary value problem, that goes back to Dahlberg (\cite{dahlberg_estimates_1977} and \cite{dahlberg_poisson_1979}), is the following: For which matrices \(A\) is the \(L^p\) Dirichlet boundary value problem solvable?  The literature, that works towards answering this question, is extensive and we do not make an effort to mention all results here, but one of the most influential conditions is the DKP condition, or also Carleson condition. The DKP condition assumes essentially that \(|\nabla A|^2\delta dX\) is a Carleson measure. Kenig and Pipher were the first ones to prove solvability of the elliptic \(L^p\) Dirichlet problem for some \(p\) under this condition (cf. \cite{kenig_dirichlet_2001}). 
Interestingly, the DKP condition found application also for other elliptic boundary value problems, like the Regularity or Neumann problem (cf. the survey article \cite{dindos_boundary_2023} or \cite{dindos_lp_2007}, \cite{dindos_elliptic_2010}, \cite{dindos_boundary_2017}, \cite{dindos_etal_regularity_2023}, \cite{feneuil_alternative_2023}, \cite{mourgoglou_solvability_2025}, \cite{dindos_elliptic_2010}, \cite{dindos_boundary_2017}, \cite{dindos_etal_regularity_2023}).
\begin{comment}
Interestingly, there are also small Carleson or DKP norm results: For a given \(p\), if this Carleson norm and the Lipschitz constant of the domain are sufficiently small, then we obtain solvability of the \(L^p\) Dirichlet problem (cf. \cite{dindos_lp_2007}). Since then, the DKP condition has been extended to much rougher domains and other boundary value problems. The survey article \cite{dindos_boundary_2023} contains many references for elliptic boundary value problems. In particular, we would like to point out, that also the elliptic Regularity problem is solvable under the DKP condition (cf. \cite{dindos_elliptic_2010}, \cite{dindos_boundary_2017} for small DKP, and  \cite{dindos_etal_regularity_2023}, \cite{feneuil_alternative_2023},\cite{mourgoglou_solvability_2025} for large DKP). Furthermore, for the elliptic Neumann problem, we have the same small DKP results (cf. \cite{dindos_elliptic_2010}, \cite{dindos_boundary_2017}), but we only have a large DKP result, if \(n=2\) (cf. \cite{dindos_etal_regularity_2023}).
\end{comment}
\smallskip

The parabolic theory of boundary value problems largely trails behind the elliptic one. However, our parabolic perturbation results for the Regularity and for the Neumann problem are not only interesting in themselves, but have also been applied to prove new parabolic boundary value problems. Specifically, the perturbation results were necessary to extend the DKP condition to the regularity problem for parabolic time-varying operators (cf. \cite{dindos_lp_2025}). Previously, the DKP condition had only been shown to be sufficient for solvability of the parabolic Dirichlet boundary value problem (cf. \cite{hofmann_dirichlet_2001},\cite{dindos_dirichlet_2018},\cite{dindos_parabolic_2020}).

\subsection{Plan of the article}

A brief summary of this article is as follows. In Section \ref{section:Prelims}, we introduce all notations and basic definitions. We also define the solution spaces, the boundary value problems and list properties of solutions. The main proof of \refthm{MainTheorem} will need to distinguish domains \(\Omega=\mathcal{O}\times\mathbb{R}\) with unbounded \(\mathcal{O}\) and bounded \(\mathcal{O}\). Section \ref{section:BoundedMainTheorem} gives the full proof for unbounded \(\mathcal{O}\), where the first subsection \ref{subsection:UnboundedStructure} gives an overview over the proof and indicates the major differences to the elliptic case, while all the following subsections contain the proofs of the different needed lemmas. Section \ref{subsection:geometryOfBoundary} discusses extensions of this perturbation theory to domains with rougher geometries and contains the proof of \refcor{cor:AltMainThm}. Section \ref{subsection:BoundNF} discusses the minor adaptations that are needed to prove \refcor{cor:MainThmWithoutSupCond}: a parabolic adaptation of a lemma that is standard in the elliptic setting is in the appendix. In Section \ref{section:BoundedMainTheorem}, we show the proof of \refthm{MainTheorem} for bounded \(\mathcal{O}\), which uses large parts of the proof for unbounded \(\mathcal{O}\), but contains too many deviations to be dealt with simultaneously.
Lastly, Section \ref{section:ProofNeumann} provides the proof of \refthm{thm:MaintheoremNeumann}, where, again, a fairly routine parabolic adaption is given in the appendix.

\section{Notations and Preliminaries}\label{section:Prelims}

\subsection{Basic definitions}

For every point in our domain we write \((X,t)\in \Omega:=\mathcal{O}\times\mathbb{R}\subset\mathbb{R}^{n+1}\), where the last component \(t\) is the time and \(\mathcal{O}\subset \mathbb{R}^n\) is a spatial domain.

We denote the parabolic distance as
\[|(Y,s)-(X,t)|_p:=\sqrt{|s-t| + \sum_{i=1}^n |y_i-x_i|^2}, \]
and the distance to the boundary as
\[\delta(X,t):=\inf_{(Y,s)\in\partial\Omega}|(X,t)-(Y,s)|_p.\]
The point that minimizes the distance of \(X\in \mathcal{O}\) to the boundary \(\partial\mathcal{O}\) is denoted by \(X^*\), i.e. \(|X-X^*|=\delta(X)\).

We introduce a variety of notations for balls and cubes
\begin{itemize}
    \item \(B(X,t,r)=\{(Y,s)\in \Omega; |(Y,s)-(X,t)|_p\leq r\}\),
    \item \(B(X,r)=\{Y\in \mathcal{O}; |Y-X|\leq r\}\),
    \item \(Q(X,t,r)=\{(Y,s)\in \Omega; |Y_i-X_i|\leq r, |s-t|\leq r^2\}\),
    \item \(\Delta(X^*,t,r)=\partial\Omega\cap Q(X^*,t,r)=\{(Y,s)\in \partial\Omega; |Y_i-X^*_i|\leq r, |s-t|\leq r^2\}\),
    \item \(T(\Delta(X^*,t,r))=\Omega\cap Q(X^*,t,r)=\{(Y,s)\in \Omega; |Y_i-X^*_i|\leq r, |s-t|\leq r^2\}\).
\end{itemize}

Now, we define families of Whitney cubes over a point \((X,t)\in\partial\Omega\) by
\[\mathcal{W}(X,t):=\{I; I=\Delta(X,t, 2^j) \textrm{ for some }j\in\mathbb{Z}\}\]
and
\[\mathcal{W}^+(X,t):=\{I^+; I=\Delta(X,t,2^j)\textrm{ for some }j\in\mathbb{Z}, 1\leq i\leq N\},\]
where 
\[(\Delta(X,t,2^j))^+:=\{(X,t)\in \Omega; |(X,t)-(X_i,t_i)|_p\leq 2^{j+1}, 2^{j-1}\leq \delta(X,t)\leq 2^{j+1}\}.\]
According to our definitions, \(\mathcal{W}\) contains boundary cubes, while \(\mathcal{W}^+\) contains Whitney cubes in \(\Omega\). 
We would also like to use the opportunity to introduce enlargements of these cubes \(I:=\Delta(Y,s,2^j)\) by
\[\tilde{I}:=\Delta(Y,s,5*2^{j-2})\qquad \textrm{and}\qquad \hat{I}:=\Delta(Y,s,6*2^{j-2}).\]

Similar to a Whitneybox decomposition over a point, we define a Whitney decomposition of a Carleson region over a boundary ball \(\Delta(X,t,r)\). We note that for \((X,t)\in \partial\Omega\), every scale \(j\in \mathbb{Z}\), and a small \(c>0\) there are \(N(j)\) points \((X^j_i,t^j_i)\) such that
\(I_i^j:=\Delta(X^j_i,t^j_i,2^j)\) have finite overlap and cover \(\Delta(X,t, r)\). Then
\[\mathcal{W}(\Delta(X,t,r)):=\{I^+|I=I_i^j=\Delta(Y^j_i,s^j_i,2^j), \textrm{ for some }j\in\mathbb{Z}, 1\leq i\leq N(j)\},\]
where all cubes in \(\mathcal{W}\) have finite overlap and cover \(Q(X,t,r)\cap\Omega\).
\smallskip

Furthermore, we define a cone with aperture \(\alpha\) as
\[\Gamma_\alpha(X,t):=\{(Y,s)\in \Omega; |(Y,s)-(X,t)|_p<(1+\alpha)\delta(Y,s)\}.\]

The \textit{nontangential maximal function} is defined as
\[\tilde{N}^{\alpha,\beta}_p(u)(X,t):=\sup_{(Y,s)\in \Gamma_\alpha(X,t)} \Big(\fint_{Q(Y,s,\beta \delta(Y,s)/2)}|u|^pdZd\tau\Big)^{1/p}.\]
If \(\alpha\) and \(\beta\) are not specified then we set them to \(1\) or take the quantity that is clear from context. If \(p\) is not specified it is set to be \(2\). If we write \(N(u)(X,t):=\sup_{(Y,s)\in \Gamma_\alpha(X,t)} |u(Y,s)|\) we mean the usual \textit{nontangential maximal function}.  

From the definitions of the enlargements, we see that
\[\Gamma_2(X,t)\subset \bigcup_{I\in \mathcal{W}(X,t)}I^+\subset\bigcup_{I\in \mathcal{W}(X,t)}\tilde{I}^+\bigcup_{I\in \mathcal{W}(X,t)}\hat{I}^+\subset\Gamma_3(X,t).\]

When writing \(\nabla\) we always mean the gradient with respect to the spatial coordinates \(X\), while the partial time derivative is always separated as \(\partial_t\). Throughout this work we will use the \textit{surface measure} on \(\partial\Omega\) which is given by
\[\sigma(E):=\mathcal{H}^{n}|_{\partial\Omega}(E) \qquad \textrm{for }E\subset\partial\Omega,\]
where \(\mathcal{H}^{n}|_{\partial\Omega}\) is the \(n\) dimensional Hausdorff measure restricted to \(\partial\Omega\). The parabolic measure with pole in \((X,t)\) will be written as \(\omega^{(X,t)}\).

We also denote the Carleson norm of a measure \(\mu\) as \[\Vert \mu\Vert_{\mathcal{C}}=\sup_{\Delta\subset\partial\Omega}\sigma(\Delta)^{-1}\int_{T(\Delta)}d\mu.\]

We set a parabolic operator as
\begin{align}L:=\partial_t-\Div(A(X,t)\nabla\cdot),\label{ParabolicOperator}\end{align}
and its adjoint as
\begin{align}L^*:=-\partial_t -\Div(A^*(X,t)\nabla \cdot),\label{AdjointOperator}\end{align}
where \(A\) is a bounded elliptic matrix, i.e there exists \(\lambda>0\) such that
\[\lambda |\xi|^2\leq A(X,t)\xi\cdot\xi\leq \frac{1}{\lambda} |\xi|,\qquad\textrm{for all }\xi\in \mathbb{R}^n,\textrm{ a.e. }(X,t)\in\Omega,\]
and every component of \(A\) is bounded.

We follow \cite{dindos_regularity_2023} (also \cite{jerison_boundary_1982}) for the following definitions.
\begin{defin}[Interior Corkscrew condition]\label{def:CorcscrewCondition}
	The domain \( \mathcal{O}\subset\mathbb R^n \) satisfies the \emph{corkscrew condition} with parameters \( M > 1, r_0 > 0 \) if, 
	    for each boundary cube \( \Delta \coloneqq \Delta(P,r) \) with \( P \in \partial \mathcal{O} \) 
	    and \( 0 < r < r_0 \),
		there exists a point \( A^X(P,r) \in \mathcal{O}  \), called a \emph{corkscrew point relative to} \( \Delta \),
		    such that \( B(A^X(P,r),M^{-1}r) \subset B(P,r) \). 
\end{defin}

We introduce the forwards and backwards in time corkscrew points of a boundary ball \(\Delta=\Delta(X,t,r)\subset\partial\Omega\), where \((X,t)\in\partial\Omega\), by
\[\bar{A}(\Delta):=(A^X(X,r),t +4r^2),\qquad\textrm{and}\qquad\underline{A}(\Delta):=(A^X(X,r),t -4r^2). \]
Here \(A^X(X,r)\in \mathcal{O}\) denotes the spatial corkscrew point from the previous definition.

\begin{defin}[Harnack chain]
	The domain \( \mathcal{O}\subset\mathbb{R}^n \) is said to satisfy the \emph{Harnack chain condition} if there is a constant \( c_0 \)
	such that for each \( \rho > 0, \; \Lambda \geq 1, \; X_1,X_2 \in \mathcal{O} \) 
	    with \( \delta(X_j) \geq \rho \) and \( |X_1 - X_2| \leq \Lambda \rho \),
	    there exists a chain of open balls \( B_1,\dots,B_N \subset \mathcal{O} \) 
	    with \( N \lesssim_{\Lambda} 1, \;  X_1 \in B_1, \; X_2 \in B_N, \; B_i \cap B_{i+1} \neq \EMP \)
	    and \( c_0^{-1} r(B_i) \leq \Dist(B_i,\partial \mathcal{O}) \leq c_0 r(B_i) \).
	The chain of balls is called a \emph{Harnack chain}.
\end{defin}

\begin{defin}[Uniform domain]\label{def:unifromDomain}
Let \(\mathcal{O}\subset\R^n \) satisfy the Harnack chain condition and the Interior Corkscrew condition. If \(\mathcal{O}\) is a set of locally finite perimeter and \((n-1)\) dimensional Ahlfors regular boundary, 
    i.e. there exists  \(C\geq 1\) so that for \(r\in (0,\mathrm{diam}(\Omega))\) and \(Q\in \partial \mathcal{O}\)
\[C^{-1}r^{n-1}\leq \sigma(B(Q,r))\leq Cr^{n-1},\]
then we call \(\mathcal{O}\) a uniform domain.
\end{defin}

Next, we define what it means for \(\mathcal{O}\) to be a Lipschitz domain.
\begin{defin}
    The set \(\mathcal{Z}\subset\mathbb{R}^n\) is an \(l\)-cylinder of diameter \(d\) if there exists an orthogonal coordinate system \((x,t)\in\mathbb{R}^{n-1}\times\mathbb{R}\) such that
    \[\mathcal{Z}=\{(x,t);|x|<d,|t|<2ld\}.\]
\end{defin}

\begin{defin}[Lipschitz domain]\label{def:Lipschitzdomain}
    Let \(\mathcal{O}\subset\mathbb{R}^n\) be a domain.
    If \(\mathcal{O}\) is bounded, we say \(\mathcal{O}\) is a \textit{Lipschitz domain} with character \((l,N,d)\) if there are \(l\)-cylinders \(\{\mathcal{Z}_j\}_{j=1}^N\) of diameter \(d\) satisfying the following conditions:
    \begin{enumerate}[(i)]
        \item \(8\mathcal{Z}_j\cap\partial\Omega\) is the graph of a Lipschitz function \(\phi_j\), \(\Vert\nabla \phi_j\Vert_\infty\leq l;\phi_j(0)=0\),
        \item \(\partial\Omega=\bigcup_{j=1}^N(\mathcal{Z}_j\cap\partial\Omega)\),
        \item \(\mathcal{Z}_j\cap\Omega\supset\Big\{(x,t)\in \Omega:|x|<d, \delta(x,t)\leq \frac{d}{2}\Big\}.\)
        \item Each cylinder \(\mathcal{Z}_j\) contains points from \(\Omega^C=\mathbb{R}^n\setminus\Omega\).
    \end{enumerate}

    If \(\mathcal{O}\) is unbounded, then we say \(\mathcal{O}\) is a \textit{Lipschitz domain} or \textit{Lipschitz graph domain} if there exists a Lipschitz function \(\phi:\mathbb{R}^{n-1}\to\mathbb{R}\) and a coordinate system \((x,t)\in \mathbb{R}^{n-1}\times\mathbb{R}\) such that \(\mathcal{O}=\{(x,t);t>\phi(x)\}\).
\end{defin}

In the following we are going to consider \textit{parabolic cylinders} as domains of the form
\[\Omega:=\mathcal{O}\times\mathbb{R},\]
where \(\mathcal{O}\) is a uniform domain. The domains in this article do not change in time direction.

\subsection{Reinforced weak and energy solutions}\label{subsection:EnergySol}

We are going to discuss solutions to \(Lu=0\) like in \eqref{eq:StandardPDE} and solutions to the Poisson equation
\begin{align}
    Lu=\partial_t u-\mathrm{div}(A\nabla u)=\mathrm{div}(h).\label{eq:PoissonEquation}
\end{align}

We follow the concept of reinforced weak and energy solutions that appeared already in \cite{auscher_l2_2020} and was slightly generalized in \cite{dindos_regularity_2023}, and \cite{dindos_lp_2025}. Let \(\Dt\) be the half-order derivative and \(\mathcal{H}_t\) the Hilbert transform with respect to the time variable \(t\), designed in such a way that \(\partial_t=\Dt\circ\mathcal{H}_t\circ\Dt\). We denote by \(W^{1,2}(\mathcal{O})\) the standard Sobolev space of real-valued functions \(v\) defined on \(\mathcal{O}\) such that \(v\) and \(\nabla v\) are in \(L^2(\mathcal{O},\mathbb{R})\) and \(L^2(\mathcal{O},\mathbb{R}^n)\), respectively. The subscript 'loc' will indicate that these conditions holds locally. 
\smallskip

We say that \(u:\Omega\to\mathbb{R}\) is a \textit{reinforced weak solution} of \(Lu=\partial_t u-\mathrm{div}(A\nabla u)=0\) on \(\Omega=\mathcal{O}\times\mathbb{R}\) if 
\[u\in \dot{E}_{loc}(\Omega):=\dot{H}^{1/2}_{loc}(\mathbb{R}, L^2_{loc}(\mathcal{O}))\cap L^2_{loc}(\mathbb{R},W^{1,2}_{loc}(\mathcal{O}))\]
and if for all \(\phi, \psi\in C^\infty_0(\Omega)\),
\begin{align}
    \int_\Omega A\nabla u\cdot\nabla (\phi\psi)+\mathcal{H}_t\Dt(u\psi)\cdot \Dt\phi + \mathcal{H}_t\Dt(u\phi)\cdot\Dt\psi dXdt=0.\label{eq:ReinforcedSolution}
\end{align}
The space \(\dot{H}^{1/2}(\mathbb{R})\) is the homogeneous Sobolev space of order \(1/2\) (it is the completion of \(C_0^\infty(\mathbb{R})\) for the norm \(\Vert \Dt \cdot\Vert_2\) and, modulo constants, it embeds into the space \(\mathcal{S}'/\mathbb{R}\) of tempered distributions modulo constants). By \(\dot{H}^{1/2}_{loc}(\mathbb{R})\) we denote all functions \(u\) such that \(u\phi\in \dot{H}^{1/2}(\mathbb{R})\) for every \(\phi\in C^\infty_0(\mathbb{R})\).

We remark that if \(u\in\dot{E}_{loc}\) satisfies \eqref{eq:ReinforcedSolution} for \(\phi,\psi\in C_0^\infty\), then
\[\int_\Omega A\nabla u\cdot\nabla (\phi\psi)-u\partial_t(\phi\psi)dXdt=0.\]
If we set \(\psi=1\) on the support of \(\phi\), this formulation states that \(u\) is a \textit{weak solution} and we see that every reinforced weak solution is a weak solution.

Furthermore, we say that a reinforced weak solution \(u\in \dot{E}_{loc}(\Omega)\) belongs to the \textit{energy class} \(\dot{E}(\Omega)\) if
\[\Vert v\Vert_{\dot{E}}:=\big(\Vert\nabla v\Vert_{L^2(\Omega)}^2 + \Vert\Dt v\Vert_{L^2(\Omega)}^2\big)^{1/2}<\infty.\]
Consequently, these are called \textit{energy solutions}.

As shown in \cite{auscher_l2_2020} (also in \cite{dindos_lp_2025}), the trace space of \(\dot{E}\) is the space \(\dot{H}^{1/4}_{\partial_t-\Delta}(\partial\Omega)\) where \(\dot{H}^{s}_{\pm\partial_t-\Delta}(\partial\Omega)\) is the closure of the space of Schwarz functions \(v\in\mathcal{S}(\partial\Omega)\) with Fourier support away from the origin with respect to the norm \(\Vert\mathcal{F}((|\xi|^2\pm i\tau)^s\mathcal{F}v)\Vert_{L^2(\partial\Omega)}\). The reverse is also true, i.e. every function \(g\in\dot{H}^{1/4}_{\partial_t-\Delta}(\partial\Omega)\) can be extended to a function \(v\in\dot{E}(\Omega)\) with \(v|_{\partial\Omega}=g\). We remark that all these arguments work on domains \(\Omega\) with Lipschitz bases \(\mathcal{O}\) because we can use a partition of the boundary to reduce the underlying arguments to arguments on the upper half space which was dealt with in \cite{auscher_l2_2020} (see also \cite{dindos_lp_2025}). 

By an energy solution \(u\) to \eqref{eq:StandardPDE} with Dirichlet data \(f\in \dot{H}^{1/4}_{\partial_t-\Delta}(\partial\Omega)\) we mean \(u\in \dot{E}(\Omega)\) such that \(u|_{\partial\Omega}=f\) and 
\[\int_\Omega A\nabla u\cdot \nabla v + \mathcal{H}_t\Dt u \cdot \Dt v dXdt=0 \]
for all \(v\in \dot{E}_0(\Omega)\), the subspace of \(\dot{E}(\Omega)\) with zero boundary trace.
By an energy solution \(u\) to \eqref{eq:StandardPDE} with Neumann data \(A\nabla u\cdot \nu=g\in \dot{H}^{-1/4}_{\partial_t-\Delta}(\partial\Omega)\) we mean \(u\in \dot{E}(\Omega)\) such that for all \(v\in\dot{E}(\Omega)\)
\begin{align}\int_\Omega A\nabla_X u\cdot \overline{\nabla_X v} + \mathcal{H}_t\Dt u \cdot \overline{\Dt v} dXdt=-\langle g, v|_{\partial\Omega}\rangle,\label{eq:NeummanSolution}\end{align}
where \(\langle \cdot, \cdot\rangle\) is the dual pairing of \(\dot{H}^{-1/4}_{\partial_t-\Delta}(\partial\Omega)\) with \(\dot{H}^{1/4}_{\partial_t-\Delta}(\partial\Omega)\).

To show wellposedness of the Dirichlet and Neumann problems with data in \(\dot{H}^{1/4}_{\partial_t-\Delta}(\partial\Omega)\) and \(\dot{H}^{-1/4}_{\partial_t-\Delta}(\partial\Omega)\) respectively, we use the modified sesquilinear form 
\[a_\delta(u,v):=\int_\Omega A\nabla_X u\cdot \overline{\nabla_X (1-\delta\mathcal{H}_t)v} + \mathcal{H}_t\Dt u \cdot \overline{\Dt (1-\delta\mathcal{H}_t)v} dXdt.\]
This form is coercive for a small choice \(\delta>0\), and hence we can apply Lax-Milgram to obtain existence of a unique energy solution \(u\in\dot{E}(\Omega)\) to 
\begin{align}a_\delta(u,v)=-\langle g, (1-\delta\mathcal{H}_t)v|_{\partial\Omega}\rangle.\label{eq:a_deltaNeumannDef}\end{align}

Since \((1-\delta\mathcal{H}_t)\) is invertible for a sufficiently small choice of \(\delta\), replacing \((1-\delta\mathcal{H}_t)v\) by \(v\) yields that \(u\) also satisfies \eqref{eq:NeummanSolution} and is a energy solution to \eqref{eq:StandardPDE} with Neumann data \(g\).

For the Dirichlet problem consider \(f\in \dot{H}^{1/4}_{\partial_t-\Delta}(\partial\Omega)\) and an extension \(w\in \dot{E}(\Omega)\) of \(f\). Then we can apply the Lax-Milgram lemma to \(a_\delta\) on \(\dot{E}_0(\Omega)\) to obtain some \(\tilde{u}\in \dot{E}_0(\Omega)\) with
\[a_\delta(\tilde{u},v)=-a_\delta(w,v)\qquad \textrm{for all }v\in \dot{E}_0(\Omega).\]
Due to linearity, \(u:=\tilde{u}+w\) is now an energy solution to \eqref{eq:StandardPDE} with Dirichlet data \(f\). This solution is unique, since for any other function \(v\) with these properties we have \(a_\delta(u-v,u-v)=0\), whence coercivity implies \(\Vert u-v\Vert_{\dot{E}}=0\).

Since an energy solution \(u\) to the Dirichlet problem is an extension of boundary data \(f\in\dot{H}^{1/4}_{\partial_t-\Delta}(\partial\Omega)\), we have \(\Vert u\Vert_{\dot{E}}\lesssim \Vert \tilde{u}\Vert_{\dot{E}} + \Vert w\Vert_{\dot{E}}\lesssim \Vert f\Vert_{\dot{H}^{1/4}_{\partial_t-\Delta}(\partial\Omega)}\). Similarly, we obtain for an energy solution \(u\) for the Neumann problem with \(g\in \dot{H}^{-1/4}_{\partial_t-\Delta}(\partial\Omega)\) from \eqref{eq:NeummanSolution}
\begin{align*}
\Vert u\Vert_{\dot{E}}^2&\lesssim a_\delta(u,u)=\langle g, (1-\delta\mathcal{H}_t)u|_{\partial\Omega}\rangle\lesssim \Vert g\Vert_{\dot{H}^{-1/4}_{\partial_t-\Delta}(\partial\Omega)} \Vert (1-\delta\mathcal{H}_t)u|_{\partial\Omega}\Vert_{\dot{H}^{1/4}_{\partial_t-\Delta}(\partial\Omega)}
\\
&\lesssim \Vert g\Vert_{\dot{H}^{-1/4}_{\partial_t-\Delta}(\partial\Omega)} \Vert (1-\delta\mathcal{H}_t)u\Vert_{\dot{E}}\lesssim \Vert g\Vert_{\dot{H}^{-1/4}_{\partial_t-\Delta}(\partial\Omega)} \Vert u\Vert_{\dot{E}}.
\end{align*}
In the last line we used Lemma 3.14 from \cite{auscher_l2_2020} and \(L^2-L^2\) boundedness of \((1-\delta\mathcal{H}_t)\). Hence we have \(\Vert u\Vert_{\dot{E}}\lesssim \Vert g\Vert_{\dot{H}^{-1/4}_{\partial_t-\Delta}(\partial\Omega)}\).

For the Poisson-Dirichlet problem we say a function \(u\in \dot{E}_0(\Omega)\) is a solution to \eqref{eq:PoissonEquation} with vanishing Dirichlet data and Poisson data \(h\in C_c^\infty(\Omega)\) if 
\begin{align}a_\delta(u,\phi)=\int_\Omega h\cdot \nabla_X(1-\delta\mathcal{H}_t)\phi dXdt\label{eq:PDSolution}\end{align}
for all \(\phi\in C_c^\infty(\Omega)\). Existence and Uniqueness of such a solution \(u\) follow again by Lax Milgram, and we note that \(\Vert u\Vert_{\dot{E}}\lesssim \Vert h\Vert_{L^2(\Omega)}\), since
\[\Vert u\Vert_{\dot{E}(\Omega)}\lesssim a_\delta(u,u)\lesssim \Vert h\Vert_{L^2(\Omega)}\Vert \nabla u\Vert_{L^2(\Omega)}\lesssim \Vert h\Vert_{L^2(\Omega)}\Vert u\Vert_{\dot{E}(\Omega)}.\]

\subsection{Boundary value problems}\label{subsection:BVPs}

\medskip

Let us define the \(L^p\) Dirichlet boundary value problem.
\begin{defin}[\((D)_p^L\)]
    Let \(f\in \dot{H}^{1/4}_{\partial_t-\Delta}(\partial\Omega)\cap L^p(\partial\Omega)\) and let \(u\in \dot{E}(\Omega)\) be the energy solution to \eqref{eq:StandardPDE} with Dirichlet data \(f\). If
    \[\Vert N(u)\Vert_{L^p(\partial\Omega)}\lesssim\Vert f\Vert_{L^p},\]
    where the implied constants are independent of \(u\) and \(f\), then we say that the \(L^p\) Dirichlet boundary value problem is solvable. We write \((D)^L_p\) holds.
\end{defin}

Before we can define the \(L^p\) Regularity boundary value problem, we need to introduce the space \(\dot{L}^p_{1,1/2}(\partial\Omega)\). For a Borel function \(f:\partial\mathcal{O}\to\mathbb{R}\) we say that a Borel function \(g:\partial\mathcal{O}\to\mathbb{R}\) is a \textit{Haj\l{}asz upper gradient} of \(f\) if
\[|f(X)-f(Y)|\leq |X-Y|(g(X)+g(Y))\qquad\textrm{for a.e. }X,Y\in \partial\mathcal{O}.\]
We denote the collection of all Haj\l{}asz upper gradients of \(f\) as \(\mathcal{D}(f)\). Now we can define the norm of \(\dot{L}_1^p(\partial\mathcal{O})\) by
\[\Vert f\Vert_{\dot{L}^p_1(\partial\mathcal{O})}:=\inf_{g\in\mathcal{D}(f)}\Vert g\Vert_{L^p(\partial\mathcal{O})}.\]
This space is also called homogeneous \textit{Haj\l{}asz Sobolev space}.
\smallskip

This definition of \(\dot{L}_1^p\) simplifies in settings of higher regularity. For example, if \(f\) is a Lipschitz function and \(\mathcal{O}\) a Lipschitz domain, then
\[\Vert f \Vert_{\dot{L}_1^p(\partial\mathcal{O})}\approx\Vert \nabla_T f \Vert_{L^p(\partial\mathcal{O})},\]
where \(\nabla_T\) is the tangential derivative (cf. (1.1) in \cite{mourgoglou_regularity_2023}). If \(\mathcal{O}\) is Lipschitz, the tangent space in almost every point \(P\in\partial\Omega\) has dimension \(n-1\) and we have a tangent space basis \(T_1(P), ..., T_{n-1}(P)\) of unit vectors. Then the tangential derivative in these points is given by \(\nabla_T f:=\sum_{i=1}^{n-1}(\nabla f\cdot T_i(P))T_i(P)\) and hence we can equally define
\[\Vert f\Vert_{\dot{L}^p_1(\partial\mathcal{O})}:=\Big(\int_{\partial\mathcal{O}}|\nabla_T f|^p d\sigma\Big)^{1/p}:=\Big(\sum_{i=1}^{n-1}\int_{\partial\mathcal{O}}|\nabla f\cdot T_i(P)|^p d\sigma\Big)^{1/p}\]
in this case.

Since in the parabolic setting the time derivative scales like the second order spatial derivative, if we want to control the tangential spatial derivative, which is a first order derivative, we also need to control the half-derivative in time. Hence, we define the \textit{parabolic homogeneous Haj\l{}asz Sobolev space} \(\dot{L}^p_{1,1/2}(\partial\Omega)\) by 
\[\Vert f\Vert_{\dot{L}^p_{1,1/2}(\partial\Omega)}:=\big(\int_\mathbb{R}\Vert f(\cdot,t)\Vert_{\dot{L}_1^p(\partial\mathcal{O})}^p dt + \Vert \Dt f\Vert_{L^p(\partial\Omega)}^p\big)^{1/p}.\] We define the Regularity problem next.

\begin{defin}[\((R)_p^L\)]\label{def:DefintionOfRegularityProblem}
    Let \(f\in \dot{H}^{1/4}_{\partial_t-\Delta}(\partial\Omega)\cap \dot{L}^p_{1,1/2}(\partial\Omega)\) and let \(u\in \dot{E}(\Omega)\) be the energy solution to \eqref{eq:StandardPDE} with Dirichlet boundary data \(f\). If
    \begin{align}\Vert \tilde{N}(\nabla u)\Vert_{L^p(\partial\Omega)}\lesssim\Vert f\Vert_{\dot{L}_{1,1/2}^p(\partial\Omega)},\label{eq:RPinequality}\end{align}
    where the implied constants are independent of \(u\) and \(f\), then we say that the \(L^p\) Regularity boundary value problem is solvable for \(L\). We write \((R)^L_p\) holds for \(L\).
\end{defin}

\begin{rem}
    It appears that in our definition of \((R)_p\) the term \(\Vert \tilde{N}(\Dt u)\Vert_p\) is missing on the left side of \eqref{eq:RPinequality}. However, \cite{dindos_regularity_2023} shows that \eqref{eq:RPinequality} already implies
    \[\Vert \tilde{N}(\Dt u)\Vert_{L^p(\partial\Omega)}\lesssim\Vert f\Vert_{\dot{L}_{1,1/2}^p(\partial\Omega)},\]
    and hence our definition is in fact equivalent to defining \((R)_p\) by 
    \[\Vert \tilde{N}(\nabla u)\Vert_{L^p(\partial\Omega)} + \Vert \tilde{N}(\Dt u)\Vert_{L^p(\partial\Omega)}\lesssim\Vert f\Vert_{\dot{L}_{1,1/2}^p(\partial\Omega)}.\]
\end{rem}

Now we can define the \(L^p\) Neumann boundary value problem.
\begin{defin}[\((N)_p^L\)]\label{def:DefintionOfNeumannProblem}
    Let \(f\in \dot{H}^{-1/4}_{\partial_t-\Delta}(\partial\Omega)\cap L^p(\partial\Omega)\) and let \(u\in \dot{E}(\Omega)\) be the energy solution to \eqref{eq:StandardPDE} with Neumann data \(f\), i.e. \(u\) satisfies \eqref{eq:NeummanSolution}. If
    \begin{align}\Vert \tilde{N}(\nabla u)\Vert_{L^p(\partial\Omega)}\lesssim\Vert f\Vert_{L^p(\partial\Omega)},\end{align}
    where the implied constants are independent of \(u\) and \(f\), then we say that the \(L^p\) Neumann boundary value problem is solvable for \(L\). We write \((N)^L_p\) holds for \(L\).
\end{defin}

\begin{rem}\label{rem:nontangnetialmeanvaluedL1L2}
    The Regularity and Neumann problem are both formulated with control over the \(L^2\) mean-valued nontangential maximal function \(\tilde{N}(\nabla u)=\tilde{N}_2(\nabla u)\). In the proof of \refthm{MainTheorem} and \refthm{thm:MaintheoremNeumann} we will show bounds on the \(L^1\) mean-valued version \(\tilde{N}_1(\nabla u)\). To reconcile these bounds we can observe that we have \(\tilde{N}_1(\nabla u)\leq \tilde{N}_2(\nabla u)\) pointwise and by the higher integrability of the gradient and purely real variable arguments, we have that
$$\left(\fint_B|\nabla u|^2\right)^{1/2}\lesssim \left(\fint_{2B}|\nabla u|\right),$$
which implies a bound of $\tilde{N}_2(\nabla u)(\cdot)$ defined using cones $\Gamma_a(\cdot)$ of some aperture $a>0$ by $\tilde{N}_1(\nabla u)(\cdot)$ defined using cones $\Gamma_b(\cdot)$ of some slightly larger aperture $b>a$ (cf. Lemma 3.2 in \cite{dindos_lp_2025}). 
\end{rem}

Let us recall some relationships between different boundary value problems that will be relevant.

\begin{prop}[\cite{dindos_satterqvist_24}]\label{prop:ParabolicPR_pimpliesRP_q}
    Let \(L\) be a parabolic operator like in \eqref{ParabolicOperator} and \(\Omega=\mathcal{O}\times\mathbb{R}\), where \(\mathcal{O}\) is a Lipschitz domain. Then \((R)^L_p\) implies \((R)^L_q\) for \(1< q<p\).
\end{prop}

An important connection between the Dirichlet and the Regularity problem is the following proposition from \cite{dindos_parabolic_2019}.

\begin{prop}\label{prop:ParabolicRPimpliesDP}
    Let \(L\) be a parabolic operator like in \eqref{ParabolicOperator} and \(\Omega=\mathcal{O}\times\mathbb{R}\), where \(\mathcal{O}\) is a Lipschitz domain. If \((R)^L_p\) is solvable, then \((D)^{L^*}_{p'}\) is solvable, where \(\frac{1}{p}+\frac{1}{p'}=1\). 
\end{prop}

Furthermore, let us denote the parabolic measure to the operator \(L\) as \(\omega_L\), the Muckenhoupt spaces as \(A_p(\sigma)\), and the reverse Hölder spaces as \(B_q(\sigma)\). Their definitions and main properties go back to \cite{coifman_weighted_1974}. In the parabolic setting specifically, we can also refer to chapter 6 of \cite{nystrom_dirichlet_1997}. For us relevant will only be the following three properties:
    \begin{enumerate}[(i)]
        \item It holds \(A_\infty(\sigma)=\bigcup_{p>1}B_p(\sigma)\);
        \item Let \(\omega\in A_\infty(\sigma)\). Then for each \(\EPS>0\) there exists \(\delta>0\) such that for all \(E\subset\Delta\subset\tilde{\Delta}\)
        \begin{align}\frac{\omega^{A(\tilde{\Delta})}(E)}{\omega^{A(\tilde{\Delta})}(\Delta)}<\delta\qquad\textrm{implies}\qquad \frac{\sigma(E)}{\sigma(\Delta)}<\EPS.\label{eq:A_inftyProperty}\end{align}
        \item Let 
        \[M_\mu(f)(Q,t):=\sup_{r>0}\frac{1}{\mu(\Delta(Q,t,r))}\int_{\Delta(Q,t,r)}|f|d\mu\]
        for a boundary point \((Q,t)\in\partial\Omega\). Then \(\omega_L\in B_q(\sigma)\) if and only if for all boundary balls \(\tilde{\Delta}\subset\partial\Omega\)
        \begin{align}\int_{\tilde{\Delta}}|M_{\omega^{A(\tilde{\Delta})}}(f)|^{q'} d\sigma\leq C\int_{\tilde{\Delta}}|f|^{q'}d\sigma.\label{eq:BoundMaxRevHölder}\end{align}
    \end{enumerate}
The important connection of the Muckenhoupt spaces to the Dirichlet problem is the same as in the elliptic setting. 
\begin{prop}[\cite{nystrom_dirichlet_1997}, Theorem 6.2]\label{prop:DPequivalentB_p}
    For a parabolic operator \(L\) and its corresponding parabolic measure \(\omega\) we have that \(\omega\in B_p(\sigma)\) if and only if \((D)_{p'}\) is solvable.
\end{prop}

For the parabolic Dirichlet boundary value problem we have the following perturbation results.

\begin{prop}\label{prop:ParabolicPerturbationTheoryForDP}
Let \(\Omega\) be a parabolic cylinder \(\Omega=\mathcal{O}\times\mathbb{R}\) where \(\mathcal{O}\) is Lipschitz. Let \(L_0\) and \(L_1\) be two parabolic operators like in \eqref{ParabolicOperator} and assume \eqref{eq:CarlesonWithSupNorm}.

Then the following small and large Carleson norm results hold:
    \begin{itemize}
        \item[(S)] There exists a sufficiently small \(\EPS=\EPS(p,\Omega,n, \lambda_0)>0\), such that if \(\omega_{L_0}\in B_p(\sigma)\) and \(\Vert \mu\Vert_{\mathcal{C}}<\EPS\), then \(\omega_{L_1}\in B_p(\sigma)\);
        \item[(L)] If \(\omega_{L_0}\in A_\infty(\sigma)\) and \(\Vert \mu\Vert_{\mathcal{C}}<\infty\), then \(\omega_{L_1}\in A_\infty(\sigma)\).
    \end{itemize}
\end{prop}

For Lip(1,1/2) domains this proposition can be found in \cite{nystrom_dirichlet_1997} (Theorem 6.6 and Theorem 6.5). The small Carleson norm perturbation has also independently been done in \cite{sweezy_bq_1998} using the same techniques that where inspired by the first Dirichlet perturbation paper in the elliptic setting \cite{fefferman_theory_1991}. It is also noteworthy that if \(\mathcal{O}=\mathbb{R}^n_+\) is the upper half space, then \cite{hofmann_dirichlet_2001} shows the large Carleson perturbation result for operators with drift term in which setting one does not have the doubling property of the parabolic measure which complicates the proof and requires adaptations.

\subsection{Properties of solutions and the Green's function}\label{subsection:PropertiesOfSolutions}

Here, we list properties of solutions \(u\) to \eqref{eq:StandardPDE} and the Green's function \(G\) on parabolic cylinder domains of form \(\Omega=\mathcal{O}\times\mathbb{R}\). The tools we need are considered to be fairly standard and are well established if \(\mathcal{O}\) is Lipschitz. However, if \(\mathcal{O}\) is uniform, these results are fairly new in the literature, and we would like to provide some references.

First, let us recall the inner properties of a solution \(u\) which are independent of the geometry of the boundary: Parabolic Cacciopolli inequality (Lemma 3.3 of Chapter I in \cite{hofmann_dirichlet_2001}), and Parabolic Harnack inequality (Lemma 3.5 of Chapter I in \cite{hofmann_dirichlet_2001}). Next, \cite{cho_greens_2008} shows the existence of the Green's function in a setting even more general than ours and proves several important properties. We are going to write the Green's function as \(G(X,t,Y,s)\) where its pole lies in the point \((X,t)\). One of the important properties, that was not shown in \cite{cho_greens_2008}, is the following growth estimate.  

\begin{lemma}[Growth estimate, Thm. 3.1 in \cite{kim_greens_2021}]\label{lemma:ParabolicGrowthEstimateForG}
    We have an growth estimate of the Green's function of the form
    \[G(X,t,Z,\tau)\lesssim \frac{1}{|t-\tau|^{n/2}}\exp\Big(-c\frac{|X-Z|^2}{|t-\tau|}\Big).\]
\end{lemma}

We use this growth estimate and a parabolic version of Boundary Hölder inequality (cf. \cite{fabes_backward_1986} or \cite{nystrom_dirichlet_1997}) to get a refined growth estimate that improves the previous one in certain situations.
\begin{lemma}\label{lemma:ParabolicGrwothEstimateForGwithalpha}
    There exists \(\alpha=\alpha(\lambda, n)\) such that for \((X,t),(Y,s)\in\Omega\) with \(t<s\)
    \begin{align}G(X,t,Y,s)\lesssim \delta(X,t)^\alpha \big(|X-Y|^2+(s-t)\big)^{-\frac{n+\alpha}{2}}.\label{eq:ParabolicGrwothEstimateForGwithalpha}\end{align}
\end{lemma}

The proof stems from private communications with Martin Dindo\v{s}, Linhan Li, and Jill Pipher and can be found as Lemma 2.8.17 in \cite{ulmer_l_2024} or in \cite{dindos_lp_2025}.
\medskip

At this point we would like to stress that with exception of the Boundary Hölder inequality, none of the above results track the boundary, and hence all of them apply in settings with rougher \(\mathcal{O}\). However, all the following properties have not previously appeared on parabolic cylinders with merely uniform base \(\mathcal{O}\). First, let us mention that the Boundary Hölder inequality has been proven as Lemma 2.3.1 in \cite{cho_boundary_2005} and Lemma 2.7.18 in \cite{ulmer_l_2024}, where the comparison to the function value at a corkscrew point is added. The thesis \cite{ulmer_l_2024} also contain Boundary Cacciopolli inequality (cf. Lemma 2.7.15), Comparison Principle (cf. Lemma 2.8.19), Backwards Harnack inequality for the Green's function (cf. Lemma 2.8.23), and 
Local Comparison Principle (cf. Lemma 2.7.18). All of these results are proved on parabolic cylinders with uniform \(\mathcal{O}\). Since this rougher setting is not the focus of this article, we omit discussing the details of these results and refer the interested reader to \cite{ulmer_l_2024}. However, this supports the claim in the introduction about the possibility to extend the main results of this article to rougher settings.

\section{Proof of \refthm{MainTheorem} for unbounded base \(\mathcal{O}\)}\label{section:UnboundedMainProof}

First, we will restrict to the case where we assume that \(\mathcal{O}\) is unbounded. The bounded base case works similarly and we will discuss its required adaptations in the last section.

\subsection{Structure of the proof of \refthm{MainTheorem}}\label{subsection:UnboundedStructure}

First, let \(f\in \dot{H}^{1/4}_{\partial_t-\Delta}(\partial\Omega)\cap \dot{L}^{q_0}_{1,1/2}(\partial\Omega)\) be the boundary data for the Regularity problem of \(L_0\) and \(L_1\)
and \(u_0\) and \(u_1\) the corresponding solutions. Assume without loss of generality that \(0=(0,0)\in \partial\Omega\).

To proof \refthm{MainTheorem}, we consider a truncated version of the nontangential maximal function first. Using truncated versions first allows us to justify the convergence of all involved integrals and is of purely technical nature.
\begin{defin}
    We define the truncated nontangential maximal function
    \[\tilde{N}_{1,\rho}[g](P,t):=\sup_{(Y,s)\in \Gamma^{\rho}(P,t)}\fint_{B(Y,s,\delta(Y,s)/2)}|g(X,\tau)|dXd\tau, \]
    where 
    \[\Gamma^{\rho}(P,t):=\Gamma(P,t)\cap Q(0,0,\frac{1}{\rho})\cap \{(X,s):\delta(X,s)>\rho\}.\]
\end{defin}

We set \(\EPS:=A_0-A_1\). We also set \(F:=u_1-u_0\) as the difference function with its explicit representation given in the following lemma.
\begin{lemma}\label{lemma:defF}
    We set \(F:=u_1-u_0\). Then
    \begin{align*}F(X,t)&=\int_\Omega \EPS(Y,s)\nabla u_0(Y,s)\cdot\nabla_Y G_1^*(X,t,Y,s) dYds
    \\
    &=\int_\Omega \EPS(Y,s)\nabla u_0(Y,s)\cdot\nabla_Y G_1(Y,s,X,t) dYds.
    \end{align*}
\end{lemma}

Apart from minor adjustments the proof of this result is a fairly standard result that is well known in the elliptic and parabolic setting (cf. \cite{kenig_neumann_1995}, \cite{fefferman_theory_1991}, \cite{milakis_harmonic_2013}, \cite{hofmann_dirichlet_2001}, \cite{dindos_perturbation_2023} and more). In our specific case the rigorous proof can also be found as Lemma 4.2.3 in \cite{ulmer_l_2024}.

\smallskip

Next, we need the following lemma whose proof we are also going to omit. In our parabolic setting, the proof can be found as Lemma 4.2.4 in \cite{ulmer_l_2024}, but it does follow almost verbatim the proof in the elliptic setting presented in \cite{kenig_neumann_1995}.

\begin{lemma}\label{lemma:defh}
    There are quasi-dualising functions \(h=h^\rho\in L^\infty(\Omega)\) such that
    \begin{align*}
        \Vert\tilde{N}_{1,\rho}[\nabla F]\Vert_{L^q(\partial\Omega)}&\lesssim \int_\Omega \nabla F\cdot h dXdt,
    \end{align*}
    where 
    \begin{align*}
        h(Z,\tau):&=\int_{\partial\Omega}g(P,t)\Big(\int_{\Gamma^{\varepsilon}(x,t)} \beta(P,t,Y,s)\frac{\alpha(Y,s,Z,\tau)}{\delta(Y,s)^{n+2}}\chi_{B(Y,s,\delta(Y,s)/2)}(Z,\tau)dYds\Big)dPdt,
    \end{align*}
    
    with 
    \begin{itemize}
        \item \(\Vert\alpha(Y,s,\cdot)\Vert_{L^\infty(B(Y,s,\delta(Y,s)/2))}=1\),
        \item \(\Vert\beta(P,t,\cdot)\Vert_{L^1(\Gamma_\rho(P,t))}=1\),
        \item \(\Vert g\Vert_{L^{q'}(\partial\Omega)}=1\).
    \end{itemize}

    It follows that \(\mathrm{supp}(h^\rho)\subset Q(0, \frac{1}{\rho})\cap\{(X,s): \delta(X,s)>\rho\}.\) 
\end{lemma}

We set \(\omega^\rho\) as the parabolic measure of \(L_1^*\) with pole in the backward in time corkscrew point \(\underline{A}(8\Delta_\rho)\) where \(\Delta_\rho:=\Delta(0, \frac{1}{\rho})\). We will need the following function \(v\) that solves \eqref{vDefinedByDivh}.

\begin{lemma}\label{lemma:defv}
    Let \(v\in \dot{E}(\Omega)\) solve the inhomogeneous parabolic PDE
    \begin{align}\begin{cases}
        L_1^* v=\Div_X(h) & \textrm{ in }\Omega, \\
        v=0 & \textrm{ on }\partial\Omega.
    \end{cases}\label{vDefinedByDivh}\end{align}
     Then, for all \(1<q<\infty\), we have 
     \[\Vert \tilde{N}_{1,\rho}(\nabla F)\Vert_{L^q(\partial\Omega)}\lesssim \Vert \mu\Vert_{\mathcal{C}}\Vert\tilde{N}_{1,\rho}(\nabla u_0)\Vert_{L^q(\partial\Omega)}\Vert S(v)\Vert_{L^{q'}(\partial\Omega)}.\]
\end{lemma}

This proof works by approximation and stopping time arguments and is presented in a later subsection. In fact, in proving this lemma we need to distinguish between whether we assume \eqref{eq:CarlesonWithSupNorm} or \eqref{eq:CarlesonWithoutSupNorm}, but both assumptions lead to the same result.

\smallskip

We introduce the following function that resembles a discretized tent space function (cf. \cite{Coifman_tentspace_85} for more about tent spaces).
\begin{defin} For \(1\leq p<\infty\) we define for \((P,t)\in\partial\Omega\)
    \[T_p(h)(P,t):=\sum_{I\in\mathcal{W}(P,t)}l(I)\Big(\fint_{I^+}|h|^{2p}\Big)^{1/{2p}}.\]
\end{defin}

Up to this point the modifications to the parabolic setting are minor and mostly of technical nature. However, the following two lemmas are key lemmas in the elliptic proof of \cite{kenig_neumann_1995} (or also in \cite{dai_carleson_2022}). Despite that the results look very similar, the parabolic setting complicates the two proofs significantly. The main difficulties are that we have to move the pole of the Green's function with more care in the parabolic setting, and that the decay of the Green's function and \(v\) are different in time than they are in space. Another idea, that appears in \cite{dai_carleson_2022} in the elliptic setting at this point, is to use a Green's function with pole in infinity. However, this Green's function has not been constructed in the parabolic setting, and we avoid its use by essentially moving the pole to a point far away from the compact support of \(h^\rho\).
\begin{lemma}\label{Lemma2.10}
Let \(v\) be as in \eqref{vDefinedByDivh}. Then for have for \(p>\frac{n}{2}+1\) and \((P,t)\in 3\Delta_\rho\)
\[N(v)(P,t)+\tilde{N}(\delta|\nabla v|)(P,t)\lesssim M_{\omega^\rho}(T_p(h))(P,t).\]
Furthermore, we have for 
\((P,t)\in \partial\Omega \setminus3\Delta_\rho\)
\[N(v)(P,t)+\tilde{N}(\delta|\nabla v|)(P,t)\lesssim \inf_{(Y,s)\in 2\Delta_\rho} M_{\omega^\rho}(T_p(h))(Y,s).\]
\end{lemma}

\begin{lemma}\label{Lemma2.11}
    For every \(\gamma>0\) there exist \(\eta>0\) such that
    \[\sigma(\{S(v)>2\lambda; N(v)+\tilde{N}(\delta|\nabla v|) + T_2(h))
    \leq \eta\lambda\}\leq \gamma\sigma(\{S(v)>\lambda\}).\]
\end{lemma}

Both of these important key lemmas will be proved in the following sections. A corollary of \reflemma{Lemma2.10} and \reflemma{Lemma2.11} is the following.

\begin{cor}\label{cor:S(v)lesssimT_p(h)}
We have for \(r>1\)
\[\int_{\partial\Omega} S(v)^r d\sigma \leq C\int_{8\Delta_\rho} M_{\omega^\rho}(T_p(h))^r d\sigma. \]
\end{cor}

Although this corollary seems to be a direct consequence of the good \(\lambda\)-inequality, we need to be a bit careful with restricting the integration domain on the right side. For that, we need good enough decay estimates for the square function far away from the support of \(h\), which are harder to establish in the parabolic setting, since the decay in time direction is different from decay in spatial directions. However, this restriction of the integration domain is necessary, so that we can bound the right side of the previous corollary in some specific \(L^r\) norm. This is the next lemma.

\begin{lemma}\label{lemma:BoundednessOfS(v)}
    There exists \(1<q_1\leq q_0\) depending on \(\Vert\mu\Vert_\mathcal{C}\) such that
    \[\int_{8\Delta_\rho} M_{\omega^\rho}(T_p(h))^{q_1'} d\sigma\leq C<\infty.\]
\end{lemma}

\begin{proof}
    Since \((R)^{L_0}_{q_0}\) is solvable, we have by \refprop{prop:ParabolicRPimpliesDP} that \((D)^{L_0^*}_{q_0'}\) is solvable, i.e. \(\omega_{L_0^*}\in B_{q_0}(\sigma)\). Hence, \refprop{prop:ParabolicPerturbationTheoryForDP} gives us \(\omega_{L_1^*}\in B_{q_1}(d\sigma)\) or equivalently \((D)^{L_0^*}_{q_1'}\) for \(q_1\leq q_0\). Note here that \(q_1\) might be equal to \(q_0\) if \(\Vert \mu\Vert_{\mathcal{C}}\) is sufficiently small.

    Since \(\omega^\rho\) is the parabolic measure to the adjoint \(L_1^*\) and \(\omega_{L_1^*}\in B_{q_1}(d\sigma)\) we have by \eqref{eq:BoundMaxRevHölder}
    \[\Vert M_{\omega^\rho}(T_p(h))\Vert_{L^{q_1'}(\Delta_\rho)}\lesssim \Vert T_p(h)\Vert_{L^{q_1'}(\partial\Omega)}.\]
    The right side is bounded by the following \reflemma{lemma:BoundednessOfT(h)}.
\end{proof}

\begin{lemma}\label{lemma:BoundednessOfT(h)}
Set \(\tilde{T}(|g|)(P,t)=\int_{\Gamma_3(P,t)}|g|\delta^{-n-1}dYds\) for \((P,t)\in\partial\Omega\). There exists a function \(\gamma\) such that we have for every \((P,t)\in\partial\Omega\) and all \(p\geq 1, 1<q<\infty\)
\begin{itemize}
    \item \(T_p(|h|)(P,t)\lesssim_p\tilde{T}(\gamma)(P,t),\)
    \item \(\Vert \tilde{T}(\gamma)\Vert_{L^{q'}(\partial\Omega)}\leq C<\infty.\)
\end{itemize}
\end{lemma}

The proof is a generalisation of Lemma 2.13 in \cite{kenig_neumann_1995} to the parabolic case but works completely analogously. Hence we omit the proof here, and refer the reader for full completeness to Lemma 4.2.12 in \cite{ulmer_l_2024}. Again through the lens of tent spaces, if we used a nondiscretized version of \(T_p\), this lemma would state that \(\delta h\) lies in every tent space \(T_q^p\). 

\smallskip
With all these statements we are finally able to proof the main theorem \refthm{MainTheorem}.

\begin{proof}[Proof of \refthm{MainTheorem}]
According to \refdef{def:DefintionOfRegularityProblem} and \refrem{rem:nontangnetialmeanvaluedL1L2} we need to show
\begin{align}\Vert \tilde{N}_1(\nabla u_1)\Vert_{L^{q_1}(\partial\Omega)}\lesssim \Vert f\Vert_{\dot{L}_{1,1/2}^{q_1}(\partial\Omega)},\label{TimeDerivativeEstimate}\end{align}
to conclude that \((R)^{L_1}_{q_1}\) is solvable. Let \(\rho>0\) and \(q_1\) be given by \reflemma{lemma:BoundednessOfS(v)}. Note that if \(\Vert\mu\Vert\) is sufficiently small, then \(q_1=q_0\), but in general we have only \(1<q_1\leq q_0\). 
First, we see with \reflemma{lemma:defv}, \refcor{cor:S(v)lesssimT_p(h)} and \reflemma{lemma:BoundednessOfS(v)} that
\begin{align}
\Vert \tilde{N}_{1,\rho}(\nabla u_1)\Vert_{L^{q_1}(\partial\Omega)}&\leq \Vert \tilde{N}_{1,\rho}(\nabla F)\Vert_{L^{q_1}(\partial\Omega)}+ \Vert \tilde{N}_{1,\rho}(\nabla u_0)\Vert_{L^{q_1}(\partial\Omega)}\nonumber
\\
&\lesssim \Vert \tilde{N}_{1,\rho}(\nabla u_0)\Vert_{L^{q_1}(\partial\Omega)} \Vert S(v)\Vert_{L^{q_1'}(\partial\Omega)} + \Vert f\Vert_{\dot{L}_{1,1/2}^{q_1}(\partial\Omega)}\nonumber
\\
&\lesssim \Vert \tilde{N}_{1,\rho}(\nabla u_0)\Vert_{L^{q_1}(\partial\Omega)} + \Vert f\Vert_{\dot{L}_{1,1/2}^{q_1}(\partial\Omega)}.\label{eq:Nontangentialnablau1BoundedByNontangnetialu0Andf}
\end{align}
\begin{comment}
At this point we need to distinguish the small Carleson norm and the large Carleson norm case. If \(\Vert \mu\Vert_{\mathcal{C}}\leq \EPS\) for a sufficiently small \(\EPS\), then \(q_1=q_0\) and we have solvability of \((R)^{L_0}_{q_1}\), whence 
\[\Vert \tilde{N}_{1,\rho}(\nabla u_0)\Vert_{L^{q_1}(\partial\Omega)}\leq \Vert \tilde{N}_{1}(\nabla u_0)\Vert_{L^{q_1}(\partial\Omega)} \lesssim \Vert f\Vert_{\dot{L}_{1,1/2}^{q_1}(\partial\Omega)}.\]
\end{comment}
If we are in the large perturbation case, i.e. \(\Vert\mu\Vert_{\mathcal{C}}\) is large, then we need \refprop{prop:ParabolicPR_pimpliesRP_q} to conclude that solvability of \((R)^{L_0}_{q_0}\) implies solvability of \((R)^{L_0}_{q_1}\), while in the small perturbation case, i.e. \(\Vert\mu\Vert_{\mathcal{C}}\) is sufficiently small, we already have \((R)^{L_0}_{q_1}\) without applying \refprop{prop:ParabolicPR_pimpliesRP_q}. Thus, we obtain
\begin{align*}
\Vert \tilde{N}_{1,\rho}(\nabla u_0)\Vert_{L^{q_1}(\partial\Omega)}\lesssim \Vert f\Vert_{\dot{L}_{1,1/2}^{q_1}(\partial\Omega)},
\end{align*}
and from \eqref{eq:Nontangentialnablau1BoundedByNontangnetialu0Andf}
\begin{align*}
\Vert \tilde{N}_{1,\rho}(\nabla u_1)\Vert_{L^{q_1}(\partial\Omega)}\lesssim \Vert f\Vert_{\dot{L}_{1,1/2}^{q_1}(\partial\Omega)}.
\end{align*}

Lastly, let \(\rho\) go to \(0\). Since the constants do not depend on \(\rho\) we obtain the same right side. Since \(\tilde{N}_{1,\rho}(\nabla u_1)\leq \tilde{N}_{1}(\nabla u_1)\) holds pointwise and above inequality also shows uniform boundedness of the quantity \(\Vert \tilde{N}_{1,\rho}(\nabla u_1)\Vert_{L^{q_1}(\Omega)}\), we get by Monotone Convergence Theorem that
\[\Vert \tilde{N}_{1,\rho}(\nabla u_1)\Vert_{L^{q_1}(\Omega)}\to \Vert \tilde{N}_{1}(\nabla u_1)\Vert_{L^{q_1}(\Omega)}.\]
This completes the proof.

\end{proof}

\subsection{Discussion of the role of the geometry of the boundary}\label{subsection:geometryOfBoundary}

The main result \refthm{MainTheorem} is formulated on cylinders \(\Omega=\mathcal{O}\times\mathbb{R}\) where \(\mathcal{O}\) is a Lipschitz domain (see \refdef{def:Lipschitzdomain}). However, we pointed out in Section \ref{section:Prelims} that some of the important ground work for us has been done on domains with rougher boundary. In particular, the article \cite{dindos_regularity_2023} proves that \(\Vert \tilde{N}(D_t^{1/2} u)\Vert_{L^p}\lesssim \Vert \tilde{N}(\nabla u)\Vert_{L^p} + \Vert f\Vert_{\dot{L}^p_{1,1/2}}\) for a solution to \eqref{eq:StandardPDE} with boundary data \(f\in \dot{L}^p_{1,1/2}(\partial\Omega)\) on cylinders with uniform base \(\mathcal{O}\) (see \refdef{def:unifromDomain}). In addition to that, Section \ref{subsection:PropertiesOfSolutions} points out that most of the important properties of solutions that are generally considered to be standard tools in this area have recently been extended to parabolic cylinder domains with base \(\mathcal{O}\) of this rough type (cf. \cite{ulmer_l_2024}).
\begin{comment}
More precisely, in both of these articles the authors assume that \(\mathcal{O}\) is a 1-sided nontangentially accessible domain, which means that \(\mathcal{O}\) satisfies the interior corkscrew and the Harnack chain condition, with a \(n-1\) dimensional Ahlfors regular boundary (cf. \cite{dindos_regularity_2023} for definitions).  
\end{comment}
\smallskip

We will now discuss why our main result \refthm{MainTheorem} cannot be formulated on such domains yet and what the limitations are. We are going to discuss the small and large perturbation case separately.
\smallskip

For the small perturbation result \((S)\) of \refthm{MainTheorem} we use solvability of \((R)_{q_0}^{L_0}\) to deduce that \((D)_{q_0'}^{L_0^*}\) is solvable (\refprop{prop:ParabolicRPimpliesDP}). This conclusion is proved in \cite{dindos_parabolic_2019}, but only under the assumption that the domain \(\Omega\) is a Lip(1,1/2) domain. This type of domain is time varying but each time slice is a Lipschitz domain which does not fit into the rougher setting with uniform base \(\mathcal{O}\). 
In addition, we also use perturbation theory of the parabolic Dirichlet boundary value problem (\refprop{prop:ParabolicPerturbationTheoryForDP}) which is only proved rigorously on parabolic cylinders with Lipschitz base \(\mathcal{O}\). However, since all the tools that were used in these proofs are now also available on rougher domains, the proofs are expected to carry through almost without modification. However, a rigorous proof has not appeared yet. The same observation is expected to be true for \refprop{prop:ParabolicRPimpliesDP}, the main result of \cite{dindos_parabolic_2019}, but this question is also still open. If these two results would be extended, then we would get conclusion \((S)\) for \(\Omega=\mathcal{O}\times\mathbb{R}\) and uniform \(\mathcal{O}\) via the proof that was presented in Section \ref{subsection:UnboundedStructure}.

\begin{rem}[Proof of \refcor{cor:AltMainThm}]
    We framed \refcor{cor:AltMainThm} as alternative version of the (S) conclusion of \refthm{MainTheorem}. Its proof follows exactly the proof of \refthm{MainTheorem}, but because we assume \(\omega_{L_1}^*\in B_{q_1}(\sigma)\), we do not need to use \refprop{prop:ParabolicRPimpliesDP} or \refprop{prop:ParabolicPerturbationTheoryForDP}, the only two results that are not known to hold on domains with uniform base \(\mathcal{O}\). Hence, \refcor{cor:AltMainThm} can be stated on these rougher domains and we don't need to provide a different proof.  
\end{rem}
\smallskip

For the large perturbation result \((L)\) of \refthm{MainTheorem} we also use \refprop{prop:ParabolicRPimpliesDP} or \refprop{prop:ParabolicPerturbationTheoryForDP} as in the small perturbation case above. Additionally, we also use \refprop{prop:ParabolicPR_pimpliesRP_q}, i.e. that
\[(R)_{L_0}^{q_0}\Rightarrow (R)_{L_0}^{p}\qquad\textrm{for all }1<p<q_0,\]
which was proved in \cite{dindos_satterqvist_24}, but only on cylinders with Lipschitz bases \(\mathcal{O}\). The limiting factor in extending this result to rougher domains lies in a gap in understanding the real interpolation spaces on rougher domains.
If these three results could be extended, the proof of conclusion \((L)\) of \refthm{MainTheorem} extends automatically to the case of domains with uniform base \(\mathcal{O}\). In this sense, we can say that the provided proof does not intrinsically need the Lipschitz base \(\mathcal{O}\). We should also note that the notion of a solution would change slightly from an energy solution to a reinforced weak solution that attains the boundary data in a nontangential almost everywhere sense (cf. \cite{dindos_regularity_2023}, see also Section \ref{subsection:EnergySol}). For perturbation questions this more general notion of solutions does not cause any problems, since one of the assumptions is always solvability, so the existence of such a solution for at least one operator.

%\section{Proofs}
%This section collects all the proofs that we spared in the previous section to increase overview over the structure of the proof for \refthm{MainTheorem}.

\subsection{Bound on \(N(v)\) and \(\tilde{N}(\delta|\nabla v|)\) (Proof of \reflemma{Lemma2.10})}
We fix a point \((P,\tau)\in 3\Delta_\rho\) in the boundary and \((X,t)\in I^+\), where \(I^+\in \mathcal{W}(P,\tau)\), and set for \(j\geq 0\) increasing boundary cubes and corresponding Carelson regions as
\begin{align*}
    &\Delta_j\coloneqq \Delta(X^*,t, 2^{j-1}\delta(X,t)),\qquad\textrm{and}\qquad\Omega_j \coloneqq \Omega \cap Q(X^*,t, 2^{j-1}\delta(X,t)),
\end{align*}
with corkscrew points \((X_j,\underline{t_j}):=\underline{A}(\Delta_j)\) and \((X_j,\overline{t_j}):=\bar{A}(\Delta_j)\). Furthermore, let us intorduce annular regions as
\begin{align*}
R_j\coloneqq \Omega_j \setminus (\Omega_{j-1}\cup Q(X,t,\delta(X,t)/2)).
\end{align*}

We consider \(v\) in \((X,t)\) and split 
\begin{align*}
v(X,t)&=\int_\Omega G_1(X,t,Y,s)\Div_Y(h)(Y,s)dYds=\int_\Omega \nabla_Y G_1(X,t,Y,s)h(Y,s)dYds
\\
&=\int_{Q(X,t,\delta(X,t)/2)} \nabla_Y G_1(X,t,Y,s)h(Y,s)dYds 
\\
&\qquad+\int_{\Omega_0} \nabla_Y G_1(X,t,Y,s)h(Y,s)dYds
\\
&\qquad +\sum_{j\geq 1}\int_{R_j} \nabla_YG_1(X,t,Y,s)h(Y,s)dYds,
\\
&=:\tilde{v}(X,t)+v_0(X,t)+\sum_{j\geq 1}v_j(X,t).
\end{align*}
First, we deal with the local part \(\tilde{v}\). Since \(p>\frac{n}{2}+1\) we have as one of the properties of the Green's function (Thm 2.7 in \cite{cho_greens_2008}) for \(q=(2p')\in [1,\frac{n+2}{n+1})\) 
\begin{align*}
    \tilde{v}(X,t)&\leq \Vert \nabla_Y G_1(X,t,\cdot)\Vert_{L^q(Q(X,t,\delta(X,t)/2))}\Big(\int_{Q(X,t,\delta(X,t)/2)}|h|^{2p}dYds\Big)^{\frac{1}{2p}}
    \\
    &\lesssim \delta(X,t)^{-n-1+(n+2)/p}\Big(\int_{Q(X,t,\delta(X,t)/2)}|h|^{2p}dYds\Big)^{\frac{1}{2p}}
    \\
    &= \delta(X,t)\Big(\fint_{Q(X,t,\delta(X,t)/2)}|h|^{2p}dYds\Big)^{\frac{1}{2p}}\leq T_p(h)(P,\tau).
\end{align*}

Next, let us consider \(v_0\). Let \(G^\rho(Y,s):=G_1(\underline{A}(8\Delta_\rho),Y,s)\) be the Green's function of \(L_1\) with pole in \(\underline{A}(8\Delta_\rho)\). We have by Cacciopolli and Harnack inequality
\begin{align*}
    |v_0(X,t)|&\lesssim \sum_{I\in \mathcal{W}(\Delta_0)}\Big(\int_{I^+}|\nabla_Y G_1(X,t,Y,s)|^2dYds\Big)^{1/2}\Big(\int_{I^+}|h(Y,s)|^2dYds\Big)^{1/2}
    \\
    &
    \lesssim \sum_{I\in \mathcal{W}(\Delta_0)}\Big(\frac{1}{l(I)^2}\int_{\tilde{I}^+}|G_1(X,t,Y,s)|^2dYds\Big)^{1/2}\Big(\int_{\tilde{I}^+}|h(Y,s)|^2dYds\Big)^{1/2}
    \\
    &
    \lesssim \sum_{I\subset \mathcal{W}(\Delta_0)}\Big(\frac{1}{l(I)^2}\int_{\tilde{I}^+}|G_1(X_2,\underline{t_2},Y,s)|^2dYds\Big)^{1/2}\Big(\int_{I^+}|h(Y,s)|^2dYds\Big)^{1/2}
    \\
    &
    \lesssim \sum_{I\in \mathcal{W}(\Delta_0)}\Big(\frac{1}{l(I)^2}\int_{\tilde{I}^+}\big|\frac{G_1(X_2,\underline{t_2},Y,s)}{G^\rho(Y,s)}G^\rho(Y,s)\big|^2dYds\Big)^{1/2}\Big(\int_{I^+}|h(Y,s)|^2dYds\Big)^{1/2}.
\end{align*}

Using the Local Comparison Principle, Backwards Harnack inequality and Comparison Principle we obtain
\begin{align*}
    &\int_{\tilde{I}^+}\big|\frac{G(X_2,\underline{t_2},Y,s)}{G^\rho(Y,s)}G^\rho(Y,s)\big|^2dYds\lesssim\int_{\tilde{I}^+}\big|\frac{G(X_2,\underline{t_2},X_0,\overline{t_0})}{G^\rho(X_0,\underline{t_0})}G^\rho(Y,s)\big|^2dYds
    \\
    &\lesssim \int_{\tilde{I}^+}\big|\frac{G(X_2,\underline{t_2},X_0,\underline{t_0})}{G^\rho(X_0,\underline{t_0})}G^\rho(Y,s)\big|^2dYds
    \lesssim \int_{\tilde{I}^+}\big|\frac{(\omega^*)^{(X_2,\underline{t_2})}(\Delta_0)}{\omega^\rho(\Delta_0)}G^\rho(Y,s)\big|^2dYds
    \\
    &\lesssim \frac{1}{\omega^\rho(\Delta_0)^2}\int_{\tilde{I}^+}\frac{\omega^\rho(I)^2}{l(I)^{2n}}dYds,
\end{align*}

and hence in total
\begin{align*}
    |v_0(X,t)|& \lesssim \sum_{I\in \mathcal{W}(\Delta_0)}\frac{1}{\omega^\rho(\Delta_0)}\Big(\frac{1}{l(I)^2}\int_{\tilde{I}^+}\frac{\omega^\rho(I)^2}{l(I)^{2n}}dYds\Big)^{1/2}\Big(\int_{I^+}|h(Y,s)|^2dYds\Big)^{1/2}
    \\
    & \lesssim \frac{1}{\omega^\rho(\Delta_0)}\sum_{I\in \mathcal{W}(\Delta_0)}l(I)\omega^\rho(I)\Big(\fint_{I^+}|h(Y,s)|^2dYds\Big)^{1/2}
    \\
    & \lesssim \frac{1}{\omega^\rho(\Delta_0)}\int_{\Delta_0}\sum_{I\in \mathcal{W}(X,t)}l(I)\Big(\fint_{I^+}|h(Y,s)|^2dYds\Big)^{1/2}d\omega^\rho(X,t)
    \\
    & \lesssim \frac{1}{\omega^\rho(\Delta_0)}\int_{\Delta_0}T_1(h)d\omega^\rho(X,t)
    \\
    &\lesssim M_{\omega^\rho}(T_p(h))(P,\tau).
\end{align*}

Next for \(v_j\), we intend to follow the approach of \(v_0\) with additional decay expected from the fact that these regions are further away from the pole of the Green's function. But first, we split the annular region \(R_j=V_j\cup W_j\) into a region \(V_j\) close to the boundary and a region \(W_j\) away from the boundary. Specifically,
\begin{align*}
    V_j&:=R_j\cap\{(Y,s);\delta(Y,s)<2^{j-3}\delta(X,t)\},\qquad\textrm{and}\qquad
    \\
    W_j&:=R_j\cap\{(Y,s);\delta(Y,s)\geq 2^{j-3}\delta(X,t)\}.
\end{align*}
The calculations for both sets are almost analogous to the case of \(v_0\). However, the switch of pole has to be undertaken more carefully. Note that for \((Y,s)\in I^+\subset V_j\) with Boundary Hölder inequality and Harnack inequality we have
\begin{align}
    G(X,t,Y,s)&\lesssim 2^{-j\alpha}G(X_j,\underline{t_j},Y,s)\lesssim 2^{-j\alpha}G(X_{j+2},\underline{t_{j+2}},Y,s)\label{eq:V_j}
    \\
    &=2^{-j\alpha}\frac{G(X_{j+2},\underline{t_{j+2}},Y,s)}{G^\rho(Y,s)}G^\rho(Y,s).\nonumber
\end{align}
Applying local Comparison Principle, Backwards Harnack inequality, and the Comparison Principle we further obtain
\begin{align*}
    &2^{-j\alpha}\frac{G(X_{j+2},\underline{t_{j+2}},Y,s)}{G^\rho(Y,s)}G^\rho(Y,s)\lesssim 2^{-j\alpha}\frac{G(X_{j+2},\underline{t_{j+2}},X_j,\overline{t_j})}{G^\rho(X_j,\underline{t_j})}G^\rho(Y,s)
    \\
    &\qquad\lesssim 2^{-j\alpha}\frac{G(X_{j+2},\underline{t_{j+2}},X_j,\underline{t_j})}{G^\rho(X_j,\underline{t_j})}G^\rho(Y,s)
    \lesssim 2^{-j\alpha}\frac{(\omega^*)^{(X_{j+2},\underline{t_{j+2}})}(\Delta_j)}{\omega^\rho(\Delta_j)}G^\rho(Y,s) 
    \\
    &\qquad\lesssim 2^{-j\alpha}\frac{1}{\omega^\rho(\Delta_j)}G^\rho(Y,s),
\end{align*}
and hence
\begin{align*}
    &\sum_{I\in \mathcal{W}(\Delta_j), I^+\cap V_j\neq \emptyset}\Big(\int_{I^+}|\nabla_Y G(X,t,Y,s)|^2dYds\Big)^{1/2}\Big(\int_{I^+}|h(Y,s)|^2dYds\Big)^{1/2}
    \\
    &\qquad\lesssim 2^{-j\alpha}\frac{1}{\omega^\rho(\Delta_j)}\int_{\Delta_j}T_1(h)d\omega^\rho(x,t)\lesssim 2^{-j\alpha} M_{\omega^\rho}(T_p(h))(P,\tau).
\end{align*}

For \(W_j\), we can use Harnack's inequality before using Boundary Hölder inequality and obtain for \((Y,s)\in I^+\subset W_j\)
\begin{align*}
    G(X,t,Y,s)&\lesssim G(X,\underline{t_{j+1}},Y,s)\lesssim 2^{-j\alpha}G(X_{j+2},\underline{t_{j+2}},Y,s).
\end{align*}
From here on the proof for \(W_j\) follows verbatim the bound for \(V_j\) from \eqref{eq:V_j} and together we obtain \[v_j(X,t)\lesssim 2^{-j\alpha}M_{\omega^\rho}(T_p(h))(P,\tau).\]

This yields in total
\begin{align*}
    v(X,t)&=\tilde{v}(X,t) + \sum_{j\geq 0} v_j(X,t)\leq M_{\omega^\rho}(T_p(h))(P,\tau) + \sum_{j\geq 0}2^{-j\alpha}M_{\omega^\rho}(T_p(h))(P,\tau)
    \\
    &\lesssim M_{\omega^\rho}(T_p(h))(P,\tau).
\end{align*}

Next let us examine \((P,t)\in \partial\Omega\setminus 3\Delta_\rho\) and \((X,t)\in \Gamma(P,t)\) using the same notation as before. Since the cones lie now away from the support of \(h\) we expect some additional decay. We have \(v_0=\tilde{v}=0\), since \(h\) is not supported in their respective integration domains. But for \(v_j\) we can observe like before that 
\begin{align*}
    v_j(X,t)&\lesssim \sum_{I\in \mathcal{W}(\Delta_j), I^+\cap V_j\neq \emptyset}\Big(\int_{I^+}|\nabla_Y G(X,t,Y,s)|^2dYds\Big)^{1/2}\Big(\int_{I^+}|h(Y,s)|^2dYds\Big)^{1/2}
    \\
    &\qquad\lesssim 2^{-j\alpha}\frac{1}{\omega^\rho(\Delta_j)}\int_{\Delta_j}T_1(h)d\omega^\rho.
\end{align*}
Now this term either vanishes because the support of \(h\) is too far away, so that \(T_1(h)=0\) on \(\Delta_j\), or if \(\Delta_{j+1}\cap 2\Delta_\rho\neq \emptyset\) we can obtain
\begin{align*}
    \frac{1}{\omega^\rho(\Delta_j)}\int_{\Delta_j}T_1(h)d\omega^\rho&
    \lesssim \fint_{\Delta_{j+2}}T_1(h)d\omega^\rho
    \\
    &\lesssim \inf_{(Y,s)\in2\Delta_\rho}M_{\omega^\rho}(T_p(h))(Y,s).
\end{align*}
This completes the pointwise bound for \(N(v)\).

\medskip

To conclude the same bound for \(\tilde{N}(\delta|\nabla v|)\) we proof the pointwise inequality
\begin{align}
    \tilde{N}(\delta|\nabla v|)^2\lesssim N^{1,5/4}(v)^2+ T_1(h)N^{1,5/4}(v) + T_1(h)^2,\label{eq:DeltaN(nablaV)}
\end{align}

which  implies the lemma by above pointwise bound on \(N(v)\). Note here that the above proved bound can also be done for \(N^{1,{5/4}}\) if we just replace every \(I\) by an enlargement \(\tilde{I}\).

To proof \eqref{eq:DeltaN(nablaV)} we use the notation from above and introduce a smooth cut-off function \(\eta\in C^\infty_0(Q(X,t, 5 \delta(X,t)/8))\) with \(|\nabla\eta|\lesssim \frac{1}{\delta(X,t)},|\partial_t\eta|\lesssim \frac{1}{\delta(X,t)^2}\) and \(\eta\equiv 1\) on \(Q(X,t, \delta(X,t)/2)\). We observe that
\begin{align*}
    \fint_{Q(X,t, \delta(X,t)/2)}|\delta\nabla v|^2dYds&\approx \delta(X,t)^2\fint_{Q(X,t, \delta(X,t)/2)}|\nabla v|^2dYds,
\end{align*}
and
\begin{align*}
    &\fint_{Q(X,t, 5\delta(X,t)/8)}|\nabla v|^2\eta^2dYds\lesssim \fint_{Q(X,t, 5\delta(X,t)/8)} A_1^*\nabla v\cdot\nabla v \eta^2 dYds
    \\
    & = \fint_{Q(X,t, 5\delta(X,t)/8)} A_1^*\nabla v^2\eta\cdot\nabla \eta dYds
     + \fint_{Q(X,t, 5\delta(X,t)/8)} A_1^*\nabla v\cdot\nabla (v \eta^2) dYds
    \\
    &=:I_1+I_2.
\end{align*}
By H\"{o}lder's and Young's inequality we have
\begin{align*}
    \delta(X,t)^2I_1&\lesssim \Big(\fint_{Q(X,t, 5 \delta(X,t)/8)}|\delta\nabla v|^2\eta^2 dYds\Big)^{1/2}N^{1,5/4}(v)(P,\tau)
    \\
    &\lesssim \sigma\fint_{Q(X,t, 5 \delta(X,t)/8)}|\delta\nabla v|^2\eta^2 dYds + C_\sigma N^{1,5/4}(v)^2(P,\tau)
\end{align*}
for small \(\sigma>0\), which allows us to hide the first term on the left hand side.

For \(I_2\) we use the PDE \(L_1^*v=\Div_X(h)\) to obtain 
\begin{align*}
    I_2&=-\fint_{Q(X,t, 5\delta(X,t)/8)} v\partial_s(v\eta^2)dYds
     + \fint_{Q(X,t, 5\delta(X,t)/8)} \Div_Y( h) v\eta^2 dYds
     \\
     &=:II_1+II_2.
\end{align*}
For the second term \(\delta(X,t)^2II_2\) we use integration by parts to move the divergence over to \(v\) or \(\eta\) and obtain the bound
\begin{align*}
    &T_1(h)(P,\tau)\Big(\fint_{Q(X,t, 5 \delta(X,t)/8)}|\delta\nabla v|^2\eta^2 dYds\Big)^{1/2} + T_1(h)(P,\tau)N^{1,5/4}(v)(P,\tau)
    \\
    &\lesssim \sigma \Big(\fint_{Q(X,t, 5 \delta(X,t)/8)}|\delta\nabla v|^2\eta^2 dYds\Big)^{1/2} + [C_\sigma T_1(h)^2 + T_1(h)N^{1,5/4}(v)](P,\tau),
\end{align*}
where we again can hide the first term on the left hand side.

Finally, we have also by integration by parts
\begin{align*}
    II_1&=-\fint_{Q(X,t, \delta(X,t)/2)} \partial_s(v^2\eta^2)dYds +\frac{1}{2} \fint_{Q(X,t,5 \delta(X,t)/8)}\partial_s(v^2)\eta^2 dYds
    \\
    &=-\frac{1}{2} \fint_{Q(X,t, 5\delta(X,t)/8)}v^2\partial_s(\eta^2) dYds
    \lesssim \frac{1}{\delta(X,t)^2}N^{1,5/4}(v)^2(P,\tau).
\end{align*}
This finishes the proof of \eqref{eq:DeltaN(nablaV)}.
\qed

\subsection{"Good \(\lambda\)-inequality" (Proof of \reflemma{Lemma2.11})}
Let \(\Delta\) be a Whitney cube of \(\{S(v)>\lambda\}\), i.e. \(l(\Delta)\approx\mathrm{dist}(\Delta,\{S(v)\leq\lambda\})\), and define 
\begin{align}E :=\{(X,t)\in \Delta; S(v)>2\lambda, N(v)+\tilde{N}(\delta|\nabla v|) + T_2(h)\leq \eta\lambda\}.\label{eq:defSetEGoodlambda}\end{align}
We can cover \(\{S(v)>\lambda\}\) by Whitney cubes whose intersection is finite, and hence it is enough to prove \(\sigma(E)\lesssim\gamma\sigma(\Delta)\).

\smallskip
From Lemma 1 in \cite{dahlberg_area_1984} we have that for every \(\tau>0\) there exists a \(\gamma>0\) 
    such that for the truncated square function
\[S_{\tau l(\Delta)}(v)^2(X,t)\coloneq \int_{\Gamma_{^{\tau l(\Delta)}(X,t)}}|\nabla v(Y,s)|^2\delta(Y,s)^{2-n}dYds>\frac{\lambda^2}{4}\]
    holds for all points \((X,t)\in E\), where \(\Gamma^{\tau r}(X,t)\coloneq \Gamma(X,t)\cap Q(X,t,\tau r)\).
We set the sawtooth regions \(\tilde{\Omega}=\bigcup_{(X,t)\in E}\Gamma_\alpha^{\tau l(\Delta)}(X,t)\) and \(\hat{\Omega}=\bigcup_{(X,t)\in E}\Gamma_{2\alpha}^{2\tau l(\Delta)}(X,t)\) 
    and choose  a smooth cut-off \(\psi\) with on \(\tilde{\Omega}\). Specifically, we choose \(\psi\) such that \(\eta\in C_0^\infty(\hat{\Omega}), 0\leq \psi\leq 1\) with \(\psi\equiv 1\) on \(\tilde{\Omega}\) and \(\psi\equiv 0\) outside of  \(\hat{\Omega}\) and \(|\nabla \psi |\delta+| \partial_s \psi|\delta^2\lesssim C\). Furthermore, we split the region between the sawtooth region \(\tilde{\Omega}\) and its enlargement \(\hat{\Omega}\) into 
\begin{align*}
    &E_1:=\hat{\Omega}\setminus\tilde{\Omega}\cap \{(Y,s)\in \Omega;\delta(Y,s)<\tau l(\Delta)\},\qquad\textrm{and}
    \\
    &E_2:=\hat{\Omega}\setminus\tilde{\Omega}\cap \{(Y,s)\in \Omega;\delta(Y,s)\geq\tau l(\Delta)\}.
\end{align*}

Since \(\Delta\) was a Whitney cube of arbitrary size, we have to consider the backwards in time corkscrew point \(\underline{A}(\Delta)\) with respect to this \(\Delta\). We denote the corresponding parabolic measure of \(L_1^*\) to this point by \(\tilde{\omega}=\omega_{L_1^*}^{\underline{A}(\Delta)}\) and the Green's function for \(L_1\) with pole in this point as \(\tilde{G}=G_1(\underline{A}(\Delta),\cdot)\). We want to show that 
\begin{align}\tilde{\omega}(E)\lesssim \eta^2\tilde{\omega}(\Delta).\label{GoodLambdaInequalityomega^*}\end{align} 
If \eqref{GoodLambdaInequalityomega^*} holds, then \(\sigma(E)\lesssim \gamma\sigma(\Delta)\) follows from \(\omega_{L_1^*}\in A_\infty(\sigma)\) (see \eqref{eq:A_inftyProperty}) for a sufficiently small choice of \(\eta\). Hence choosing \(\eta\) small enough finishes the proof.

\smallskip

Now let us proof \eqref{GoodLambdaInequalityomega^*}. We start with the truncated square function, and the Comparison Principle to obtain
\begin{align*}
    \tilde{\omega}(E)&\lesssim \frac{1}{\lambda^2}\int_E S_{\tau l(\Delta)}(v)^2d\tilde{\omega}
    =\frac{1}{\lambda^2}\int_E\int_{\tilde{\Omega}}|\nabla v|^2\delta^{-n-1}\chi_{\Gamma^{\tau l(\Delta)}(X,t)}dYdsd\tilde{\omega}(X,t)
    \\
    &=\frac{1}{\lambda^2}\int_{\tilde{\Omega}}|\nabla v|^2\delta^{-n-1}\Big(\int_E\chi_{\Gamma^{\tau l(\Delta)}(X,t)}d\tilde{\omega}(X,t)\Big)dYds
    \\
    &\lesssim \frac{1}{\lambda^2}\int_{\tilde{\Omega}}|\nabla v|^2\delta^{-n-1}\tilde{\omega}(\Delta(X,t, \delta(X,t)))dYds
    \lesssim \frac{1}{\lambda^2}\int_{\tilde{\Omega}}|\nabla v|^2\tilde{G}dYds.
\end{align*}

By ellipticity, we can bound 
\[\frac{1}{\lambda^2}\int_{\tilde{\Omega}}|\nabla v|^2\tilde{G}dYds\lesssim\frac{1}{\lambda^2}\int_{\hat{\Omega}} A_1^*\nabla v\cdot\nabla v \tilde{G}\psi dYds.\]
The product rule allows us the write \(\nabla v \tilde{G}\psi=-\nabla \tilde{G} v\psi - \nabla \psi \tilde{G}v + \nabla(v\tilde{G}\psi)\) and we can use the PDE for \(v\) on the term with \(\nabla(v\tilde{G}\psi)\) to obtain
\begin{align*}
    &-\frac{1}{2\lambda^2}\int_{\hat{\Omega}} A_1^*\nabla v^2 \cdot\nabla \tilde{G}\psi dYds 
     - \frac{1}{2\lambda^2}\int_{\hat{\Omega}} A_1^*\nabla v^2  \cdot \nabla\psi \tilde{G} dYds
     + \frac{1}{2\lambda^2}\int_{\hat{\Omega}} \partial_t v^2 \tilde{G}\psi dYds
    \\
    &\qquad + \frac{1}{\lambda^2}\int_{\hat{\Omega}} \Div_Y(h) v \tilde{G}\psi dYds =:I_1+I_2+I_3+I_4.
\end{align*}
We group \(I_1\) and \(I_3\) together and use the product rule again to rewrite these expressions so that the PDE for the Green's function can be used on a function that is supported away from its pole. Specifically, we get
\begin{align*}
    I_1+I_3&=-\frac{1}{2\lambda^2}\int_{\hat{\Omega}} A_1\nabla \tilde{G}\cdot \nabla (v^2\psi) dYds + \frac{1}{2\lambda^2}\int_{\hat{\Omega}} A_1\nabla \tilde{G}\cdot \nabla \psi v^2 dYds 
    \\
    & \qquad + \frac{1}{2\lambda^2}\int_{\hat{\Omega}} \tilde{G}\partial_t(v^2\psi) dYds -\frac{1}{2\lambda^2}\int_{\hat{\Omega}} \tilde{G}v^2\partial_t\psi dYds=I'_1+I'_2+I'_3+I'_4.
\end{align*}
The terms \(I'_1+I'_3=0\), since we can use the PDE for the Green's function. The other two terms and \(I_2\) will work very similar and we can note that by using the properties of the cut-off function we obtain
\begin{align*}
I'_2+I'_4+I_2&\lesssim \frac{1}{\lambda^2}\int_{E_1\cup E_2}(|\delta \nabla v| v + v^2)|\tilde{G}| \frac{1}{\delta^2} dYds + \frac{1}{\lambda^2}\int_{E_1\cup E_2}v^2|\nabla\tilde{G}|\frac{1}{\delta} dYds
\\
&\lesssim \eta^2\int_{E_1\cup E_2}|\tilde{G}| \frac{1}{\delta^2} dYds + \eta^2\int_{E_1\cup E_2}|\nabla\tilde{G}|\frac{1}{\delta} dYds=:J_1+J_2.
\end{align*}
In the last line we used that all the expression containing \(v\) or \(\delta\nabla v\) can be bounded by the nontangential maximal function of the respective term, and they are bounded by \(\eta\lambda\) due to the choice of \eqref{eq:defSetEGoodlambda}. To be rigorous here, we would like to point out that the term containing \(\delta|\nabla v|\) can only be bounded by the mean valued nontangential maximal function and we would have a supremum over a Whitney cube on the Green's function. This supremum could be removed by Harnack's inequality under enlarging the sets \(E_1\) and \(E_2\). However, this enlargement does not influence the following argument for \(J_1\) and \(J_2\).
\medskip

Before we are able to calculate \(J_1\) and \(J_2\), we need to make the following observation rigorous: We want to split the sets \(E_1\) and \(E_2\) into Whitney cubes \(I^+\). Each Whitney cube of length \(2^i\) corresponds to at least one boundary cube, which will be called \(\Delta_i\), in the sense that \(l(I^+)\approx 2^i\approx l(\Delta_i)\approx \mathrm{dist}(\Delta_i,I^+)\). Then \(\frac{\tilde{G}}{\delta}(X,t)\approx M[\tilde{k}](Y,s)\) for a point \((X,t)\in I^+\) and \((Y,s)\in \Delta_i\), where \(\tilde{k}=\frac{d\tilde{\omega}}{d\sigma}\). Hence, we can replace integrating \(\frac{\tilde{G}}{\delta}\) over a Whitney cube of \(E_1\) or \(E_2\) by integrating \(M[\tilde{k}]\) over the boundary. Now, let us make this argument rigorous.  

First, we need to introduce new quasi-annuli. We define
\[\Ann_i:=\{(Y,s)\in\partial\Omega; c2^{i-1}\leq\mathrm{dist}_p((Y,s),E)\leq c2^{i}\}\qquad\textrm{ for all }{i\leq N},\]
where \(N\in \mathbb{Z}\) is chosen such that \(\tau l(\Delta)\leq c2^N\leq 2\tau l(\Delta)\).
Next, we cover \(\Ann_i\) by \(N_i\) boundary cubes \((\Delta_i^m)_{1\leq m\leq N_i}\), which have diameter comparable to \(c2^{i}\).
Since \(\tau<1\), we have that \(E\cup\bigcup_{i=-\infty}^N \Ann_i\subset 3\Delta\). Next, we observe that if \(I\in \mathcal{W}(3\Delta)\) with \(l(I)=c2^i\) and \((Y,s)\in I^+\cap E_1\neq \emptyset\) then there exists a vertex of a cone \((P,\tau)\in E\) with
\[\alpha c 2^{i-1}\leq |(Y,s)-(P,\tau)|_p\leq 2\alpha c2^i.\]
Due to the definition of \(\Ann_i\), we know that there exists a \(\Delta_i^m\) such that 
\\ \(\mathrm{dist}((P,\tau),\Delta^m_i)\approx 2^i\). Hence we get combined that
\begin{align}c2^{i-2}\leq\mathrm{dist}(\Delta_i^m,I^+)\approx c2^{i+2}.\label{eq:DistanceDelta^mtoI^+}\end{align}

\smallskip

Hence we showed that for every \(I\in W(5\Delta)\) with \(l(I)=c2^i\) and \(I^+\cap E_1\neq \emptyset\) we have an \(m=m(I^+)\leq N_i\) with \(\mathrm{dist}(\Delta_i^m,I^+)\approx c2^i\). 
Reversely, for every such \(\Delta_i^m\) there is at most a finite number \(\tilde{N}\) of cubes \(I^+\) of above form that satisfy (\ref{eq:DistanceDelta^mtoI^+}), where \(\tilde{N}\) does just depend on the dimension of the ambient space and the geometrical constants of the uniform domain. To see this, we note that the volume of \(\{(Y,s)\in \mathbb{R}^{n+1};\mathrm{dist}(\Delta_i^m,(Y,s))\lesssim 2^i\}\) is comparable to \(2^{i(n+2)}\), and hence can only be covered by at most \(\tilde{N}\) disjoint \(I^+\) that have volume comparable to \(2^{i(n+2)}\).

\smallskip

In total we have found that for every \(I^+\), where \(I\in W(3\Delta)\) with \(l(I)=c2^i\) and \(I^+\cap E_1\neq \emptyset\), there exists a \(\Delta_i^m\) so that \eqref{eq:DistanceDelta^mtoI^+} holds, but there cannot be more than \(\tilde{N}\) of such \(I^+\) for the same \(\Delta_i^m\).

\smallskip

This observation allows the following calcuation for the Green's function evaluated at \((X,t)\in I^+\), where \(l(I)=c2^i\) and \(I^+\cap E_1\neq \emptyset\) and \(\Delta_i^m\) the corresponding cube with \eqref{eq:DistanceDelta^mtoI^+}. We have 
\begin{align*}
    &\mathrm{dist}((X^*,t+4\delta(X,t)^2), \Delta_i^m)
    \\&\qquad\lesssim |(X^*,t+4\delta(X,t)^2)-(X,t)|_p + |(X,t)-(Q,\tau)|_p +\mathrm{dist}((Q,\tau), \Delta_i^m)
    \\
    &\qquad\lesssim 2^i,
\end{align*}
and hence with the doubling property of the parabolic measure and the Comparison Principle for all \((Y,s)\in \Delta_i^m\)
\begin{align}
    \frac{\tilde{G}(X,t)}{\delta(X,t)}&\approx \delta(X,t)^{-n-1}\tilde{\omega}(\Delta(X^*,t+4\delta(X,t),\delta(X,t)))\nonumber
    \\
    &\approx \delta(X,t)^{-n-1}\tilde{\omega}(\Delta(Y,s,5\delta(X,t)))\nonumber
    \\
    &=\fint_{\Delta(Y,s,5\delta(X,t))} \tilde{k}(z,\tau)dzd\tau\nonumber
    \\
    &\leq M[\tilde{k}\chi_{5\Delta}](Y,s).\label{eq:G^*overdeltalesssimM[k]}
\end{align}

We can return to the expressions \(J_1\) and \(J_2\) now. As a consequence of \eqref{eq:G^*overdeltalesssimM[k]} we obtain for the part of \(J_1\) that is integrated over \(E_1\)
\begin{align*}
    &\int_{E_1}\frac{|\tilde{G}|}{\delta^2}dYds
    \lesssim \sum_{I\in W(5\Delta)}\int_{E_1\cap I^+}\frac{|\tilde{G}|}{\delta^2}dYds
    \\
    &\lesssim \sum_{I\in W(5\Delta)}\int_{E_1\cap I^+}\frac{\inf_{(\tilde{Y},\tilde{s})\in \Delta_i^{m(I^+)}}M[\tilde{k}\chi_{5\Delta}](\tilde{Y},\tilde{s})}{\delta}dYds
    \\
    &\lesssim \sum_{I\in W(5\Delta)}l(I)^{n+1}\inf_{(Y,s)\in \Delta_i^{m(I^+)}}M[\tilde{k}\chi_{5\Delta}](Y,s)
    \\
    &\lesssim \sum_{i\leq N}\sum_{I_i\in W(5\Delta)}(2^i)^{n+1}\inf_{(Y,s)\in \Delta_i^{m(I_i^+)}}M[\tilde{k}\chi_{5\Delta}](Y,s)
    \\
    &\lesssim \sum_{i\leq N}\sum_{1\leq m\leq N_i}\tilde{N}(2^i)^{n+1}\inf_{(Y,s)\in \Delta_i^{m}}M[\tilde{k}\chi_{5\Delta}](Y,s)
    \\
    &\lesssim \sum_{i\leq N}\sum_{1\leq m\leq N_i}\int_{\Delta_i^{m}}M[\tilde{k}\chi_{5\Delta}](Y,s)d\sigma(Y,s)
    \\
    &\lesssim \sum_{i\leq N}\int_{\Ann_i}M[\tilde{k}\chi_{5\Delta}](Y,s)d\sigma(Y,s)
    \lesssim \int_{5\Delta}M[\tilde{k}\chi_{5\Delta}](Y,s)d\sigma(Y,s).
\end{align*}

Similarly, but simpler since we do not need to sum over different scales \(i\), we obtain for the integral of \(J_1\) over \(E_2\)
\begin{align*}
    &\int_{E_2}\frac{|\tilde{G}|}{\delta^2}dYds\lesssim \int_{5\Delta}M[\tilde{k}\chi_{5\Delta}](Y,s)d\sigma(Y,s).
\end{align*}
This gives that \(J_1\lesssim \eta^2\int_{5\Delta}M[\tilde{k}\chi_{5\Delta}](Y,s)d\sigma(Y,s)\). Let us show that \(J_2\) and \(I_4\) can be bounded by the same right hand side.
\begin{comment}
First, we have
\begin{align*}
    I_2&\lesssim \frac{2}{\lambda^2}\tilde{N}(\delta|\nabla v|)N(v)\int_{E_1\cup E_2}\frac{\tilde{G}|\nabla \psi|}{\delta}dYds
    \lesssim \eta^2\int_{E_1\cup E_2}\frac{\tilde{G}}{\delta^2}dYds=\eta^2(J_1+J_2).
\end{align*}
The term \(I_3\) works very similarly. However, we do need to use Cacciopolli inequality first to get back from \(\nabla \tilde{G}\) to \(\tilde{G}\), before we can apply \eqref{eq:G^*overdeltalesssimM[k]}.
\end{comment}
For \(J_2\) we proceed similar after using Cacciopolli inequality. We have

\begin{align*}
    J_2&\lesssim \eta^2\sum_{I\in W(\Delta)}\int_{(E_1\cup E_2)\cap I^+}\frac{|\nabla \tilde{G}|}{\delta}dYds
    \\
    &\lesssim \eta^2\sum_{I\in W(\Delta)}\frac{|(E_1\cup E_2)\cap I^+|^{1/2}}{l(I)}\big(\int_{(E_1\cup E_2)\cap I^+}|\nabla \tilde{G}|^2dYds\big)^{1/2}
    \\
    &\lesssim \eta^2\sum_{I\in W(\Delta)}\frac{|(E_1\cup E_2)\cap I^+|^{1/2}}{l(I)}\big(\int_{(E_1\cup E_2)\cap \tilde{I}^+}\big|\frac{\tilde{G}}{\delta}\big|^2dYds\big)^{1/2}
    \\
    &\lesssim\eta^2\sum_{I\in W(5\Delta)}\frac{|(E_1\cup E_2)\cap I^+|^{1/2}}{l(I)}\Big(\int_{(E_1\cup E_2)\cap \tilde{I}^+}|\inf_{(\tilde{Y},\tilde{s})\in \Delta_i^{m(I^+)}}M[\tilde{k}\chi_{5\Delta}](\tilde{Y},\tilde{s})|^2dYds\Big)^{1/2}
    \\
    &\lesssim \eta^2\sum_{I\in W(5\Delta), I^+\cap (E_1\cup E_2)\neq \emptyset}l(I)^{n+1}\inf_{(Y,s)\in \Delta_i^{m(I^+)}}M[\tilde{k}\chi_{5\Delta}](Y,s)
    \\
    &\lesssim \eta^2\int_{5\Delta}M[\tilde{k}\chi_{5\Delta}](Y,s)d\sigma(Y,s),
\end{align*}
where the last inequality follows as for \(J_1\).

\smallskip

Lastly, for \(I_4\), we integrate by parts and get by converting a full integral back into a double integral over the boundary and cones
\begin{align*}
    -I_4&\lesssim \frac{1}{\lambda^2}\int_{\hat{\Omega}} h \cdot\nabla v \tilde{G}\psi + h \cdot\nabla \tilde{G} v \psi + h \cdot\nabla\psi v \tilde{G} dYds
    \\
    &\lesssim\frac{1}{\lambda^2}\int_{E}\int_{\Gamma^{\tau l(\Delta)}(P,\tau)} \delta^{-n-1}h \cdot\nabla v \tilde{G}\psi + \delta^{-n-1}h \cdot\nabla \tilde{G} v \psi + \delta^{-n-1}h \cdot\nabla\psi v \tilde{G} dYds d\sigma(P,\tau)
    \\
    &=:H_1+H_2+H_3.
\end{align*}

For the first term, we use Cauchy-Schwarz inequality, \eqref{eq:G^*overdeltalesssimM[k]}, and properties of the set \(E\) in \eqref{eq:defSetEGoodlambda} to get
\begin{align*}
    H_1&\lesssim \frac{1}{\lambda^2} \int_E\tilde{N}(\delta|\nabla v|)\sum_{\substack{I\in W(5\Delta),\\ I^+\cap \Gamma^{\tau l(\Delta)}(Q,\tau)\neq \emptyset}}\frac{\sup_{(Y,s)\in I^+}\tilde{G}(Y,s)}{l(I)}l(I)\Big(\fint_{I^+}|h|^2dYds\Big)^{1/2} d\sigma(P,\tau)
    \\
    &\lesssim \frac{1}{\lambda^2}\int_E M[\tilde{k}\chi_{5\Delta}]T_
    2(h)\tilde{N}(\delta|\nabla v|) d\sigma(P,\tau)\lesssim \eta^2\int_{3\Delta} M[\tilde{k}\chi_{3\Delta}] d\sigma(P,\tau).
\end{align*}
For \(H_2\), we first apply Cacciopolli inequality and proceed then similarly to \(H_1\):
\begin{align*}
    H_2&\lesssim \frac{1}{\lambda^2} \int_E\sum_{\substack{I\in W(5\Delta),\\ I^+\cap \Gamma^{\tau l(\Delta)}(P,\tau)\neq \emptyset}}l(I)\Big(\fint_{I^+}|h|^2dYds\Big)^{1/2}\Big(\fint_{I^+}|\nabla \tilde{G}|^2dYds\Big)^{1/2}N(v)d\sigma(P,\tau)
    \\
    &\lesssim \frac{1}{\lambda^2} \int_E\sum_{\substack{I\in W(5\Delta),\\ I^+\cap \Gamma^{\tau l(\Delta)}(P,\tau)\neq \emptyset}}l(I)\Big(\fint_{I^+}|h|^2dYds\Big)^{1/2}\Big(\fint_{I^+}\frac{|\tilde{G}|^2}{l(I)^2}dYds\Big)^{1/2}N(v)d\sigma(P,\tau)
    \\
    &\lesssim \frac{1}{\lambda^2}\int_E T_2(h)M[\tilde{k}\chi_{3\Delta}]N(v) d\sigma(P,\tau)\lesssim \eta^2\int_{5\Delta} M[\tilde{k}\chi_{5\Delta}] d\sigma(P,\tau).
\end{align*}
The last term \(H_3\) is completely analogous to \(H_1\) with \(N(v)\) replacing \(\tilde{N}(\delta|\nabla v|)\).
\begin{comment}
Analogously to \(H_1\), we have for the last term
\begin{align*}
    H_3&\lesssim \frac{1}{\lambda^2} \int_E\sum_{\substack{I\in W(5\Delta),\\ I^+\cap \Gamma^{\tau l(\Delta)}(P,\tau)\neq \emptyset}}\frac{\sup_{(Y,s)\in I^+}\tilde{G}(Y,s)}{l(I)}l(I)\Big(\fint_{I^+}|h|^2dYds\Big)^{1/2}\tilde{N}(v)d\sigma(P,\tau)
    \\
    &\lesssim \frac{1}{\lambda^2}\int_E M[\tilde{k}\chi_{3\Delta}]T_2(h)\tilde{N}(v) d\sigma(P,t)\lesssim \eta^2\int_{3\Delta} M[\tilde{k}\chi_{3\Delta}] d\sigma(P,\tau).
\end{align*}
\end{comment}
\smallskip

Hence it remains to bound \(\int_{3\Delta} M[\tilde{k}\chi_{3\Delta}] d\sigma(P,\tau)\). Since \(\omega^*\in B_{q_1}(d\sigma)\) for some \({q_1}>1\) we obtain by \eqref{eq:BoundMaxRevHölder} the reverse Hölder inequality that defines the reverse Hölder space \(B_{q_1}(\sigma)\) and the doubling property of the parabolic measure
\begin{align*}
    &\int_{3\Delta} M[\tilde{k}\chi_{3\Delta}] d\sigma(x,t)\lesssim \sigma(\Delta)\Big(\fint_{3\Delta} M[\tilde{k}\chi_{3\Delta}]^{q_1} d\sigma(x,t)\Big)^{1/{q_1}}
    \\
    &\quad\lesssim \sigma(\Delta)\Big(\fint_{3\Delta} (\tilde{k})^{q_1} d\sigma(x,t)\Big)^{1/{q_1}}
    \lesssim \sigma(\Delta)\fint_{3\Delta} \tilde{k} d\sigma(x,t) = \tilde{\omega}(3\Delta)\lesssim \tilde{\omega}(\Delta).
\end{align*}
\qed

%\subsection{The difference function \(F\) (Proof of \reflemma{lemma:defF})}
%\input{ProofFrigorous}
\subsection{Bound on \(\tilde{N}_{1,\rho}(\nabla F)\) (Proof of \reflemma{lemma:defv})}\label{subsection:BoundNF}
The proof consists of two parts. First, we need to show 
\begin{align}\Vert \tilde{N}_{1,\rho}(\nabla F)\Vert_{L^q(\partial\Omega)}\lesssim\int_\Omega \EPS(Y,s) \nabla_Y u_1(Y,s)\nabla_Y v(Y,s) dYds,\label{eq:proofLEmma2.19/1}\end{align}
before applying a stopping time argument to prove that
\begin{align}\int_\Omega \EPS(Y,s) \nabla_Y u_1(Y,s)\nabla_Y v(Y,s) dYds\lesssim \Vert \mu\Vert_{\mathcal{C}} \Vert \tilde{N}_{2}(\nabla u_1)\Vert_{L^q(\partial\Omega)}\Vert S(v)\Vert_{L^{q'}(\partial\Omega)}.\label{eq:proofLEmma2.19/2}\end{align}

The difficulty in the first part lies in finding appropriate approximations such that all the calculations can be justified. Since this is fairly standard, we would like to refer for all the details to the proof of Lemma 4.2.5 in \cite{ulmer_l_2024}.

\medskip

For the second part we need to distinguish between the two different assumptions \eqref{eq:CarlesonWithSupNorm} and \eqref{eq:CarlesonWithoutSupNorm}.
First we show \eqref{eq:proofLEmma2.19/2} assuming \eqref{eq:CarlesonWithSupNorm}. We set 
\[\mathcal{O}_j:=\{(P,t)\in\partial\Omega; \tilde{N}(\nabla u)S(v)(P,t)>2^j\}\]
as a super level set of \(\tilde{N}(\nabla u)S(v)\) and its enlargement as
\(\tilde{\mathcal{O}}_j:=\{M(\chi_{\mathcal{O}_j})>\frac{1}{2}\}\). Note that \(\sigma(\tilde{\mathcal{O}}_j)\approx\sigma(\mathcal{O}_j)\). Next, we introduce a decomposition of the domain into
\[F_j:=\bigg\{(X,t)\in\Omega; 
\begin{matrix}
    |\Delta(X^*,t,\delta(X,t))\cap\mathcal{O}_j|&> \frac{1}{2}|\Delta(X^*,t,\delta(X,t))|,
    \\
    |\Delta(X^*,t,\delta(X,t))\cap\mathcal{O}_{j+1}|&\leq \frac{1}{2}|\Delta(X^*,t,\delta(X,t))|\end{matrix}\bigg\}.
\]

We define \(\tilde{\Omega}(\tilde{\mathcal{O}}_j):=\{(X,t)\in\Omega; \Delta(X^*,t,\delta(X,t))\subset\tilde{\mathcal{O}}_j\}\) as the tent over \(\tilde{\mathcal{O}}_j\) and observe that due to construction \(F_j\subset \tilde{\Omega}(\tilde{\mathcal{O}}_j)\).
\smallskip

Now we apply the stopping time argument. We have 
\begin{align*}
 &\int_\Omega \EPS \nabla u_1\nabla v dYds\lesssim \sum_j \int_{\Omega\cap F_j} |\EPS\nabla u_1\nabla v| dYds
 \\
 &\qquad\leq \sum_j \int_{\tilde{\mathcal{O}}_j\setminus\mathcal{O}_j}\int_{\Gamma(P,t)\cap F_j} |\EPS|| \nabla u_1||\nabla v| \delta^{-n-1} dYds dPdt
 \\
 &\qquad\leq \sum_j \int_{\tilde{\mathcal{O}}_j\setminus\mathcal{O}_j}\Big(\int_{\Gamma(P,t)\cap F_j} |\nabla v|^2 \delta^{-n} dYds\Big)^{1/2}\Big(\int_{\Gamma(P,t)\cap F_j} |\nabla u_1|^2|\EPS|^2 \delta^{-n-2} dYds\Big)^{1/2} dPdt,
\end{align*}
where we applied Cauchy-Schwarz in the last line. For the integral involving \(u_1\) continue with
\begin{align}
    &\int_{\Gamma(P,t)\cap F_j} |\nabla u_1|^2|\EPS|^2 \delta^{-n-2} dYds\leq \sum_{I\in \mathcal{W}(P,t)}\int_{I^+\cap F_j} |\nabla u_1|^2|\EPS|^2 \delta^{-n-2} dYds\nonumber%\label{eq:BeginOfNinsteadofTildeN}
    \\
    &\qquad\leq \sum_{I\in \mathcal{W}(P,t), I^+\cap F_j\neq \emptyset}\sup_{I^+\cap F_j}|\EPS|^2\fint_{I^+} |\nabla u_1|^2  dYds\label{eq:BeginOfNinsteadofTildeN}
    \\
    &\qquad\leq \tilde{N}_2(\nabla u_1)^2(P,\tau)\sum_{I\in \mathcal{W}(P,t)}\sup_{I^+\cap F_j}|\EPS|^2\nonumber
\end{align}
Since 
\[\sum_{I\in \mathcal{W}(P,t)}\sup_{I^+\cap F_j}|\EPS|^2\lesssim \int_{\Gamma_2(P,t)\cap \tilde{\Omega}(\tilde{\mathcal{O}}_{j-1})}\sup_{B(Y,s,\delta(Y,s)/2)}|\EPS|^2\delta^{-n-2}dYds,\]
we get by applying the Carleson measure property of \(\mu\) and Cauchy-Schwarz inequality
\begin{align}
    &\sum_j \int_{\tilde{\mathcal{O}}_j\setminus\mathcal{O}_j}\Big(\int_{\Gamma(P,t)\cap F_j} |\nabla v|^2 \delta^{-n} dYds\Big)^{1/2}\Big(\int_{\Gamma(P,t)\cap F_{j-1}} |\nabla u_1|^2|\EPS|^2 \delta^{-n-2} dYds\Big)^{1/2} dPdt\nonumber
    \\
    &\lesssim\sum_j \int_{\tilde{\mathcal{O}}_j\setminus\mathcal{O}_j}\tilde{N}_2(\nabla u_1)S(v)(P,t)\Big(\int_{\Gamma_2(P,t)\cap \tilde{\Omega}(\tilde{\mathcal{O}}_{j-1})}\sup_{B(Y,s,\delta(Y,s)/2)}|\EPS|^2 \delta^{-n-2} dYds\Big)^{1/2} dPdt\nonumber
    \\
    &\lesssim\sum_j 2^j|\mathcal{O}_j|^{1/2}\Big(\int_{\tilde{\mathcal{O}}_j}\int_{\Gamma_2(P,t)\cap \tilde{\Omega}(\tilde{\mathcal{O}}_j)}\sup_{B(Y,s,\delta(Y,s)/2)}|\EPS|^2 \delta^{-n-2} dYds dPdt\Big)^{1/2}\nonumber
    \\
    &\lesssim\sum_j 2^j|\mathcal{O}_j|^{1/2}\Big(\int_{\tilde{\Omega}(2\tilde{\mathcal{O}}_j)}\sup_{B(Y,s,\delta(Y,s)/2)}|\EPS|^2 \delta^{-1} dYds dPdt\Big)^{1/2}\label{eq:ReplaceNbyTildeN}
    \\
    &\lesssim\Vert \mu\Vert_{\mathcal{C}}\sum_j 2^j|\mathcal{O}_j|\lesssim \Vert \mu\Vert_{\mathcal{C}}\int_{\partial\Omega}\tilde{N}_2(\nabla u_1)S(v)dPdt\lesssim \Vert \mu\Vert_{\mathcal{C}}\Vert \tilde{N}_2(\nabla u_1)\Vert_{L^q(\partial\Omega)}\Vert S(v)\Vert_{L^{q'}(\partial\Omega)}.\nonumber
\end{align}

%=============================================================

Let us now turn to the proof of \eqref{eq:proofLEmma2.19/2} assuming \eqref{eq:CarlesonWithoutSupNorm} instead.
If we assume \(\delta|\nabla A|\leq M\) we can follow an idea by Joseph Feneuil and Svitlana Mayboroda. Inspired by Lemma 2.15 in \cite{feneuil_green_2023}, which gives the pointwise bound \(|\nabla u_1(X,t)|\lesssim \frac{u_1(X,t)}{\delta(X,t)}\) for \((X,t)\in \Omega\) for an elliptic operator, we prove a parabolic version thereof in \reflemma{lemma:nablaUpointwiseBound}. The proof is provided in the appendix. Using this lemma, the only difference in the proof of \eqref{eq:proofLEmma2.19/2} lies in the estimate \eqref{eq:BeginOfNinsteadofTildeN}. Instead of putting the \(L^\infty\) norm on the \(\EPS\), we can now put it on \(\nabla u_1\) and hence replace \eqref{eq:BeginOfNinsteadofTildeN} by 
\begin{align*}
&\sum_{I\in \mathcal{W}(P,t), I^+\cap F_j\neq \emptyset}\sup_{I^+}|\nabla u_1|^2 dYds\fint_{I^+\cap F_j} |\EPS|^2 \delta^{-n-2} dYds
\\
&\lesssim \sum_{I\in \mathcal{W}(P,t), I^+\cap F_j\neq \emptyset}\fint_{\tilde{I}^+}|\nabla u_1|^2 dYds\fint_{I^+\cap F_j} |\EPS|^2 \delta^{-n-2} dYds.
\end{align*}
As a consequence, we get the Carleson norm of \(d\mu(X,t)=\frac{|\EPS(X,t)|^2}{\delta(X,t)}\) in \eqref{eq:ReplaceNbyTildeN}, which yields \eqref{eq:proofLEmma2.19/2}. The rest of the proof work completely analogously.

\subsection{Proof of the bound on \(S(v)\) (Proof of \refcor{cor:S(v)lesssimT_p(h)})}
The good \(\lambda\)-inequality in \reflemma{Lemma2.11} allows to conclude via by layer cake representation of the integral
    \begin{align*}
        \int_{\partial\Omega} S(v)^r d\sigma&= \gamma\int_{\partial\Omega} S(v)^r d\sigma +C_\eta  \int_{\partial\Omega} N(v)^r d\sigma +  C_\eta\int_{\partial\Omega} \tilde{N}(\delta|\nabla v|)^r d\sigma
        \\
        &\qquad + C_\eta  \int_{\partial\Omega} T_2(h)^r d\sigma.
    \end{align*}
    Since the first term can be hidden on the left hand side for a small choice of \(\gamma\), we obtain
    \begin{align*}
        \int_{\partial\Omega} S(v)^r d\sigma&\lesssim \int_{\partial\mathcal{O}\times\mathbb{R}} N(v)^r d\sigma + \int_{\partial\mathcal{O}\times\mathbb{R}} \tilde{N}(\delta|\nabla v|)^r d\sigma + \int_{8\Delta_\rho} T_2(h)^r d\sigma.
    \end{align*}

    Now we would like to make use of the pointwise bounds on the nontangential maximal functions from \reflemma{Lemma2.10}, and claim that
    \begin{align}
        \int_{\partial\mathcal{O}\times\mathbb{R}} N(v)^r d\sigma + \int_{\partial\mathcal{O}\times\mathbb{R}} \tilde{N}(\delta|\nabla v|)^r d\sigma\lesssim \int_{8\Delta_\rho} M_{\omega^\rho}(T_p(h))^r d\sigma,\label{eq:CorN(v)lesssimT_p}
    \end{align}
    which completes the proof. However, the restriction of the integration domain on the right hand side makes this is a nontrivial statement.
    
    \medskip
    To show \eqref{eq:CorN(v)lesssimT_p}, we observe that \(\tilde{N}(\delta|\nabla v|)\) is pointwise bounded by \(N(v)\) like shown in the proof of \reflemma{Lemma2.10} and hence it is enough to show the bound
    \begin{align}\int_{\partial\mathcal{O}\times\mathbb{R}} N(v)^r d\sigma\lesssim \int_{8\Delta_\rho} M_{\omega^\rho}(T_p(h))^r d\sigma.\label{eq:N(v)boundedbyDelta8}\end{align}

    We start by defining parabolic cylinder subdomains. Set a region around the support of \(h\), a flattened Carleson box, as
    \[\tilde{U}:=T(6\Delta_\rho)\setminus\{(X,t)\in \Omega; \delta(X,t)\geq \frac{4}{\rho}\},\]
    and divide \(\Omega\) in \(\tilde{U}\) and the three time pieces
    \[U_0:=\mathcal{O}\times(\frac{16}{\rho^2},\infty), U_1:=\mathcal{O}\times[-\frac{16}{\rho^2},\frac{16}{\rho^2})\setminus \tilde{U},\textrm{ and }U_2:=\mathcal{O}\times(-\infty,-\frac{16}{\rho^2}).\] 
   
    Now, we define
    \begin{align}w(X,t):=\frac{\inf_{2\Delta_\rho}M_{\omega^\rho}(T_p(h))}{\rho^n}G^*_1(\bar{A}(8\Delta_\rho), X,t),\label{eq:defOfuForMaxPrinc}\end{align}
    and claim that 
    \begin{align}v(X,t)\lesssim w(X,t)\textrm{ for all }(X,t)\in U_1\cup U_2.\label{eq:Claimvlesssimu}\end{align}
    Since \(v\equiv 0\) on \(U_0\), we do not incorporate this case into the claim. The advantage of replacing estimating \(N(v)\) by \(N(w)\) in \eqref{eq:N(v)boundedbyDelta8} is that \(w\) is basically a Green's function with explicit growth bounds which we will utilize. Let us postpone the proof of the claim until the end.

    \smallskip
    We begin by partitioning \(\partial\Omega=\partial\mathcal{O}\times\mathbb{R}\). Therefore, we use \((X,t)=\bar{A}(8\Delta_\rho)\) and its projection \(X^*\in \partial\mathcal{O}\) into the boundary of the spatial domain. We define a decomposition of \(\partial\mathcal{O}\) with respect to this point \(X^*\). Set
    \(\Ann^l:=\partial\mathcal{O}\cap (Q(X,\frac{1}{\rho}(l+1))\setminus Q(X,\frac{1}{\rho}l))\) for \(l\geq 1\), and \(\Ann^0:=\partial\mathcal{O}\cap Q(X,\frac{1}{\rho}l)\). Due to \(\mathcal{O}\) being a uniform domain we know that \(\sigma(\Ann^l)\approx \big(\frac{1}{\rho}\big)^{n-1}((l+1)^{n-1}-l^{n-1})\approx \big(\frac{1}{\rho}\big)^{n-1}l^{n-2}\). 

    Next, we can construct in time translated sets
    \[\Delta_k^l:=\Ann^l\times (t+\frac{1}{\rho^2}k, t+\frac{1}{\rho^2}(k+1)) \qquad \textrm{ for }k\in \mathbb{Z},\]
    which have volume \(|\Delta_k^l|\approx \big(\frac{1}{\rho}\big)^{n+1}l^{n-2}\). This gives a decomposition of \(\partial\Omega\), i.e.
    \[\partial\Omega=\bigcup_{k\in \mathbb{Z},l\geq 0}\Delta_k^l.\]
    Now, we are going to bound \(N(v)\) separately on each \(\Delta_k^l\) by two competing bounds \eqref{eq:BoundonN(v)inProof1} and \eqref{eq:BoundonN(v)inProof2}. First, assume that \(k\geq 0\) and that \((P,\tau)\in \Delta_k^l\setminus 8\Delta_\rho\). For a sufficiently small aperture less than \(\frac{1}{3}\) (which we call \(\tilde{c}\) here to avoid confusion with the \(\alpha\) from Boundary Hölder) we have that \(\Gamma_{\tilde{c}}(P,\tau)\subset U_0\cup U_1\cup U_2\), and hence \(N(v)\lesssim N(w)\). Furthermore, let us define \(\bar{c}:=\sqrt{(1+\tilde{c})^2-1}\). Since \(v(Y,s)=0\) for \(s>t-\frac{4}{\rho^2}\), we obtain by the growth estimate of the Green's function (\reflemma{lemma:ParabolicGrwothEstimateForGwithalpha})
    \begin{align*}
        N(v)(P,\tau)&\lesssim \sup_{(Y,s)\in \Gamma_{\tilde{c}}(P,\tau)\cap\{s\leq \frac{4}{\rho^2}\}}\frac{\inf_{2\Delta_\rho}M_{\omega^\rho}(T_p(h))}{\rho^n}\delta(X,t)^\alpha\Big(|X-Y|^2+(t-s)\Big)^{-\frac{n+\alpha}{2}}
        \\
        &\lesssim \sup_{(Y,s)\in \Gamma_{\tilde{c}}(P,\tau)}\frac{\inf_{2\Delta_\rho}M_{\omega^\rho}(T_p(h))}{\rho^n}\big(\frac{1}{\rho}\big)^{\alpha}
        \\
        &\qquad\cdot\Big(\Big|\sqrt{\big(\frac{l}{\rho}-\bar{c}\delta(Y,s)\big)^2+\delta(Y,s)^2}\Big|^2+\Big|\frac{1+k}{\rho^2}-\bar{c}^2\delta(Y,s)^2\Big|\Big)^{-\frac{n+\alpha}{2}}
        \\
        &\lesssim \sup_{\delta\geq 0}\inf_{2\Delta_\rho}M_{\omega^\rho}(T_p(h))\Big((l-\bar{c}\delta)^2+\delta^2+|(k+1)-\bar{c}^2\delta^2|\Big)^{-\frac{n+\alpha}{2}}
        \\
        &\lesssim \sup_{\delta\geq 0}\inf_{2\Delta_\rho}M_{\omega^\rho}(T_p(h))\Big(|l|^2+(k+1)+\delta^2\Big)^{-\frac{n+\alpha}{2}}.
    \end{align*}
    Hence, we obtain
    \begin{align}
        N(v)(P,\tau)\lesssim \inf_{2\Delta_\rho}M_{\omega^\rho}(T_p(h))\Big(|l|^2+(k+1)\Big)^{-\frac{n+\alpha}{2}}.\label{eq:BoundonN(v)inProof1}
    \end{align}
    If we have \((P,\tau)\in \Delta_k^l\) for a negative time step \(k<0\), then we obtain the same with \(|k|\) replacing \(k\). 
\smallskip
    
    Similarly, let us obtain a different bound on \(\Delta_k^l\) for \(k\in\mathbb{Z}, l\geq 0\). Using the other growth estimate of the Green's function (\reflemma{lemma:ParabolicGrowthEstimateForG}) we have
    \begin{align*}
        &N(v)(P,\tau)\lesssim \sup_{(Y,s)\in \Gamma_{\tilde{c}}(P,\tau)}\frac{\inf_{2\Delta_\rho}M_{\omega^\rho}(T_p(h))}{\rho^n}(t-s)^{-\frac{n}{2}}e^{-C\frac{|X-Y|^2}{t-s}}
        \\
        &\quad\lesssim \sup_{(Y,s)\in \Gamma_{\tilde{c}}(P,\tau)\cap\mathcal{O}\times(-\infty, \frac{16}{\rho^2}]}\frac{\inf_{2\Delta_\rho}M_{\omega^\rho}(T_p(h))}{\rho^n}\big(\frac{1}{\rho}\big)^{n}(\rho^{-2}(k+1)-\bar{c}^2\delta(Y,s)^2)^{-\frac{n}{2}}
        \\
        &\hspace{80mm}\cdot e^{-C\frac{(l-\bar{c}\rho\delta(Y,s))^2+\rho^2\delta(Y,s)^2}{k+1-\bar{c}^2\rho^2\delta(Y,s)^2}}
        \\
        &\quad\lesssim \sup_{\frac{\sqrt{k}}{\bar{c}}\geq \delta\geq 0}\inf_{2\Delta_\rho}M_{\omega^\rho}(T_p(h))(k+1-\bar{c}^2\delta^2)^{-\frac{n}{2}}e^{-C\frac{l^2+\delta^2}{k+1-\bar{c}^2\delta^2}}
        \\
        &\quad\lesssim \sup_{\frac{\sqrt{k}}{\bar{c}}\geq \delta\geq 0}\inf_{2\Delta_\rho}M_{\omega^\rho}(T_p(h))(k+1-\bar{c}^2\delta^2)^{-\frac{n}{2}}e^{-C\frac{l^2}{k+1-\bar{c}^2\delta^2}}
    \end{align*}

    Since the function \(u\mapsto u^{-n/2}e^{-c\frac{l^2}{u}}\) for \(u\geq 0\) is maximized in \(u=\sqrt{c}l^2\), we have that for \(l>\frac{1}{\sqrt{C}}\sqrt{k+1}\)
    \[(k+1-\bar{c}^2\delta^2)^{-\frac{n}{2}}e^{-C\frac{l^2}{k+1-\bar{c}^2\delta^2}}\lesssim (k+1)^{-\frac{n}{2}}e^{-C\frac{l^2}{k+1}}.\]
\begin{comment}
    If we consider \(f_{k,l}(\delta):=(k+1-\bar{c}^2\delta^2)^{-\frac{n}{2}}e^{-C\frac{l^2+\delta^2}{k+1-\bar{c}^2\delta^2}}\) as a function in \(\delta\), we can see that its first derivative
    \[f_{k,l}^\prime(\delta)=\delta(k+1-\bar{c}^2\delta^2)^{-\frac{n}{2}-2}e^{-C\frac{l^2+(1-c\bar{c}^2)\delta^2}{k+1-\bar{c}^2\delta^2}}\big(2C(\bar{c}^2l^2+(1-c\bar{c}^2)k) + \bar{c}^2n(\bar{c}^2\delta^2-k)\big)\]
    \[f_{k,l}^\prime(\delta)=\delta(k+1-\bar{c}^2\delta^2)^{-\frac{n}{2}-2}e^{-C\frac{l^2+(1-c\bar{c}^2)\delta^2}{k+1-\bar{c}^2\delta^2}}\big(2(1+\bar{c}l^2+k)-\bar{c}n(k+1-\bar{c}\delta^2)\big)\]
    is strictly negative on \([0,\frac{k}{\bar{c}}]\) if 
    \[l>\sqrt{\frac{\bar{c}^2n-2C(1-c\bar{c}^2)}{2C\bar{c}}}\sqrt{k}=:C_{\tilde{c}}\sqrt{k}.\]
    \[l>\sqrt{\frac{\bar{c}n-2}{2}}\sqrt{k}=:C_{\tilde{c}}\sqrt{k}.\]
\end{comment}
    Hence there exists \(\tilde{C}>0\), such that if \((P,\tau)\in \Delta_k^l\) with \(l>\tilde{C}\sqrt{k}\), we also obtain
    \begin{align}
        N(v)(P,\tau)\lesssim \inf_{2\Delta_\rho}M_{\omega^\rho}(T_p(h))(k+1)^{-\frac{n}{2}}e^{-C\frac{l^2}{k+1}}.\label{eq:BoundonN(v)inProof2}
    \end{align}

    Now, we combine the two different estimates \eqref{eq:BoundonN(v)inProof1} and \eqref{eq:BoundonN(v)inProof2}.
    We take \(\vartheta>0\) to be a small constant chosen later and have
    \begin{align*}
        &\int_{\partial\mathcal{O}\times\mathbb{R}} N(v)^r d\sigma\leq \int_{8\Delta_\rho} N(v)^r d\sigma + \sum_{k\in \mathbb{Z}, l\geq 0}\int_{\Delta_k^l}N(v)^r d\sigma
        \\
        &\qquad\leq\int_{8\Delta_\rho} N(v)^r d\sigma + \sum_{k\in \mathbb{Z}}\Big(\sum_{l\leq \tilde{C}k^{\frac{1}{2}+\vartheta}}\int_{\Delta_k^l}N(v)^r d\sigma + \sum_{l>\tilde{C}k^{\frac{1}{2}+\vartheta}}\int_{\Delta_k^l}N(v)^r d\sigma\Big).
    \end{align*}

    First, let us estimate the sum over small \(l\) by using \eqref{eq:BoundonN(v)inProof1} and observing that
    \begin{align*}
        &\sum_{l\leq \tilde{C}k^{\frac{1}{2}+\vartheta}}|\Delta_k^l|\Big(|l|^2+|k|+1\Big)^{-r\frac{n+\alpha}{2}}
        \lesssim \Big(\frac{1}{\rho}\Big)^{n+1}\sum_{l\leq \tilde{C}k^{\frac{1}{2}+\vartheta}}l^{n-2}\Big(|l|^2+|k|+1\Big)^{-r\frac{n+\alpha}{2}}
        \\
        &\qquad\lesssim \Big(\frac{1}{\rho}\Big)^{n+1}\sum_{l\leq \tilde{C}k^{\frac{1}{2}+\vartheta}}(k^{\frac{1}{2}+\vartheta})^{n-2}\Big(|k|+1\Big)^{-r\frac{n+\alpha}{2}}
        \lesssim \Big(\frac{1}{\rho}\Big)^{n+1} k^{-\frac{n}{2}(r-1)-r\frac{\alpha}{2}-1 +(n-1)\vartheta}.
    \end{align*}
    Choosing \(\vartheta\) sufficiently small only depending on \(n\), such that \(\frac{\alpha}{2}>(n-1)\vartheta\), we obtain 
    \[k^{-\frac{n}{2}(r-1)-r\frac{\alpha}{2}-1 +(n-1)\vartheta}\leq k^{-1-\frac{n}{2}(r-1)}.\]
    Hence, we obtain for such a choice of \(\vartheta\)
    \begin{align*}
        &\sum_{l\leq \tilde{C}k^{\frac{1}{2}+\vartheta}}\int_{\Delta_k^l}N(v)^r d\sigma\lesssim\Big(\frac{1}{\rho}\Big)^{n+1}\inf_{2\Delta_\rho}M_{\omega^\rho}(T_p(h))^r k^{-1-\frac{n}{2}(r-1)}.
    \end{align*}

    Next, we have for the sum over large spacial displacements \(l\) under the use of \eqref{eq:BoundonN(v)inProof2} that
    \begin{align*}
         &\sum_{l> \tilde{C}k^{\frac{1}{2}+\vartheta}}|\Delta_k^l|(k+1)^{-r\frac{n}{2}}e^{-rC\frac{l^2}{k+1}}
         \lesssim\Big(\frac{1}{\rho}\Big)^{n+1}\sum_{l>\tilde{C}k^{\frac{1}{2}+\vartheta}}l^{n-2}e^{-rC\frac{l^2}{k+1}}
         \\
         &\qquad\lesssim\Big(\frac{1}{\rho}\Big)^{n+1}(k+1)^{-r\frac{n}{2}}\int_{\tilde{C}k^{\frac{1}{2}+\vartheta}}^\infty l^{n-2}e^{-rC\frac{l^2}{k+1}}dl
         \\
         &\qquad\lesssim\Big(\frac{1}{\rho}\Big)^{n+1}(k+1)^{-r\frac{n}{2}}k^{\frac{n-1}{2}}\int_{k^{\vartheta}}^\infty u^{n-2}e^{-rCu^2}du.
    \end{align*}
    If \(n\) is odd, then the antiderivative of \(u\mapsto u^{n-2}e^{-rCu^2}\) is given by \(P(u)e^{-u^2}\), where \(P\) is a polynomial in \(u\) of degree \(n-3\). To see this, we can use the substitution \(\tilde{u}=u^2\) to get \(\tilde{u}^{\frac{n}{2}-1}e^{-\tilde{u}}\), and then apply integration by parts iteratively until the monomial \(u^{\frac{n}{2}-1}\) vanishes. However, if \(n\) is even, we can just bound the integrand by the same integrand with \(n+1\) instead of \(n\) to obtain the odd \(n\) case. Overall we get for some polynomial \(P\) with degree depending only on \(n\)
    \[\int_{k^{\vartheta}}^\infty u^{n-2}e^{-rCu^2}d\rho\lesssim P(k^{\vartheta})e^{-rCk^{2\vartheta}}.\]
    Hence, we get for the sum over large \(l\)
    
    \begin{align*}
         &\sum_{l> \tilde{C}k^{\frac{1}{2}+\vartheta}}\int_{\Delta_k^l}N(v)^r d\sigma      \lesssim\Big(\frac{1}{\rho}\Big)^{n+1}\inf_{2\Delta_\rho}M_{\omega^\rho}(T_p(h))^r\frac{P(k^{\vartheta})}{k^{\frac{n}{2}(r-1)+\frac{1}{2}}}e^{-rCk^{2\vartheta}}.
    \end{align*}
    
    Hence in total, we obtain
    \begin{align*}
        &\int_{\partial\mathcal{O}\times\mathbb{R}} N(v)^r d\sigma\leq\int_{8\Delta_\rho} N(v)^r d\sigma + \sum_{k\in \mathbb{Z}}\Big(\sum_{l\leq \tilde{C}k^{\frac{1}{2}+\vartheta}}\int_{\Delta_k^l}N(v)^r d\sigma + \sum_{l>\tilde{C}k^{\frac{1}{2}+\vartheta}}\int_{\Delta_k^l}N(v)^r d\sigma\Big)
        \\
        &\qquad\lesssim \int_{8\Delta_\rho} N(v)^r d\sigma + \sum_{k\in \mathbb{Z}} \inf_{2\Delta_\rho}M_{\omega^\rho}(T_p(h))^r\Big(\frac{1}{\rho}\Big)^{n+1}\Big(k^{-1-\frac{n}{2}(r-1)} + \frac{P(k^{\vartheta})}{k^{\frac{n}{2}(r-1)+\frac{1}{2}}}e^{-rCk^{2\vartheta}}\Big)
        \\
        &\qquad\lesssim \int_{8\Delta_\rho} N(v)^r d\sigma + \int_{2\Delta_\rho}M_{\omega^\rho}(T_p(h))^r\sum_{k\in \mathbb{Z}} \Big(k^{-1-\frac{n}{2}(r-1)} + \frac{P(k^{\vartheta})}{k^{\frac{n}{2}(r-1)+\frac{1}{2}}}e^{-rCk^{2\vartheta}}\Big)
        \\
        &\qquad\lesssim \int_{8\Delta_\rho} M_{\omega^\rho}(T_p(h))^rd\sigma,
    \end{align*}
    which completes the proof of \eqref{eq:CorN(v)lesssimT_p}.
    \smallskip

    \medskip
    To complete the proof, we need to address the claim \eqref{eq:Claimvlesssimu}. Recall the definition of \(w\) in \eqref{eq:defOfuForMaxPrinc}. Instead of proving \eqref{eq:Claimvlesssimu} directly, we will apply the maximum principle on in space truncated versions \(U_1^R:=\mathcal{O}^R\times[-\frac{16}{\rho^2},\frac{16}{\rho^2})\setminus \tilde{U}\) and \(U_2^R:=\mathcal{O}^R\times(-\infty,-\frac{16}{\rho^2})\) where \(\mathcal{O}^R:=\mathcal{O}\cap T(\Delta(0, M_R))\). Here \(M_R\) is chosen such that \(v(X,t)\leq \frac{1}{R}\) for points \((X,t)\in U_1\setminus U_1^R\) and \((X,t)\in U_2\setminus U_2^R\).

    We will now show that there are uniform implicit constants such that
    \begin{align}
        v(Y,s)\lesssim w(Y,s) + \frac{1}{R} \qquad \textrm{for all }(Y,s)\in \partial U_1^R,\label{eq:CorMaxPrinciple1}
    \end{align}
    and 
    \begin{align}
        v(Y,s)\lesssim w(Y,s) + \frac{1}{R} \qquad \textrm{for all }(Y,s)\in \partial U_2^R.\label{eq:CorMaxPrinciple2}
    \end{align}
    Then, by maximum principle and letting \(R\) go to infinity, we obtain \(v(Y,s)\lesssim w(Y,s)\) on \(U_1\cup U_2\).
    \smallskip
    
    To prove \eqref{eq:CorMaxPrinciple1} and \eqref{eq:CorMaxPrinciple2}, we first observe that if \(t=\frac{16}{\rho^2}\), then
    \[v(Y,s)=0\leq w(Y,s).\]
    If \((Y,s)\in \partial \tilde{U}\) with \(\delta(Y,s)\approx \frac{1}{\rho}\), then
    \[G^*_1(\bar{A}(8\Delta_\rho), Y,s)\approx \rho^n\omega_{L_1}(\bar{A}(8\Delta_\rho), \Delta(Y^*,s,\frac{1}{\rho}))\geq c\rho^n,\]
    where the last inequality follows from Bourgain's estimate (Theorem 1.6 in \cite{genschaw_weak_2020}). Hence, \reflemma{Lemma2.10} yields
    \[v(Y,s)\lesssim \frac{\inf_{2\Delta_\rho}M_{\omega^\rho}(T_p(h))}{\rho^n}\rho^n\lesssim w(Y,s).\]
    If \((Y,s)\in \partial \tilde{U}\) with \(\delta(Y,s)\leq \frac{1}{\rho}\), then we split \(v=v^+-v^-\) into its positive and negative part in the sense that
    \begin{align*}
        &\begin{cases}
        L_1^*v^+=\mathfrak{h}^+:=\max \{0, \mathrm{div}_X(h)\} & \textrm{ in }\Omega,
        \\
        v^+=0&\textrm{ on }\partial\Omega,
        \end{cases} \qquad \textrm{and}
    \\ 
        &\begin{cases}
        L_1^*v^-=\mathfrak{h}^-:=\max \{0, -\mathrm{div}_X(h)\} & \textrm{ in }\Omega,
        \\
        v^-=0&\textrm{ on }\partial\Omega.
    \end{cases}
    \end{align*}
    From \(v^\pm(X,t)=\int_\Omega G(X,t,Y,s)\mathfrak{h}^\pm(Y,s) dYds\) it is clear that the integrand is nonnegative and hence \(v^\pm\) are nonnegative solutions on \(\tilde{U}\). This allows us to apply the local Comparison Principle which yields
    \begin{align*}
        \frac{v^\pm(Y,s)}{G^*_1(\bar{A}(8\Delta_\rho),Y,s)}&\lesssim \frac{v^\pm(\underline{A}_{1/\rho}(Y^*,s))}{G^*_1(\bar{A}(8\Delta_\rho),\bar{A}_{1/\rho}(Y^*,s))}\lesssim \frac{1}{\rho^n}v^\pm(\underline{A}_{1/\rho}(Y^*,s))
        \\
        &\lesssim \frac{1}{\rho^n}\inf_{2\Delta_\rho}M_{\omega^\rho}(T_p(h)),
    \end{align*}
    and hence \(v(X,t)\lesssim w(X,t)\).
    Since we chose \(M_R\) such that for all \((X,t)\in \Omega\) with \(|(X,t)-0|\geq M_R\) we obtain \(v(X,t)\leq \frac{1}{R}\leq w(X,t)\), we completed the proof of \eqref{eq:CorMaxPrinciple1} and \eqref{eq:CorMaxPrinciple2}.
    \qed

%\subsection{The quasi-dualising function \(h\) (Proof of \reflemma{lemma:defh})}
%\input{ProofhRigorous}
%\subsection{Boundedness of \(T(h)\) (Proof of \reflemma{lemma:BoundednessOfT(h)})}
%\input{ProofT(h)}

\section{Proof of \refthm{MainTheorem} for bounded base \(\mathcal{O}\)}\label{section:BoundedMainTheorem}
As mentioned before, the proof of \refthm{MainTheorem} for \(\Omega=\mathcal{O}\times\mathbb{R}\) with bounded \(\mathcal{O}\) is very similar to the above given proof for unbounded \(\mathcal{O}\). In this subsection, we would like to point out the differences.

Here we consider an only in time and close to the boundary truncated version of the nontangential maximal function.
\begin{defin}
    We define the truncated nontangential maximal function
    \[\tilde{N}_{1,\rho}[g](P,t):=\sup_{(Y,s)\in \Gamma^{\rho}(P,t)}\fint_{B(Y,s,\delta(Y,s)/2)}|g(X,\tau)|dXd\tau, \]
    where 
    \[\Gamma^{\rho}(P,t):=\Gamma(P,t)\cap \mathcal{O}\times [-\frac{1}{\rho^2},\frac{1}{\rho^2})\cap \{(X,s):\delta(X,s)>\rho\}.\]
\end{defin}

We obtain the same representation of the difference function \(F\) as in \reflemma{lemma:defF} and the quasi-dualising function \(h=h^\rho\), that satisfies
\begin{align*}
    \Vert\tilde{N}_{1,\rho}[\nabla F]\Vert_{L^q(\partial\Omega)}&\lesssim \int_\Omega \nabla F\cdot h dXdt,
\end{align*}
is also obtained in exactly the same way as in \reflemma{lemma:defh}. We define \(v\) exactly as in \eqref{vDefinedByDivh} and obtain 
\[\Vert \tilde{N}_{1,\rho}(\nabla F)\Vert_{L^q(\partial\Omega)}\lesssim \Vert\tilde{N}_{1,\rho}(\nabla u_0)\Vert_{L^q(\partial\Omega)}\Vert S(v)\Vert_{L^{q'}(\partial\Omega)}\]
like in \reflemma{lemma:defv}. Up to this point there were no real adjustments that differ from the case of unbounded base \(\mathcal{O}\). 
\smallskip

The reason why the core lemmas \reflemma{Lemma2.10}, \reflemma{Lemma2.11} and \refcor{cor:S(v)lesssimT_p(h)} need to be adjusted is that if \(\rho\) is sufficiently small, then there is no corkscrew point of \(\partial\Omega\cap Q(0,0,\frac{1}{\rho})\) in \(\Omega\). Hence, we do need to adapt these proofs, and partition the boundary of the domain into smaller boundary balls \(\Delta_k\), whose corresponding corkscrew points lie in the domain. Then the arguments of \reflemma{Lemma2.10}, \reflemma{Lemma2.11} and \refcor{cor:S(v)lesssimT_p(h)} work on thee smaller balls similarly to before, while we only need to demonstrate that the results on each \(\Delta_k\) can be reassembled to a global bound.
\smallskip

Specifically, let us introduce a countable collection of boundary balls \(\Delta_k\) such that \(\partial\mathcal{O}\times \mathbb{R}\subset \bigcup_k \Delta_k\) and the \(\Delta_k\) have finite overlap. Furthermore, we ask that \(\mathrm{diam}(\mathcal{O})\approx\mathrm{diam}(\Delta_k)\) such that there exists a backwards in time corkscrew point \(A^k=\underline{A}(2\Delta_k)\in \Omega\). We call \(\omega^k\) the parabolic measure of \(L_1^*\) with pole \(A^k\).  
\smallskip

Now we can replace \reflemma{Lemma2.10} with a for this setting modified version.

\begin{lemma}\label{lemma:ParabolicLemma2.10BoundedBase}
We have for \(p>\frac{n}{2}+1\) and \((P,t)\in \Delta^k\)
\[N(v)(P,t)+\tilde{N}(\delta|\nabla v|)(P,t)\lesssim M_{\omega^k}(T_p(h))(P,t).\]
\end{lemma}

\begin{proof}
Fix \(k\), set \(R:=\mathrm{diam}(\mathcal{O})\), and let \((P,\tau)\in \Delta_k\) with \((X,t)\in \Gamma(P,\tau)\). We define \(\Delta_j,\Omega_j,R_j,(X_j,t_j),(X_j,\bar{t}_j)\) for \(0\leq j\leq N\) as in the case for unbounded \(\mathcal{O}\) in the proof of \reflemma{Lemma2.10}. We choose \(N\) such that \(4R\leq 2^{N-1}\delta(X)\leq 8R\) and define  thick times slices \(U_i:=\mathcal{O}\times(t-2^iR^2,t-2^{i-1}R^2]\) for \(i\geq 1\) with \(U_0:=\mathcal{O}\times (t-R^2,t+R^2]\). Note that \((X,t)\in U_0\).

Now, we can split \(v\) up into different time piece integration domains
\begin{align*}
    v(X,t)&=\int_\Omega G_1(X,t,Y,s)\Div_Y(h)(Y,s)dYds=\int_\Omega \nabla_Y G_1(X,t,Y,s)h(Y,s)dYds
    \\
    &=\sum_{i\geq 0}\int_{U_i} \nabla_Y G_1(X,t,Y,s)h(Y,s)dYds
    =:\sum_{i\geq 0} I_i.
\end{align*}

Since \(U_0\subset\bigcup_{j=1}^\infty R_j\cup\Omega_0\cup Q(X,t,\delta(X,t)/2)\), we can do the same calculations as in the unbounded \(\Omega\) case (proof of \reflemma{Lemma2.10}) to obtain
\[I_0\lesssim M_{\omega^\EPS}(T_p(h))(P,\tau).\]

Next, let us consider \(I_i\). We call \(S_i\) the lateral boundary of \(U_i\), and denote by \(\bar{A}(S_i)=(a_x,t-(2^{i-1}-1)R^2)\) and \(\underline{A}(S_i)=(a_x,t-(2^{i}+1)R^2)\) a "forward or backwards in time corkscrew point" of \(S_i\). Here, \(a_x\in \mathcal{O}\), the center point of \(\mathcal{O}\), is chosen such that \(\mathrm{dist}(a_x,\partial\mathcal{O})\approx R\). The points \(\bar{A}(S_i)\) and \(\underline{A}(S_i)\) are not corkscrew points in the classical sense, since \(S_i\) does not have any that lie in the domain \(\Omega\), but they play the same role here. Furthermore, we set \(G^k(Y,s):=G(A^k,Y,s)\). Completely analogous to the proof of \(v_0\) in the unbounded \(\mathcal{O}\) case, we obtain by Cacciopolli and Harnack inequality

\begin{align*}
    |I_i|&\lesssim \sum_{I\in \mathcal{W}(S_i)}\Big(\int_{I^+}|\nabla_Y G(X,t,Y,s)|^2dYds\Big)^{1/2}\Big(\int_{I^+}|h(Y,s)|^2dYds\Big)^{1/2}
    \\
    &
    \lesssim \sum_{I\in \mathcal{W}(S_i)}\Big(\frac{1}{l(I)^2}\int_{\tilde{I}^+}|G(X,t,Y,s)|^2dYds\Big)^{1/2}\Big(\int_{I^+}|h(Y,s)|^2dYds\Big)^{1/2}
    \\
    &
    \lesssim \sum_{I\in \mathcal{W}(S_i)}\Big(\frac{1}{l(I)^2}\int_{\tilde{I}^+}\big|\frac{G(X,t,Y,s)}{G^k(Y,s)}G^k(Y,s)\big|^2dYds\Big)^{1/2}\Big(\int_{I^+}|h(Y,s)|^2dYds\Big)^{1/2}.
\end{align*}

Further, we use the local Comparison Principle, Backwards Harnack inequality, and the Comparison Principle to get
\begin{align*}
    \int_{\tilde{I}^+}\big|\frac{G(X,t,Y,s)}{G^k(Y,s)}G^k(Y,s)\big|^2dYds&\lesssim\int_{\tilde{I}^+}\big|\frac{G(X,t,\overline{A}(S_i))}{G^k(\underline{A}(S_i))}G^k(Y,s)\big|^2dYds
    \\
    &\lesssim \int_{\tilde{I}^+}\big|\frac{G(X,t, \underline{A}(S_i))}{G^k(\underline{A}(S_i))}G^k(Y,s)\big|^2dYds
    \\
    &\lesssim \int_{\tilde{I}^+}\big|\frac{R^nG(X,t, \underline{A}(S_i))}{\omega^k(\Delta_{i+2})}G^k(Y,s)\big|^2dYds
    \\
    &\lesssim \frac{R^{2n}G(X,t, \underline{A}(S_i))^2}{\omega^k(\Delta_{i+2})^2}\int_{\tilde{I}^+}\frac{\omega^k(I)^2}{l(I)^{2n}}dYds.
\end{align*}
By the growth estimate of the Green's function (\reflemma{lemma:ParabolicGrowthEstimateForG}) we observe
\[R^{n}G(X,t, \underline{A}(S_i))\lesssim R^n\frac{1}{2^{n(i-1)}R^n}\lesssim \frac{1}{2^{in}},\]
whence
\begin{align*}
    |I_i|& \lesssim \sum_{I\in \mathcal{W}(S_i)}\frac{2^{-in}}{\omega^k(\Delta_{i+2})}\Big(\frac{1}{l(I)^2}\int_{\tilde{I}^+}\frac{\omega^k(I)^2}{l(I)^{2n}}dYds\Big)^{1/2}\Big(\int_{I^+}|h(Y,s)|^2dYds\Big)^{1/2}
    \\
    &\lesssim 2^{-in} M_{\omega^k}(T_p(h))(P,\tau).
\end{align*}
Here the last inequality followed as in the proof for \(v_0\) in the unbounded \(\Omega\) case. As a consequence, we can finish the proof by summing over all \(i\)
\[|v(X,t)|\lesssim \sum_{i=0}^\infty I_i\lesssim \sum_{i=0}^\infty 2^{-in} M_{\omega^k}(T_p(h))(P,\tau)\lesssim M_{\omega^k}(T_p(h))(P,\tau).\]
\end{proof}

The good \(\lambda\)-inequality (\reflemma{Lemma2.11}) also holds verbatim. However, the proof needs to be adapted for larger Whitney cubes.

\begin{proof}[Proof adjustments for \reflemma{Lemma2.11}]
Analogously to the unbounded \(\mathcal{O}\) case, we take a Whitney cube \(\Delta\) of \(\{S[v]>\lambda\}\). We distinguish two cases here. If the Whitney cube \(\Delta\) is small, i.e. \(l(\Delta)<\frac{1}{2}\mathrm{diam}(\mathcal{O})\), then we can proceed completely analogously to the unbounded case without modifications. 

\medskip

If \(\Delta\) is a large Whitney cube, i.e. \(l(\Delta)\geq\frac{1}{2}\mathrm{diam}(\mathcal{O})\), then we set \(E_k:=E\cap \Delta_k\), where \(E\) is as before given by \eqref{eq:defSetEGoodlambda}. Now, we notice that each \(\Delta_k\) is a small cube in above sense, but no Whitney cube of the original set \(\{S(v)>\lambda\}\). However, we only used this fact to deduce the existence of \(\tau\) such that the truncated square function \(S_{\tau l(\Delta)}(v)>\frac{\lambda}{2}\) on \(E\). This argument is not necessary in this case, and we can replace every truncated square function \(S_{\tau l(\Delta)}(v)\) by the usual nontruncated square function \(S(v)\) and notice that still \(S(v)>\lambda\) on \(E_k\). Hence, we can reproduce the argument from the unbounded \(\mathcal{O}\) case and we obtain
\[\omega^k(E_k)\lesssim \eta^2\omega^k(\Delta_k).\]
Here we denote by \(\omega^k\) the parabolic measure to the operator \(L_1^*\) with pole in \(A^k\), the backwards in time corkscrew point of \(\Delta_k\).
Since \(\omega_{L_1^*}\in A_\infty(\sigma)\) there exists a sufficiently small \(\eta>0\) such that
\[\sigma(E_k)\leq \frac{\gamma}{C}\sigma(\Delta_k),\]
where \(C\) is the doubling constant of \(\omega_{L_1^*}\). Thus, by our construction and the doubling property we have
\[\sigma(E)\leq \sum_{k}\sigma(E_k)\leq \frac{\gamma}{C}\sum_{k}\sigma(\Delta_k)\lesssim \frac{\gamma}{C}\sigma(2\Delta)\leq \gamma\sigma(\Delta).\]
\end{proof}

The consequence of the previous two lemmas is the following modified corollary, which replaces \refcor{cor:S(v)lesssimT_p(h)}.
\begin{cor}\label{cor:S(v)lesssimT_p(h)Bounded}
We have for \(r>1\)
\[\int_{\partial\Omega} S(v)^r d\sigma \leq C\sum_k\int_{\Delta_k} M_{\omega^k}(T_p(h))^r d\sigma. \]
\end{cor}

This time, this corollary is a direct consequence of the good \(\lambda\)-inequality. The right hand side is now different from the right hand side of \reflemma{lemma:BoundednessOfS(v)}. However, we still have a modified version of \reflemma{lemma:BoundednessOfS(v)}.

\begin{lemma}\label{lemma:BoundednessOfS(v)Bounded}
    We have
    \[\sum_k\int_{\Delta_k} M_{\omega^k}(T_p(h))^{q_1'} d\sigma\leq C<\infty.\]
\end{lemma}

\begin{proof}
By \eqref{eq:BoundMaxRevHölder} we have
\[\sum_k\int_{\Delta_k} M_{\omega^k}(T_p(h))^{q_1'} d\sigma\lesssim \sum_k\int_{\Delta_k} T_p(h)^{q_1'} d\sigma=\int_{\partial\Omega} T_p(h)^{q_1'} d\sigma.\]
Since the proof of \reflemma{lemma:BoundednessOfT(h)} is verbatim the same for a bounded base \(\mathcal{O}\), we have \(\int_{\partial\Omega} T_p(h)^{q_1'} d\sigma\leq C<\infty\) for some constant \(C>0\) which finishes the proof.
\end{proof}

The final proof of \refthm{MainTheorem} for bounded \(\mathcal{O}\) follows as before in Subsection \ref{subsection:UnboundedStructure}, since all the necessary have been established.

\section{Proof of perturbation for Neumann problem (\refthm{thm:MaintheoremNeumann})}\label{section:ProofNeumann}

First, let \(g\in \dot{H}^{-1/4}_{\partial_t-\Delta}(\partial\Omega)\cap L^{q_0}(\Omega)\) be the boundary data for the Neumann problem of \(L_0\) and \(u_0\) the corresponding solutions. 
%By (???, parabolic result) we have that \(\Vert u_0\Vert_{\dot{E}}\lesssim \Vert g\Vert_{\dot{H}^{-1/4}_{-\partial_t-\Delta}(\partial\Omega)}\) and from solvability of the \(L^{q_0}\) Neumann problem for we also have \(\Vert\tilde{N}(\nabla u_0)\Vert_{L^{q_0}(\partial\Omega)}\lesssim \Vert g\Vert_{L^{q_0}(\partial\Omega)}\).
We set \(\EPS_0:=\Vert\mu\Vert_{\mathcal{C}}\) and \(u_1(X,t):=u_0(X,t)+F(X,t)\) where \(F\) is defined as in \reflemma{lemma:defF}. As in the proof for the perturbation of the Regularity problem, we have that \(L_1u_1=0\). However, in contrast to the Regularity problem, \(F\) does not have zero Neumann boundary data which means that the Neumann boundary data obtained by \(u_1\) is not \(g\). Instead, we can observe that for every test function \(\phi\in C_C^\infty(\mathbb{R}^{n+1})\)
\begin{align}
    &\int_\Omega -u_1(X,t)\partial_t\phi(X,t) + A_1(X,t)\nabla u_1(X,t) \cdot\nabla\phi(X,t)dXdt\nonumber
    \\
    &=\int_{\partial\Omega}g(X,t)\phi(X,t)d\sigma(X,t) + \int_\Omega \EPS(X,t)\nabla u_0(X,t)\cdot \nabla v(X,t) dXdt,\label{eq:Neumannsomething}
\end{align}
where \(v\) solves the Dirichlet problem
\begin{align}\begin{cases}
        L_1^* v=0 & \textrm{ in }\Omega, \\
        v=\phi & \textrm{ on }\partial\Omega.
    \end{cases}\label{vDefinedNeumann}\end{align}
We postpone the proof of \eqref{eq:Neumannsomething} to the appendix \ref{subsection:appendix2}, since it mainly consists of standard approximation arguments. We continue with defining the operator \(T\) on \(L^{q_0}(\partial\Omega)\cap \dot{H}^{-1/4}_{\partial_t-\Delta}(\partial\Omega)\) by setting
\[\int_\Omega T(g)(X,t)\phi(X,t)dXdt=\int_\Omega \EPS(X,t)\nabla u_0(X,t)\cdot \nabla v(x,t) dXdt\]
for every test function \(\phi\in C_C^\infty(\mathbb{R}^{n+1})\).
\smallskip

With this in hand we observe that by a stopping time argument exactly as in the proof of \reflemma{lemma:defv} we obtain 
\[\int_\Omega \EPS(X,t)\nabla u_0(X,t)\cdot \nabla v(x,t) dXdt\lesssim \Vert\mu\Vert_{\mathcal{C}}\Vert \tilde{N}_2(\nabla u_0)\Vert_{L^{q_0}(\partial\Omega)}\Vert S(v)\Vert_{L^{q_0'}(\partial\Omega)}.\]
By assumption of solvability of the \(L^{q_0}\) Neumann problem for \(L_0\) we have \(\Vert\tilde{N}(\nabla u_0)\Vert_{L^{q_0}(\partial\Omega)}\lesssim \Vert g\Vert_{L^{q_0}(\partial\Omega)}\). Since we also assumed solvability of the parabolic Dirichlet problem for \(L_1^*\) we have \(\Vert S(v)\Vert_{L^{q_0'}}\approx \Vert N(v)\Vert_{L^{q_0'}}\lesssim \Vert \phi\Vert_{L^{q_0'}}\) by Theorem 5.1 in \cite{nystrom_dirichlet_1997} (the theorem is only formulated for domains with bounded spatial base \(\mathcal{O}\), but the proof works verbatim for unbounded Lipschitz graph spatial domains \(\mathcal{O}\)). Hence we obtain that \(\Vert T(g)\Vert_{L^{q_0}}\lesssim \EPS_0\Vert g\Vert_{L^{q_0}}\).

On the other hand, we have \[\int_\Omega \EPS(X,t)\nabla u_0(X,t)\cdot \nabla v(x,t) dXdt\lesssim \EPS_0\Vert \nabla u_0\Vert_{L^2(\Omega)}\Vert \nabla v\Vert_{L^2(\Omega)}\lesssim \EPS_0\Vert u_0\Vert_{\dot{E}}\Vert v\Vert_{\dot{E}}.\]
By Section \ref{subsection:EnergySol} we have 
\(\Vert u_0\Vert_{\dot{E}}\lesssim \Vert g\Vert_{\dot{H}^{-1/4}_{\partial_t-\Delta}(\partial\Omega)}\) and \(\Vert v\Vert_{\dot{E}}\lesssim \Vert \phi\Vert_{\dot{H}^{1/4}_{\partial_t-\Delta}(\partial\Omega)}\).
Thus we obtain \(\Vert T(g)\Vert_{\dot{H}^{-1/4}_{\partial_t-\Delta}(\partial\Omega)}\lesssim \EPS_0\Vert g\Vert_{\dot{H}^{-1/4}_{\partial_t-\Delta}(\partial\Omega)}\), and that
\[T: \dot{H}^{-1/4}_{\partial_t-\Delta}(\partial\Omega)\cap L^{q_0}(\Omega)\to \dot{H}^{-1/4}_{\partial_t-\Delta}(\partial\Omega)\cap L^{q_0}(\Omega)\]
is a bounded linear operator.
Furthermore, for sufficiently small \(\EPS_0\)
\[R:\dot{H}^{-1/4}_{\partial_t-\Delta}(\partial\Omega)\cap L^{q_0}(\Omega)\to \dot{H}^{-1/4}_{\partial_t-\Delta}(\partial\Omega)\cap L^{q_0}(\Omega), R(g):=g+T(g)\]
is still a linear operator bounded in the norms of \(L^{q_0}\) and \(\dot{H}^{-1/4}_{\partial_t-\Delta}(\partial\Omega)\), and has a bounded inverse by the open mapping theorem. 

Now given \(f\in \dot{H}^{-1/4}_{-\partial_t-\Delta}(\partial\Omega)\cap L^{q_0}(\Omega)\), there exists \(g=R^{-1}(f)\in \dot{H}^{-1/4}_{-\partial_t-\Delta}(\partial\Omega)\cap L^{q_0}(\Omega)\) and we can define \(u_0\) as the solution to \(L_0u_0=0\) with Neumann data \(g\). Setting \(u_1:=u_0+F\) gives a solution \(u_1\) to \(L_1u_1=0\) that satisfies \eqref{eq:Neumannsomething}. However the right hand side of \eqref{eq:Neumannsomething} is now 
\begin{align*}
    &\int_{\partial\Omega}g(X,t)\phi(X,t)d\sigma(X,t) + \int_{\partial\Omega}T(g)(X,t)\phi(X,t) d\sigma(X,t)
    \\
    &\qquad=\int_{\partial\Omega}R(g)(X,t)\phi(X,t)d\sigma(X,t)=\int_{\partial\Omega}f(X,t)\phi(X,t)d\sigma(X,t).
\end{align*}
Hence \(u_1\) attains the given Neumann data \(f\). Furthermore as in the proof of \refthm{MainTheorem}, we can use the combination of \reflemma{lemma:defv}, \refcor{cor:S(v)lesssimT_p(h)}, \reflemma{lemma:BoundednessOfS(v)} to conclude \(\Vert\tilde{N}(\nabla F)\Vert_{L^{q_0}}\lesssim \Vert \mu\Vert_{\mathcal{C}} \Vert\tilde{N}(\nabla u_0)\Vert_{L^{q_0}}\). Hence we have
\begin{align*}
    \Vert\tilde{N}(\nabla u_1)\Vert_{L^{q_0}}&\leq \Vert\tilde{N}(\nabla u_0)\Vert_{L^{q_0}} + \Vert\tilde{N}(\nabla F)\Vert_{L^{q_0}}\lesssim (1+\EPS_0)\Vert\tilde{N}(\nabla u_0)\Vert_{L^{q_0}}
    \\
    &\lesssim \Vert g\Vert_{L^{q_0}}\lesssim \Vert f\Vert_{L^{q_0}},
\end{align*}
where we used the boundedness of \(R^{-1}\) in \(L^{q_0}\) in the last step. This proves solvability of \((N)^{L_1}_{q_0}\) and hence \refthm{thm:MaintheoremNeumann}.

\appendix
\section{Appendix}

\subsection{Condition \eqref{eq:CarlesonWithoutSupNorm}: Gradient bound for solution}
Inspired by the Lemma 2.15 in (\cite{feneuil_green_2023}) we prove the following parabolic version.
\begin{lemma}\label{lemma:nablaUpointwiseBound}
    Assume that \(L:=\partial_t - \mathrm{div}(A\nabla \cdot)\) and \(|\nabla A(X,t)|\lesssim \frac{1}{\delta(X,t)}\). Let \(u\) be a solution of \(Lu=0\). Then 
    \[\sup_{Q(X,t,\delta(X,t)/4)}|\nabla u|\lesssim  \Big(\fint_{Q(X,t,\delta(X,t)/3)}|\nabla u|^2\Big)^{\frac{1}{2}} \lesssim \frac{\sup_{Q(X,t,\delta(X,t)/2)}|u|}{\delta(X,t)}.\]
\end{lemma}

To prove \reflemma{lemma:nablaUpointwiseBound} we need to establish the following Cacciopolli type lemma first. Therefore, we set \(w_i:=\partial_i u\) for \(i=1,...,n\) and notice that
\[Lw_i=\mathrm{div}(\partial_i A\nabla u).\] Recall also that we use the notation \(Q(X,r)\) for a cube in \(\mathcal{O}\) and \(Q(X,t,r)\) for a cube in \(\Omega\) with radius \(r\).

\begin{lemma}[Cacciopolli type inequality]\label{lemma:CaccioppoliTypeIterationInequality}
Assume that \(L:=\partial_t - \mathrm{div}(A\nabla \cdot)\) and \(|\nabla A(X,t)|\lesssim \frac{1}{\delta(X,t)}\). Let \(0<r_1<r_2<\delta(X,t), \alpha>0\) and \(u\) a solution to \(Lu=0\). Then for \(w_i:=\partial_i u\)
\[\sum_{i=1}^n\int_{Q(X,t,r_1)} |w_i|^\alpha|\nabla w_i|^2dYds\lesssim \sum_{i=1}^n\Big(\frac{1}{(r_2-r_1)^2}+\frac{1}{\delta(X,t)^2}\Big)\int_{Q(X,t,r_2)} |w_i|^{2+\alpha}dYds,\]
and
\begin{align*}&\sum_{i=1}^n\sup_{s\in (-r_1+t,t+r_1)}\int_{Q(X,r_1)} |w_i(Y,s)|^{\alpha+2}dY
\\
&\qquad \lesssim \sum_{i=1}^n\Big(\frac{1}{(r_2-r_1)^2}+\frac{1}{\delta(X,t)^2}\Big)\int_{Q(X,t,r_2)} |w_i|^{2+\alpha}dYds.\end{align*}
\end{lemma}

\begin{proof}
    Let \(0\leq \eta\leq 1\) be a smooth cut-off function with \(\eta\equiv 1\) on \(Q(X,t,r_1)\) and \(\eta\equiv 0\) on \(\Omega\setminus Q(X,t,r_2)\), where we also assume \(|\nabla \eta|(r_2-r_1)+|\partial_t\eta|(r_2-r_1)^2\lesssim 1\). Then we have

    \begin{align*}
        \int_{Q(X,t,r_2)}|w|^\alpha|\nabla w_i|^2\eta^2dYds&\lesssim \int_{Q(X,t,r_2)}A\nabla w_i\nabla (w_i|w_i|^\alpha\eta^2)dYds
        \\
        &\qquad - \int_{Q(X,t,r_2)}A\nabla w_i\nabla \eta^2 |w_i|^\alpha w_idYds.
    \end{align*}

    The second integral can be bounded above via Cauchy-Schwarz by
    \[\Big(\int_{Q(X,t,r_1)} |w_i|^\alpha|\nabla w_i|^2\eta^2dYds\Big)^{\frac{1}{2}}\Big(\int_{Q(X,t,r_1)} |w_i|^{2+\alpha}|\nabla \eta^2|dYds\Big)^{\frac{1}{2}}.\]
    We can apply Young's inequality and hide the first factor on the left hand side. The second factor is what we expect on the right hand side.

    \smallskip
    The first integral becomes via the PDE
    \begin{align*}
        \int_{Q(X,t,r_2)}A\nabla w_i\nabla (w_i|w_i|^\alpha\eta^2)dYds&=\int_{Q(X,t,r_2)}\partial_i A\nabla u\nabla (w_i|w_i|^\alpha\eta^2)dYds
        \\
        &- \int_{Q(X,t,r_2)}\partial_t w_i (w_i|w_i|^\alpha\eta^2)dYds =:I_i+II_i.
    \end{align*}

    For each \(II_i\) we obtain by integration by parts
    \[II_i=\frac{1}{2}\int_{Q(X,t,r_2)}(|w_i|^{2+\alpha}+w_i|w_i|^{1+\alpha})\partial_t\eta^2dYds\lesssim \frac{1}{(r_2-r_1)^2}\int_{Q(X,t,r_2)}|w_i|^{2+\alpha}dYds.\]

    For \(I_i\) we need to sum over all \(i\) and obtain
    \begin{align*}
        \sum_{i=1}^n I_i&\lesssim \sum_{i=1}^n\sup_{Q(X,t,r_2)}|\nabla A|\Big(\int_{Q(X,t,r_2)}|\nabla u|^2|w_i|^\alpha dYds\Big)^{\frac{1}{2}}
        \\
        &\qquad\cdot\Big(\int_{Q(X,t,r_2)}|\nabla w_i\eta^2 + w_i\nabla \eta^2|^2|w_i|^\alpha dYds\Big)^{\frac{1}{2}}
    \end{align*}
    We use the pointwise bounds \(|\nabla A|\lesssim \frac{1}{\delta}\) and \(|\nabla \eta|\lesssim \frac{1}{r_2-r_1}\), that
    \begin{align*}
        \int_{Q(X,t,r_2)}|\nabla u|^2|w_i|^\alpha dYds=\int_{Q(X,t,r_2)}\sum_{j=1}^n |w_j|^2|w_i|^\alpha dYds \lesssim \int_{Q(X,t,r_2)}\sum_{j=1}^n |w_i|^{2+\alpha} dYds,
    \end{align*}
    and Young's inequality to get
    \begin{align*}
        \sum_{i=1}^nI_i
        &\lesssim\rho\sum_{i=1}^n\int_{Q(X,t,r_2)} |w_i|^\alpha|\nabla w_i|^2dYds
        \\
        &\qquad + C_\rho\sum_{i=1}^n\Big(\frac{1}{(r_2-r_1)^2} +\frac{1}{\delta(X,t)^2}\Big)\int_{Q(X,t,r_2)} |w_i|^{2+\alpha}dYds,
    \end{align*}
    where the first sum can be hidden on the left side with a sufficiently small \(\rho\).
    \smallskip
    %=====================================================================

     For the second part of the lemma let \(0\leq \eta\leq 1\) be a smooth cut-off function with \(\eta\equiv 1\) on \(Q(X,r_1)\) and \(\eta\equiv 0\) on \(\Omega\setminus Q(X,\tilde{r})\), where \(\tilde{r}:=r_1+\frac{r_2-r_1}{2}\) and we also assume \(|\nabla \eta|(r_2-r_1)\lesssim 1\). We define also a rough cut-off in time by for \(s\in (t-r_1^2,t+ r_1^2)\)
    \[\phi(\tau)=\begin{cases}
        1& \textrm{if }\tau\in (t-r_1^2,s),\\
        \frac{\tau-t +\tilde{r}^2}{\tilde{r}^2-r_1^2} & \textrm{if }\tau\in(t-\tilde{r}^2, t-r_1^2),\\
        0 & \textrm{else}.
    \end{cases}\]
    Then we have by the Fundamental Theorem of Calculus
    \begin{align*}
        \int_{Q(X, r_1)}|w_i(Y,s)|^{2+\alpha}dY&\leq \int_{Q(X,t, \tilde{r})}\partial_t(|w_i(Y,s)|^{2+\alpha}\eta^2\phi^2) dYd\tau 
        \\
        &\leq \int_{Q(X,t, \tilde{r})}\partial_tw_i w_i|w_i(Y,s)|^{\alpha}\eta^2\phi^2 dYd\tau
        \\
        &\qquad + \frac{1}{(r_2-r_1)^2}\int_{Q(X,t, \tilde{r})}|w_i(Y,s)|^{2+\alpha}\eta^2\phi^2 dYd\tau.
    \end{align*}
    By use of the PDE for \(w_i\) we obtain for the first term
    \begin{align*}
        -\int_{Q(X,t,\tilde{r})} A\nabla w_i \nabla ( w_i|w_i|^\alpha\eta^2\phi^2)dYd\tau + \int_{Q(X,t,\tilde{r})} \partial_iA\nabla u \nabla ( w_i|w_i|^\alpha\eta^2\phi^2)dYd\tau.
    \end{align*}
    The second term can be bound as in \(II_i\) and the first one can be bounded by the left hand side of the first statement of the lemma for \(r_1=\tilde{r}\). Hence applying the already established part of the lemma on the second term yields, together with the other bounds, 
    \begin{align*}&\sum_{i=1}^n\sup_{s\in (-r_1+t,t+r_1)}\int_{Q(X,r_1)} |w_i|^{\alpha+2}dYds
    \\
    &\qquad\lesssim \sum_{i=1}^n\Big(\frac{1}{(r_2-r_1)^2}+\frac{1}{\delta(X,t)^2}\Big)\int_{Q(X,t,r_2)} |w_i|^{2+\alpha}dYds.\end{align*}

\end{proof}

Finally we can prove \reflemma{lemma:nablaUpointwiseBound}.

\begin{proof}[Proof of \reflemma{lemma:nablaUpointwiseBound}]
    We would like to apply a Moser iteration scheme. For that, we set \(\alpha_j=2\kappa^j-2, r_j:=\delta(X,t)/4+2^{-j-1}\delta(X,t)/4\) and \(\kappa:=\frac{2^*}{2}\). Furthermore let \(0\leq \eta_j\leq 1\) be a smooth cut-off function with \(\eta_j\equiv 1\) on \(Q(X,t,r_j)\) and \(\eta_j\equiv 0\) on \(\Omega\setminus Q(X,t,r_{j+1})\), where we also assume \(|\nabla \eta_j|(r_{j+1}-r_j)+|\partial_t\eta|(r_{j+1}-r_j)^2\lesssim 1\).
    
    Then we obtain for \(v_i^j:=|w_i|^{\frac{\alpha_j}{2}}w_i\eta_j\) 
    
    \begin{align*}
        &\sum_{i=1}^n\Big(\fint_{Q(X,t,r_j)} |w_i|^{(\alpha_j+2)\kappa}dYds\Big)^{\frac{2}{2\kappa}}\leq \sum_{i=1}^n\Big(\fint_{Q(X,t,r_{j+1})} |w_i|^{(\alpha_j+2)\kappa}\eta^{2\kappa}dYds\Big)^{\frac{2}{2\kappa}}
        \\
        &\qquad=\sum_{i=1}^n\Big(\fint_{t-r_{j+1}^2}^{t+r_{j+1}^2}\fint_{Q(X,t,r_{j+1})} |v_i^j|^{2^*}dYds\Big)^{\frac{2}{2^*}}
        \\
        &\qquad=\sum_{i=1}^n\Big(\fint_{t-r_{j+1}^2}^{t+r_{j+1}^2}\Big(\fint_{Q(X,r_{j+1})} |v_i^j|^{2^*}dY\Big)^{\frac{2}{2^*}}ds\Big)^{\frac{2}{2^*}} 
        \\
        &\hspace{20mm}\cdot\Big(\sup_{s\in (t-r_j^2,t+r_j^2)}\fint_{Q(X,r_{j+1})} |v_i^j|^{2^*}dY\Big)^{\frac{2}{2^*}(1-\frac{2}{2^*})}.
    \end{align*}
    By \reflemma{lemma:CaccioppoliTypeIterationInequality} and Young's inequality we bound this for a sufficiently small \(\sigma>0\) by
    \begin{align*}
        \sum_{i=1}^n C_\sigma(1+2^{2j})^{\frac{2^*}{2}} \Big(\fint_{t-r_{j+1}^2}^{t+r_{j+1}^2}\Big(\fint_{Q(X,r_{j+1})} |v_i^j|^{2^*}dY\Big)^{\frac{2}{2^*}}ds\Big)
        \\
        +  \sigma\Big(\fint_{t-r_{j+1}^2}^{t+r_{j+1}^2}\fint_{Q(X,r_{j+1})} |v_i^j|^{2^*}dY\Big)^{\frac{2}{2^*}}.
    \end{align*}
    Now, we can hide the second term on the left hand side and continue with the Sobolev inequality on the spatial cube \(Q(X,r_{j+1})\):
    \begin{align*}
        &C_\sigma(1+2^{2j})^{\frac{2^*}{2}} \sum_{i=1}^n  \Big(\fint_{t-r_{j+1}^2}^{t+r_{j+1}^2}\Big(\fint_{Q(X,r_{j+1})} |v_i^j|^{2^*}dY\Big)^{\frac{2}{2^*}}ds\Big)
        \\
        &\qquad\lesssim C_\sigma(1+2^{2j})^{\frac{2^*}{2}}\delta(X,t)^2\sum_{i=1}^n  \Big(\fint_{t-r_{j+1}^2}^{t+r_{j+1}^2}\fint_{Q(X,r_{j+1})} |\nabla v_i|^{2}dYds\Big)
    \end{align*}
    Substituting in the expression for \(v_i^j\) and using \reflemma{lemma:CaccioppoliTypeIterationInequality} yields that
    \begin{align*}
        &\sum_{i=1}^n\fint_{t-r_{j+1}^2}^{t+r_{j+1}^2}\fint_{Q(X,r_{j+1})} |\nabla v_i|^{2}dYds
        \\
        &\lesssim\sum_{i=1}^n\fint_{Q(X,t,r_{j+1})} \big(\frac{\alpha_j}{2}+1\big)^2|w_i|^{\alpha}|\nabla w_i|^2\eta_j^{2} + |w_i|^{(\alpha_j+2)}|\nabla \eta_j|^2dYds
        \\
        &\lesssim\sum_{i=1}^n\fint_{Q(X,t,r_{j+1})} \big(\frac{\alpha_j}{2}+2\big)^2|w_i|^{2+\alpha_j}\big(|\nabla\eta_j|^{2}+\frac{1}{2^{2j}\delta(X,t)^2}+\frac{1}{\delta(X,t)^2}\big)dYds.
    \end{align*} 
    All together, we have now that
     \begin{align*}
        &\sum_{i=1}^n\Big(\fint_{Q(X,t,r_j)} |w_i|^{(2+\alpha_j)\kappa}dYds\Big)^{\frac{2}{2\kappa}}
        \\
        &\leq C_\sigma(1+2^{2j})^{\frac{2^*}{2}} \delta(X,t)^2 \Big(\frac{(\frac{\alpha_j}{2}+2)^2}{(r_{j+1}-r_j)^2} + \frac{(\frac{\alpha_j}{2}+2)^2}{(\delta(X,t))^2}\Big)\sum_{i=1}^n\int_{Q(X,t,r_{j+1})} |w_i|^{2+\alpha_j}dYds.
        \\
        &\qquad\lesssim C_\sigma(1+2^{2j})^{\frac{2^*}{2}+1} (\frac{\alpha_j}{2}+2)^2\sum_{i=1}^n\fint_{Q(X,t,r_{j+1})} |w_i|^{2+\alpha_j}dYds.
    \end{align*}
    
    If we set
    \[W_j^{(i)}:=\sum_{i=1}^n\Big(\fint_{B(X,t,r_j)} |w_i|^{2\kappa^j}dYds\Big)^{\frac{1}{2\kappa^j}},\]
    then the previous observation yields
    \[W_{j+1}^{(i)}\lesssim \Big( C_\sigma(1+2^{2j})^{\frac{2^*}{2}+1} (\frac{\alpha_j}{2}+2)^2 \Big)^{\frac{1}{2\kappa^j}}W_{j}^{(i)}.\]

    Iteration leads to
    \begin{align*}
    \sum_{i=1}^n\sup_{B(X,t,\delta(X,t)/4)} |w_i|\lesssim \sum_{i=1}^n \prod_{j=1}^\infty \Big( C_\sigma(1+2^{2j})^{\frac{2^*}{2}+1} (\frac{\alpha_j}{2}+2)^2 \Big)^{\frac{1}{2\kappa^j}} W_1^{(i)},
    \end{align*}

    where the product rewrites to the exponential function of  
    \[\sum_{j=1}^\infty\frac{1}{2\kappa^j}\ln\Big( C_\sigma(1+2^{2j})^{\frac{2^*}{2}+1} (\frac{2\kappa^j-2}{2}+2)^2 \Big).\]
    Since this sum converges against a fixed value, we obtain with the usual Caccioppoli and Harnack inequality for \(u\)
    \begin{align*}
    &\sup_{B(X,t,\delta(X,t)/4)} |\nabla u|\lesssim \sum_{i=1}^n\sup_{B(X,t,\delta(X,t)/4)} |w_i|\lesssim \sum_{i=1}^n\Big(\fint_{B(X,t,\delta(X,t)/2)} |w_i|^2 dYds\Big)^{1/2} 
    \\
    &\qquad\lesssim (\fint_{B(X,t,\delta(X,t)/2)} |\nabla u|^2 dYds\Big)^{1/2}
    \lesssim \frac{1}{\delta(X,t)}(\fint_{B(X,t,3\delta(X,t)/4)} |u|^2 dYds\Big)^{1/2}
    \\
    &\qquad\lesssim \frac{\sup_{B(X,t,3\delta(X,t)/4)} |u|}{\delta(X,t)},
    \end{align*}

\end{proof}

\subsection{Approximation argument in \eqref{eq:Neumannsomething}}\label{subsection:appendix2}

To make \eqref{eq:Neumannsomething} rigorous, we note that
\begin{align}
    \eqref{eq:Neumannsomething}=&\int_\Omega -u_1\partial_t\phi + A_1\nabla u_1 \cdot\nabla\phi dXdt\nonumber
    \\
    &= \int_\Omega -u_0\partial_t\phi + A_1\nabla u_0 \cdot\nabla\phi dXdt\nonumber
     +  \int_\Omega -F\partial_t\phi+ A_1\nabla F \cdot\nabla\phi dXdt\nonumber
    \\
    &= \int_\Omega -u_0\partial_t\phi + A_0\nabla u_0 \cdot\nabla\phi dXdt\nonumber
     + \int_\Omega \EPS \nabla u_0 \cdot\nabla\phi dXdt\nonumber
    \\
    &\qquad +\int_\Omega -F\partial_t\phi + A_1\nabla F \cdot\nabla\phi dXdt.\label{eq:Neumannequation1}
\end{align}
The first term is the PDE for \(u_0\) and will give us the term containing the Neumann data \(g\). The second term is also what we expect on the right hand side of \eqref{eq:Neumannsomething}. Hence it remains to study the last term in detail. We can find sequences \((A_l)_l, (\EPS_m)_m\subset C^\infty(\Omega)\) and \((u_m)_m\subset C_C^\infty(\Omega)\) such that \(A_l\to A_1\) in \(L^2(\Omega\cap\mathrm{supp}(\phi))\), \(u_m\to u\) in \(\dot{E}\), and \(\EPS_m\nabla u_m\to \EPS\nabla u_0\) for in \(L^2(\Omega)\). With these approximations in hand we will produce error terms, which we will call \(R_1,...,R_6\) and which we will discuss in the end. But first, we have that
\begin{align*}
    &\int_\Omega -F\partial_t\phi + A_1\nabla F \cdot\nabla\phi dXdt
    \\
    &=\int_\Omega -F\partial_t\phi dXdt + \int_\Omega(A_1-A_l)\nabla F \cdot\nabla\phi dXdt
    + \int_\Omega A_l^*\nabla\phi \cdot \nabla F dXdt
    \\
    &=:I + R_1 + J 
\end{align*}
For \(I\), we insert the definition of \(F\) (see \reflemma{lemma:defF}) and obtain by Fubini
\begin{align*}
    I&=-\int_\Omega\Big(\int_\Omega\EPS(Y,s)\nabla u_0(Y,s)\cdot \nabla_Y G_1(Y,s,X,t) dYds\Big)\partial_t\phi(X,t)dXdt
    \\
    &=-\int_\Omega \Big(\int_\Omega\EPS_m(Y,s)\nabla u_m(Y,s)\cdot \nabla_Y G_1(Y,s,X,t) dYds\Big)\partial_t\phi(X,t)dXdt
    \\
    &\qquad -\int_\Omega \Big(\int_\Omega(\EPS\nabla u_0-\EPS_m\nabla u_m)(Y,s)\cdot \nabla_Y G_1(Y,s,X,t) dYds\Big)\partial_t\phi(X,t)dXdt
    \\
    &=-\int_\Omega \EPS_m(Y,s)\nabla u_m(Y,s)\cdot \nabla_Y \Big(\int_\Omega G_1(Y,s,X,t) \partial_t\phi(X,t)dXdt\Big) dYds + R_2.
\end{align*}
We call the first term \(\tilde{I}\).

For \(J\) we have also by Fubini
\begin{align*}
    J&=-\int_\Omega \mathrm{div}_X(A_l^*(X,t)\nabla\phi(X,t)) \Big(\int_\Omega\EPS(Y,s)\nabla u_0(Y,s)\cdot \nabla_Y G_1(Y,s,X,t) dYds\Big) dXdt
    \\
    &=-\int_\Omega \mathrm{div}_X(A_l^*(X,t)\nabla\phi(X,t)) \Big(\int_\Omega\EPS_m(Y,s)\nabla u_m(Y,s)\cdot \nabla_Y G_1(Y,s,X,t) dYds\Big) dXdt
    \\
    &\qquad -\int_\Omega \mathrm{div}_X(A_l^*(X,t)\nabla\phi(X,t)) \Big(\int_\Omega(\EPS\nabla u_0-\EPS_m\nabla u_m)(Y,s)\cdot \nabla_Y G_1(Y,s,X,t) dYds\Big) dXdt
    \\
    &=-\int_\Omega \mathrm{div}_Y(\EPS^m\nabla u_m)(Y,s)\cdot\Big(\int_\Omega A_l^*(X,t)\nabla\phi(X,t)\cdot \nabla_X G_1(Y,s,X,t) dXdt\Big) dYds
    \\
    &\qquad -\int_\Omega \mathrm{div}_X(A_l^*(X,t)\nabla\phi(X,t)) \Big(\int_\Omega(\EPS\nabla u_0-\EPS_m\nabla u_m)(Y,s)\cdot \nabla_Y G_1(Y,s,X,t) dYds\Big) dXdt
    \\
    &=-\int_\Omega \mathrm{div}_Y(\EPS_m\nabla u_m)(Y,s)\cdot\Big(\int_\Omega A_1^*(X,t)\nabla\phi(X,t)\cdot \nabla_X G_1(Y,s,X,t) dXdt\Big) dYds
    \\
    &\qquad -\int_\Omega \mathrm{div}_Y(\EPS_m\nabla u_m)(Y,s)\cdot\Big(\int_\Omega (A_l^*-A_1^*)(X,t)\nabla\phi(X,t)\cdot \nabla_X G_1(Y,s,X,t) dXdt\Big) dYds
    \\
    &\qquad -\int_\Omega \mathrm{div}_X(A_l^*(X,t)\nabla\phi(X,t)) \Big(\int_\Omega(\EPS\nabla u_0-\EPS_m\nabla u_m)(Y,s)\cdot \nabla_Y G_1(Y,s,X,t) dYds\Big) dXdt
    \\
    &=:\tilde{J}+R_3+R_4.
\end{align*}

\begin{comment}
We continue with 
\begin{align*}
    \tilde{J}&=\int_\Omega \EPS(Y,s)\nabla u_0(Y,s)\cdot\nabla_Y\Big(\int_\Omega A_1^*(X,t)\nabla\phi(X,t)\cdot \nabla_X G(Y,s,X,t) dXdt\Big) dYds
    \\
    &\qquad+ \int_\Omega (\EPS_l^m\nabla u_m-\EPS\nabla u_0)(Y,s)\cdot\nabla_Y\Big(\int_\Omega A_1^*(X,t)\nabla\phi(X,t)\cdot \nabla_X G(Y,s,X,t) dXdt\Big) dYds
    \\
    &=\int_\Omega \EPS(Y,s)\nabla u_0(Y,s)\cdot\nabla_Y\Big(\int_\Omega A_1(X,t) \nabla_X G(Y,s,X,t)\cdot \nabla\phi(X,t) dXdt\Big) dYds+R_4.
\end{align*}
\end{comment}

Now we obtain in total
\begin{align*}
    \tilde{I}+\tilde{J}&=\int_\Omega \EPS_m(Y,s)\nabla u_m(Y,s)\cdot\nabla_Y\Big(\int_\Omega L_1G_1(Y,s,X,t) \phi(X,t) dXdt\Big) dYds
    \\
    &=\int_\Omega \EPS_m(Y,s)\nabla u_m(Y,s)\cdot\nabla_Y\Big(\int_\Omega L_1G_1(Y,s,X,t) (\phi-v)(X,t) dXdt\Big) dYds
    \\
    &\qquad + \int_\Omega \EPS_m(Y,s)\nabla u_m(Y,s)\cdot\nabla_Y\Big(\int_\Omega L_1G_1(Y,s,X,t)  v(X,t) dXdt\Big) dYds
    \\
    &=\int_\Omega \EPS(Y,s)\nabla u_0(Y,s)\cdot\nabla_Y(\phi-v)(Y,s) dYds
    \\
    &\qquad +\int_\Omega (\EPS_m\nabla u_m-\EPS\nabla u_0)(Y,s)\cdot\nabla_Y(\phi-v)(Y,s) dYds
    \\
    &\qquad + \int_\Omega \EPS_m(Y,s)\nabla u_m(Y,s)\cdot\nabla_Y\Big(\int_\Omega L_1G_1(Y,s,X,t)  v(X,t) dXdt\Big) dYds
    \\
    &=\int_\Omega \EPS(Y,s)\nabla u_0(Y,s)\cdot\nabla_Y(\phi-v)(Y,s) dYds + R_5+R_6.
\end{align*}

Plugging this back into \eqref{eq:Neumannequation1} yields by use of the PDE for \(u_0\)
\begin{align*}
    \eqref{eq:Neumannequation1}&= \int_\Omega -u_0\partial_t\phi + A_0\nabla u_0 \cdot\nabla\phi dXdt + \int_\Omega \EPS \nabla u_0 \cdot\nabla\phi dXdt\nonumber
    \\
    &\qquad +\int_\Omega -F\partial_t\phi + A_1\nabla F \cdot\nabla\phi dXdt
    \\
    &=\int_{\partial\Omega} g\phi dXdt +\int\EPS\nabla u_0\cdot v dXdt+\sum_{i=1}^6 R_i.
\end{align*}
Thus, if we can show that \(R_1\) to \(R_6\) can be made arbitrarily small, then we obtain \eqref{eq:Neumannsomething}.
\medskip

We begin with
\[R_1\lesssim \Vert \nabla F\Vert_{L^2(\Omega)}\Vert (A_1-A_l)\nabla\phi\Vert_{L^2(\Omega)}\leq \Vert F\Vert_E \Vert\nabla\phi\Vert_{L^\infty}\Vert A_1-A_l\Vert_{L^2(\Omega\cap\mathrm{supp}(\phi))}\to 0\]
when \(l\) tends to \(\infty\). 
For \(R_2\) we observe that if we set \(w\in \dot{E}(\Omega)\) to be the solution to \(L_1^*w=\partial_t\phi\) with zero Dirichlet boundary data, then
\begin{align*}
    R_2&=-\int_\Omega (\EPS\nabla u_0-\EPS_m\nabla u_m)(Y,s)\cdot \nabla_Y \Big(\int_\Omega G_1(Y,s,X,t) \partial_t\phi(X,t)dXdt\Big) dYds
    \\
    &=-\int_\Omega (\EPS\nabla u_0-\EPS_m\nabla u_m)(Y,s)\cdot \nabla_Y w(Y,s) dYds
    \\
    &\lesssim \Vert \EPS\nabla u_0-\EPS_m\nabla u_m\Vert_{L^2}\Vert \nabla w\Vert_{L^2}.
\end{align*}
Since \(\Vert \nabla w\Vert_{L^2}\lesssim C(\phi)<\infty\), the right hand side converges to zero if we take the limit in \(m\).
Analogously, if we set \(w_l\) to be the solution to \(L_1^*w_l=\mathrm{div}_X(A^*_l\nabla\phi)\) with zero Dirichlet boundary data, then
\begin{align*}
    R_4&\lesssim \Vert \EPS\nabla u_0-\EPS_m\nabla u_m\Vert_{L^2}\Vert \nabla w_l\Vert_{L^2},
\end{align*}
which converges to zero when we fix \(l\) and take the limit in \(m\). Furthermore, we obtain for \(\tilde{w}_m\in \dot{E}(\Omega)\) satisfying \(L_1\tilde{w}_m=\mathrm{div}_X(\EPS_m\nabla u_m)\) with zero Dirichlet boundary data
\begin{align*}
    R_3&=\int_\Omega (A_l^*-A_1^*)(Y,s)\nabla\phi(Y,s)\cdot \nabla_X \Big(\int_\Omega G(Y,s,X,t) \mathrm{div}(\EPS_m\nabla u_m)(Y,s)dYds\Big) dXdt
    \\
    &\lesssim \Vert A_l-A_1\Vert_{L^2}\Vert \nabla\phi\Vert_{L^\infty}\Vert \nabla\tilde{w}_m\Vert_{L^2}\lesssim \Vert A_l-A_1\Vert_{L^2}\Vert \nabla\phi\Vert_{L^\infty}\Vert \EPS_m \nabla u_m\Vert_{L^2}.
\end{align*}
The inequality \(\Vert \nabla\tilde{w}_m\Vert_{L^2}\lesssim\Vert \EPS_m \nabla u_m\Vert_{L^2}\) follows from Section \ref{subsection:EnergySol} by standard arguments using the hidden coercivity of the sesquilinear form.
Since \(\EPS_m\nabla u_m\to\EPS\nabla u_0\) when \(m\) approaches \(\infty\), taking the limit in \(m\) yields a bound of \(C\Vert A_l-A_1\Vert_{L^2}\Vert \EPS \nabla u_0\Vert_{L^2}\lesssim \Vert A_l-A_1\Vert_{L^2}\) which converges to zero if \(l\) approaches \(\infty\).
For \(R_5\) we have the immediate bound
\[R_5\lesssim \Vert \EPS\nabla u_0-\EPS_m\nabla u_m\Vert_{L^2}\Vert \nabla(\phi-v)\Vert_{L^2}\]
which converges to zero when taking the limit in \(m\).
Lastly, for \(R_6\) we note that
\begin{align*}
    R_6&=\int_\Omega \mathrm{div}_Y(\EPS_m\nabla u_m)(Y,s) 
    \\
    &\qquad \cdot\Big(\int_\Omega A_1(X,t) \nabla_X G(Y,s,X,t)\cdot \nabla v(X,t)-G(Y,s,X,t)\partial_tv(X,t) dXdt\Big) dYds
    \\
    &=\int_\Omega A_1(X,t) \nabla_X \Big(\int_\Omega G(Y,s,X,t)\mathrm{div}_Y(\EPS_m\nabla u_m)(Y,s)dYds\Big)\cdot \nabla v(X,t)
    \\
    &\qquad -\Big(\int_\Omega G(Y,s,X,t)\mathrm{div}_Y(\EPS_m\nabla u_m)(Y,s)dYds\Big)\partial_tv(X,t)  dYds.
\end{align*}
Letting \(\hat{w}\in \dot{E}(\Omega)\) be the solution to \(L_1\hat{w}=\mathrm{div}(\EPS_m\nabla u_m)\) with zero Dirichlet boundary data, allows us to rewrite this to
\begin{align*}
    &=\int_\Omega A_1(X,t) \nabla \hat{w}(X,t)\cdot \nabla v(X,t)-\hat{w}(X,t)\partial_tv(X,t) dXdt
    \\
    &=\int_\Omega \hat{w}(X,t)L_1^*v(X,t)dXdt.
\end{align*}
Since \(\hat{w}\) vanishes at the boundary and \(v\) satisfies the PDE \eqref{vDefinedNeumann}, this term vanishes. Overall we obtain that if we first take the limit in \(m\) and then in \(l\) all rest terms \(R_1-R_6\) vanish. This establishes \eqref{eq:Neumannsomething}.

\bibliographystyle{alpha}
\addcontentsline{toc}{section}{References}
\bibliography{references} 
\end{document}